\documentclass[a4paper,11pt, oneside]{amsart}
\usepackage[english]{babel}
\usepackage{amsthm}
\usepackage{amssymb}
\usepackage{amsfonts}
\usepackage{amsmath}
\usepackage{amsthm}
\usepackage[all]{xy}
\usepackage{geometry}
\usepackage{ae,aecompl}
\usepackage{eurosym}
\usepackage{verbatim}
\usepackage{mathrsfs}
\usepackage{caption}
\usepackage{url}
\usepackage{tikz}
\usepackage{pifont}
\usepackage{caption}
\usepackage{hyperref}
\usepackage{calligra}
\usepackage{array}
\usetikzlibrary{arrows.meta}

\newcounter{stepcounter}
\newtheoremstyle{smallcaps}
    {3pt}                    
    {3pt}                    
    {\itshape}                   
    {}                           
    {\sc}                   
    {.}                          
    {.5em}                       
    {}  
    
\newtheoremstyle{smallcapsdef}
    {3pt}                    
    {3pt}                    
    {}                   
    {}                           
    {\sc}                   
    {.}                          
    {.5em}                       
    {}  

\theoremstyle{plain}

\newtheorem{thm}{Theorem}[section]

\newtheorem{lem}[thm]{Lemma}
\newtheorem{prop}[thm]{Proposition}
\newtheorem{cor}[thm]{Corollary}

\theoremstyle{definition}

\newtheorem{eg}[thm]{Example}
\newtheorem{defn}[thm]{Definition}
\newtheorem{rem}[thm]{Remark}

\date{}

\newcommand\bit{\begin{itemize}}
\newcommand\eit{\end{itemize}}
\newcommand\bet{\begin{enumerate}}
\newcommand\eet{\end{enumerate}}
\newcommand\ed{\end{document}}

\DeclareFontFamily{U}{mathx}{\hyphenchar\font45}
\DeclareFontShape{U}{mathx}{m}{n}{
      <5> <6> <7> <8> <9> <10>
      <10.95> <12> <14.4> <17.28> <20.74> <24.88>
      mathx10
      }{}
\DeclareSymbolFont{mathx}{U}{mathx}{m}{n}
\DeclareFontSubstitution{U}{mathx}{m}{n}
\DeclareMathAccent{\widecheck}{0}{mathx}{"71}
\DeclareMathAccent{\wideparen}{0}{mathx}{"75}

\renewcommand{\d}{\delta}
\newcommand{\e}{\varepsilon}
\newcommand{\f}{\varphi}

\newcommand\Om{\Omega}

\newcommand\del{\partial}
\newcommand\adel{\ol{\partial}}

\newcommand\G{\Gamma}

\newcommand\F{{\mathcal F}}

\renewcommand{\O}{\mathcal{O}}

\newcommand\fh{{\mathfrak h}}

\newcommand\exd{\mathrm{d}}

\newcommand\unit{\mathrm{U}}

\newcommand\id{\mathrm{id}}

\newcommand\inv{^{-1}}

\def\qbinom#1#2{\ensuremath{\left[\kern-.3em\left[\genfrac{}{}{0pt}{}{#1}{#2}\right]\kern-.3em\right]_q}}

\newcommand\ol{\overline}

\def\clap#1{\hbox to 0pt{\hss#1\hss}}
\def\mathllap{\mathpalette\mathllapinternal}

\def\mathllapinternal#1#2{%
  \llap{$\mathsurround=0pt#1{#2}$}}

\newsavebox\qModFirst
\newsavebox\qModSecond
\newcommand{\Mod}{\mathrm{Mod}}
\newcommand{\modz}[2]{
  \sbox{\qModFirst}{$#1$}
  \sbox{\qModSecond}{$#2$}
  \ifdim\wd\qModFirst>\wd\qModSecond
    {}^{\phantom{#1}\mathllap{#1}}_{\phantom{#1}\mathllap{#2}}%
  \else
    {}^{\phantom{#2}\mathllap{#1}}_{\phantom{#2}\mathllap{#2}}%
  \fi\mathrm{mod}_0}
\newcommand{\qMod}[4]{{%
    \sbox{\qModFirst}{$#1$}
    \sbox{\qModSecond}{$#2$}
    \ifdim\wd\qModFirst>\wd\qModSecond
      {}^{\phantom{#1}\mathllap{#1}}_{\phantom{#1}\mathllap{#2}}%
    \else
      {}^{\phantom{#2}\mathllap{#1}}_{\phantom{#2}\mathllap{#2}}%
    \fi
    \Mod^{#3}_{#4}}}

\newcommand{\qmod}[4]{{%
      \sbox{\qModFirst}{$#1$}
      \sbox{\qModSecond}{$#2$}
      \ifdim\wd\qModFirst>\wd\qModSecond
        {}^{\phantom{#1}\mathllap{#1}}_{\phantom{#1}\mathllap{#2}}%
      \else
        {}^{\phantom{#2}\mathllap{#1}}_{\phantom{#2}\mathllap{#2}}%
      \fi
  \mathrm{mod}^{#3}_{#4}}}

\newcommand\FF{{\mathcal F}}

\newcommand{\OO}{\mathcal{O}}

\usepackage{tikz}
\usetikzlibrary{decorations.pathreplacing}

\usepackage{ytableau}

\usepackage[english]{babel}
 \usepackage{mathtools}
 \usetikzlibrary{shapes.geometric}
 
\title[Noncommutative Geometric Yang--Baxter Solutions]{Covariant Bimodule Connections and Noncommutative Geometric Yang--Baxter Solutions}

\thanks{R\'OB is supported by the GA\v{C}R/NCN grant \emph{Quantum Geometric Representation Theory and Noncommutative Fibrations} 24-11728K. All four authors acknowledge support from COST Action 21109 CaLISTA, supported by COST (European Cooperation in Science and Technology) and HORIZON-MSCA-2022-SE-01-01 CaLIGOLA. AC is supported by ``Implementation of cutting-edge research and its application as part of the Scientific Center of Excellence for Quantum and Complex Systems, and Representations of Lie Algebras'' Grant No. PK.1.1.10.0004, co-financed by the European Union through the European Regional Development Fund: Competitiveness and Cohesion Programme 2021--2027.
AC acknowledges support from  MSCA-DN CaLiForNIA -- 101119552 and from an INDAM project 2024. AC is grateful for the hospitality offered by Charles University and by the Institute of Mathematics of the Czech Academy of Sciences.}

\author[A. Carotenuto]{Alessandro Carotenuto}
\address{Department of Mathematics, Faculty of Science, Bijenička Cesta 30, 10000, Zagreb, Croatia} 
\email{acaroten91@gmail.com}

\author[A. Del Donno]{Antonio Del Donno}
\address{Mathematical Institute of Charles University, Sokolovsk\'a 83, Prague, Czech Republic} 
\email{antonio.deldonno2@gmail.com}

\author[R. \'O Buachalla]{R\'eamonn \'O Buachalla}
\address{Mathematical Institute of Charles University, Sokolovsk\'a 83, Prague, Czech Republic} 
\email{reamonnobuachalla@gmail.com}

\author[J. Razzaq]{Junaid Razzaq}
\address{Fabit, Universita di Bologna.
Via Piero Gobetti 87, 40129 Bologna, Italy} 
\email{ junaid.razzaq2@unibo.it}

\begin{document}
\raggedbottom

\maketitle

\begin{abstract}
We introduce a new class of Yang--Baxter solutions,  which we call \emph{noncommutative geometric solutions}, constructed from quantum group noncommutative geometries. These solutions are built from a quantum homogeneous space, endowed with a covariant differential graded algebra and a noncommutative Koszul connection $\nabla$. If $\nabla$ is a bimodule connection, that is, if its generalised flip map $\sigma$ is well defined, then a simple set of sufficient conditions for $\sigma$ to be a Yang--Baxter solution is given. These results are applied to the cominuscule quantum flag manifolds $\OO_q(G/L_S)$ endowed with their Heckenberger--Kolb calculi. The noncommutative complex geometry of the calculi decomposes their  Yang--Baxter operator $\sigma$ into three parts: holomorphic, anti-holomorphic, and cross-type. For the non-exceptional examples, the holomorphic and anti-holomorphic solutions recover, up to scalar multiple, the usual quasitriangular braidings of the associated Drinfeld--Jimbo $U_q(\mathfrak{l}_S^s)$-modules. However, we get a new extension of these braidings to the direct sum, which remarkably restricts to an involutive operator on the cross terms. As an application, we produce a simple uniform presentation of the degree two relations of the Heckenberger--Kolb calculi as a \emph{quantum exterior algebra} associated to the generalised flip map $\sigma$. 
\end{abstract}

\tableofcontents


\section{Introduction}


For any vector space $V$, the Yang--Baxter equation is a deceptively simple condition for a linear map 
$$
\sigma: V \otimes V \to V \otimes V.
$$
Explicitly, it asks for the following equality of linear operators on $V \otimes V \otimes V$
$$
(\sigma \otimes \id) \circ (\id \otimes \sigma)  \circ (\sigma \otimes \id) = (\id \otimes \sigma) \circ  (\sigma \otimes \id) \circ (\id \otimes \sigma).
$$
It was introduced independently by Yang, in the setting of quantum many-body problems \cite{Yang}, and by Baxter, in the setting of statistical mechanics \cite{Baxter}. 
Since its introduction, the Yang--Baxter equation has found applications in a dizzying number of areas. These span pure mathematics, physics, and more recently quantum information theory. In pure mathematics, for example, it plays a foundational role in braided monoidal categories and quantum groups \cite{Etingof,Kassel,Maj.Purple}, it is deeply connected to knot theory and low-dimensional topology \cite{Turaev}, it appears in  von Neumann algebras through the theory of factors \cite{Jones}, and in many other areas too numerous to mention.  
Despite its simple form, the equation is a highly overdetermined system of polynomial equations, so finding explicit solutions is notoriously difficult. 
In fact, the first examples of what are now called Drinfeld--Jimbo quantum groups were discovered by the Leningrad school in the course of developing the quantum inverse-scattering method and studying the Yang--Baxter equation \cite{Fadeev.Survey}.


The noncommutative geometry of quantum groups has a rich history of interaction with the Yang--Baxter equation. Two examples close to the spirit of this paper are the Nichols algebra construction of higher forms for bicovariant differential calculi due to Woronowicz \cite{Worst.Mistake}, and the construction of bicovariant calculi for coquasitriangular Hopf algebras, see \cite[\textsection 14.5]{KSLeabh} and \cite[\textsection 2.5]{BM.Leabh} for details. In this paper we open a new perspective and give a concrete new method for constructing Yang--Baxter solutions from quantum group noncommutative geometry. This introduces a new class of solutions that we call \emph{noncommutative geometric Yang--Baxter solutions}.


Our approach to constructing Yang--Baxter solutions centres around the notion of a bimodule connection for a differential calculus, which we now briefly recall: A \emph{differential calculus} is a differential graded algebra $\Omega^{\bullet}$ generated in degree zero. In the philosophy of noncommutative geometry it is viewed as a noncommutative generalisation of the de Rham complex of a manifold. A \emph{left connection} for a left $B$-module $\FF$ is a linear map  $\nabla:\FF \to \Omega^1(B) \otimes_B \FF$ satisfying the standard left Leibniz rule. It is viewed as a noncommutative generalisation of a Koszul connection. If $\FF$ is a $B$-bimodule, then a bimodule connection for $\F$, as first considered in the 1990s in \cite{Dubois.Violette.Masson, Dubois.Violette.Michor, Mourad}, is a left connection $\nabla$ with an additional compatibility between the connection and the right module structure. Explicitly, the generalised flip map
\begin{align*}
\sigma: \FF \otimes_B \Omega^1(B) \to \Omega^1(B) \otimes_B \FF, & & f \otimes \exd b  \mapsto \nabla(fb) - \nabla(f)b,
\end{align*}
is required to give a well-defined bimodule map. We call $\sigma$ the \emph{generalised flip  map}, or \emph{bimodule map}, associated to $\nabla$. Bimodule connections have been the focus of renewed interest in recent years, appearing in a variety of different areas: quantum Riemannian geometry \cite[\textsection 3.4]{BM.Leabh}, quantum group noncommutative geometry \cite{JBBGAKROB, Bimodule.Lines.Flag, MMAdv}, monoidal categories \cite[\textsection 3.4.2]{BM.Leabh}, Hopf algebroids \cite{Ghobadi2022Isotropy}, Drinfeld twist deformations \cite{AschieriSchenkel}, and  spectral triples \cite{Mes.Ren.Podles,Mesland.Rennie.Adv}.  In the commutative case, $\sigma$ reduces to the flip map. Thus, for the special case of $\FF = \Omega^1(B)$, it is natural to ask when $\sigma$ satisfies the Yang--Baxter equation.


Known examples for which the generalised flip map of a bimodule connection satisfies the Yang--Baxter equation are relatively rare, but three important examples come from the theory of Drinfeld--Jimbo quantum groups $U_q(\mathfrak{g})$: Woronowicz's $3D$-calculus for $\OO_q(\mathrm{SU}_2)$ \cite[Example 3.77]{BM.Leabh}, the quantum plane \cite{DuboisVioletteMadoreMassonMourad1995}, and the prototypical example of a quantum homogeneous space: the standard Podle\'s sphere \cite[Example 5.79]{BM.Leabh}. In this paper, motivated by the Podle\'s example, we undertake a systematic study of bimodule connections for covariant differential calculi $\Omega^1(B)$ over a quantum homogeneous $A$-space $B$, for a Hopf $*$-algebra $A$. We do so under the assumption that our differential calculi satisfy the condition $\Omega^1(B)B^+ = B^+\Omega^1(B)$, where $B^+$ is the augmentation ideal of $B$. This produces many results, of which we highlight the following: For a left relative Hopf module $\mathcal{F}$ endowed with its canonical right $B$-module structure, the generalised flip map associated to a bimodule connection $\nabla$ is determined by the simple identity 
\begin{align} \label{eqn:intro.sigma.formula} 
\sigma(f \otimes \exd b) = \exd\big(f_{(-2)}bS(f_{(-1)})\big) \otimes f_{(0)}, & & \textrm{ for all } f \in \mathcal{F}, \, b \in B.
\end{align}
In particular, we see that $\sigma$ is independent of the choice of $\nabla$.

Covariant first-order calculi over $B$ are dual to and in bijective correspondence with \emph{quantum tangent spaces}. Quantum tangent spaces $T$ are noncommutative generalisations of the tangent space of a homogeneous space at the coset of the identity. By definition, their unitisation $T \oplus \mathbb{C}1$ is a coideal of the finitary dual $B^{\circ}$, which is moreover closed under the left $L$-module structure of $B^{\circ}$, where $L$ is the quantum generalisation of the universal enveloping algebra of the isotropy Lie subalgebra of a homogeneous space. 
Dualising the identity \eqref{eqn:intro.sigma.formula}  to the tangent space setting offers a crucial new perspective. It allows us to prove that if $T$ is generated by the restriction to $B$ of twisted primitive elements of $A^{\circ}$, then $\sigma$ is \emph{always} a Yang--Baxter solution. This is the key result of the paper, giving the \emph{noncommutative geometric Yang--Baxter solutions} of the title.


In order to convince the reader of the importance of this class of solutions, we must of course produce a significant family of motivating examples. To this end, we look to the standard Podle\'s sphere $\OO_q(S^2)$. The differential calculus of $\OO_q(S^2)$, known as the \emph{Podle\'s calculus}, is the simplest example of a Heckenberger--Kolb calculus $\Omega^{\bullet}_q(G/L_S)$ for a cominuscule quantum flag manifold $\OO_q(G/L_S)$. This is a family of Drinfeld--Jimbo quantum homogeneous spaces, labelled by the cominuscule roots $\alpha_x$ of a complex simple Lie algebra $\mathfrak{g}$. (Here we have denoted $S := \Pi\backslash\{\alpha_x\}$, with $\Pi$ being the set of simple roots of $\mathfrak{g}$.) For the simplest case of $\mathfrak{sl}_2$, with $S = \varnothing$, we recover  $\OO_q(S^2)$. These differential calculi form what is arguably the most important class of quantum group noncommutative geometries. They have  been the subject of intense investigation over the last twenty years, and we are now approaching a reasonable understanding of their structure. One of the main results of this paper is to apply our general framework to the Heckenberger--Kolb calculi. Each of these calculi possesses a unique covariant connection~\cite[Theorem 4.5]{HVBQFM}. Indeed, this connection has recently been identified as a \emph{noncommutative Levi-Civita connection} \cite[Theorem 6.14]{JBBGAKROB} (in the sense of Beggs and Majid \cite[\textsection 8.1]{BM.Leabh}) for the canonical metric of the calculi. (See also earlier work of Matassa for the special case of quantum projective space \cite{MMAdv}.) In particular, all of these connections are bimodule connections. It essentially follows from the construction of Heckenberger and Kolb that their tangent spaces are generated by twisted primitive elements, and so their generalised flip maps $\sigma$ are Yang--Baxter solutions.


The classical flag manifolds are K\"ahler, and in particular complex manifolds. Much of this complex geometry survives to the quantum setting. For example, the differential calculi $\Omega^{\bullet}_q(G/L_S)$ have a unique covariant double complex structure. This allows us to decompose the  Yang--Baxter operator $\sigma$ into three parts: a holomorphic term, an anti-holomorphic term, and a cross term. Each of these maps is diagonalisable, and our explicit formulae for $\sigma$ allow us to calculate their spectra. Using this, we show that the holomorphic and anti-holomorphic solutions recover (up to scalar multiple) the usual quasitriangular braidings of the associated Drinfeld--Jimbo $U_q(\mathfrak{l}_S^s)$-modules. However, we get a new extension of these braidings to the direct sum, which remarkably restricts to an involutive operator on the cross terms. This is quite surprising behaviour for non-Kac-type quantum groups, where naturally occurring involutive maps are rare. 


A very important application of our analysis of $\sigma$ is to give a conceptual description of the higher Heckenberger–Kolb forms. We recall that in noncommutative geometry one usually starts with a first-order differential calculus (an algebraic structure modelled on the space of $1$-forms of a smooth manifold) and then extends it to a higher-order differential calculus. In the classical setting, this extension is simply the exterior algebra of the $1$-forms. However, in the noncommutative setting things are far more subtle. There does, however, exist a universal, or maximal prolongation, an extension such that every other extension arises as a quotient of it. Explicitly describing this maximal prolongation (for example, determining its space of highest degree non-zero forms) is a very challenging problem. (In many ways, it is analogous to determining when the Nichols algebra of a braided vector space is finite-dimensional.) Moreover, even when one does determine the relations explicitly, it can be difficult to make conceptual sense of the resulting terms.



This is the case for the Heckenberger--Kolb calculi. In \cite{HKdR} a technically intricate $R$-matrix description of the maximal prolongation of the first-order forms was derived, and it was shown that the calculi have classical dimension. 
Here we use the generalised flip map $\sigma$ to give a new concise conceptual description of the  maximal prolongation. In particular, we show that the space of generating relations 
$$
N^{(2)} \subseteq \Omega^1_q(G/L_S) \otimes_{\OO_q(G/L_S)} \Omega^1_q(G/L_S),
$$ 
that is, the space of $q$-deformed degree-$2$ symmetric forms, is given by
\begin{align*}
N^{(2)} = \mathrm{im}(\sigma + \id).
\end{align*}
In other words, we show that $\Omega^{\bullet}_q(G/L_S)$ is the \emph{quantum exterior algebra} of the generalised flip map $\sigma$. We then conclude from  \eqref{eqn:intro.sigma.formula} that $N^{(2)}$ is generated as an $\OO_q(G/L_S)$-bimodule by the elements
\begin{align*}
\Big\{ \exd b \otimes \exd c + \exd \big(b_{(1)}cS(b_{(2)})\big) \otimes \exd b_{(3)} \,|\, \, b,c \in \OO_q(G/L_S)\Big\}.
\end{align*}


\subsubsection*{\bf Summary of the Paper}

The paper is organised as follows: In \textsection \ref{section:preliminaries} we recall some necessary preliminaries about differential calculi, maximal prolongations of first-order calculi, connections, bimodule connections, and torsion.

In \textsection \ref{section:torsionfreeextension} we establish general results relating bimodule connections, their generalised flip maps, and the degree two relations of the maximal prolongation. We then study invertible generalised flip maps and introduce the notion of $*$-invertibility. 

In \textsection \ref{section:QHS} we specialise to the case of a covariant differential calculus $\Omega^1(B)$ over a quantum homogeneous space $B \subseteq A$. We show that for a relative Hopf module $\mathcal{F}$, with canonical right module structure, the generalised flip map of a bimodule connection $\nabla$ satisfies the identity \eqref{eqn:intro.sigma.formula}. Moreover, we show that in the $*$-calculus case, the generalised flip map $\sigma$ is always $*$-invertible.

In \textsection \ref{section:YangBaxter} we pass to the dual quantum tangent space picture and show that tangent spaces generated by twisted primitive elements give rise to generalised flip maps $\sigma$ satisfying the Yang--Baxter equation. This constitutes one of the principal results of the paper.

In \textsection \ref{section:ALC}, motivated by the delicate relationship between the Yang--Baxter solution $\sigma$ and noncommutative complex geometry of the quantum flags, we undertake a systematic treatment of noncommutative complex structures and bimodule connections, producing a number of results of independent interest. For example, we give explicit criteria for when a covariant first-order almost complex structure can be extended to an almost complex structure. Surprisingly, we also show that factorisability of the almost complex structure is equivalent to $\sigma$ squaring to the identity on tensor products of forms of mixed holomorphic and anti-holomorphic type.

In \textsection \ref{section:HK} we apply our general framework to the Heckenberger--Kolb calculi over the cominuscule quantum flag manifolds. We show that their bimodule maps satisfy the Yang--Baxter equation and that the calculi are $\sigma$-exterior with respect to their unique covariant connection. Moreover, we recover their almost complex structures and conclude that $\sigma$ squares to the identity on mixed forms. 

In \textsection \ref{section:spectrum}  we calculate the spectrum of $\sigma$ for the classical series of cominuscule quantum flag manifolds and (for the non-exceptional examples) identify its restrictions to purely holomorphic and anti-holomorphic forms as rescalings of the quasitriangular braiding. We also compare $\sigma$ with Drinfeld’s normalised braiding, and with the braiding underlying the Nichols algebra presentation of the anti-holomorphic exterior algebra of the quantum Grassmannians given in \cite{Nichols.MMF}. 

In \textsection \ref{section:examples} we calculate the generalised flip maps and the degree two relations explicitly for several low rank cominuscule quantum flag manifolds: the Podle\'s sphere, the quantum projective plane, the quantum Lagrangian quadric, and the $(4,2)$-quantum Grassmannian. 

Finally, in Appendix \ref{section:Appendix} we present some technical tables detailing representation theoretic facts used throughout the second half of the paper.

\subsubsection*{Acknowledgements:} We would like to thank Edwin Beggs, Jyotishman Bhowmick, Bappa Ghosh, Jan Grabowski, Andrey Krutov, 
Shahn Majid, Bram Mesland, and Thomas Weber for many useful discussions. 

\subsubsection*{AI statement:} GPT-5.6 Sol and Claude Opus 5.5 were used to assist with proofreading. No LLMs were used for the mathematical arguments, which were solely produced by
the authors.


\section{Preliminaries} \label{section:preliminaries}

This section provides basic material on differential calculi, connections, bimodule connections, and the torsion of a connection needed throughout the paper. Throughout the paper, all algebras are over $\mathbb{C}$ and are assumed to be unital. Moreover, all unadorned tensor products are over $\mathbb{C}$.

\subsection{Differential Calculi}

A {\em differential calculus}, or a \emph{dc}, is a differential graded algebra 
$
\big(\Om^\bullet \simeq \bigoplus_{k \in \mathbb{Z}_{\geq 0}} \Om^k, \exd\big)
$  
that is generated as an algebra by the elements $a, \exd b$, for $a,b \in \Om^0$. When no confusion arises, we denote the dc by $\Omega^{\bullet}$, omitting the exterior derivative $\exd$. 
For a given algebra $B$, a differential calculus {\em over} $B$ is a differential calculus such that $B = \Om^0$.

A \emph{$*$-differential calculus}, or a \emph{$*$-dc}, over a $*$-algebra $B$ is a differential calculus over $B$ such that the $*$-map of $B$ extends to a (necessarily unique) conjugate-linear involution $*: \Omega^{\bullet}(B) \to \Omega^{\bullet}(B)$ satisfying the identity $\exd(\omega^*) = (\exd \omega)^*$, for all $\omega \in \Omega^{\bullet}$, and for which
\begin{align*}
(\omega \wedge \nu)^* = (-1)^{kl} \nu^* \wedge \omega^*, & & \textrm{ for all } \omega \in \Omega^k, \, \nu \in \Omega^l.
\end{align*}

A {\em first-order differential calculus} (fodc) over an algebra $B$ is a pair $(\Om^1(B),\exd)$, where $\Omega^1(B)$ is a $B$-bimodule and $\exd: B \to \Omega^1(B)$ is a derivation such that $\Om^1$ is generated as a left (or equivalently right) $B$-module by those elements of the form~$\exd b$, for~$b \in B$. A {\em first-order $*$-differential calculus} ($*$-fodc) over a $*$-algebra $B$ is defined in analogy with the $*$-dc case.

We say that a dc $(\G^\bullet,\exd_\G)$ over $B$ {\em prolongs} a fodc $(\Om^1,\exd_{\Om})$ over $B$ if there exists a bimodule isomorphism $\f:\Om^1 \to \G^1$ such that $\exd_\G = \f \, \circ \, \exd_{\Om}$. It can be shown that any fodc admits a prolongation $\Om^\bullet$ that is \emph{maximal}, in the sense that there exists a unique differential graded algebra morphism from $\Om^\bullet$ onto any other prolongation of $\Om^1$. We call this extension the {\em maximal prolongation} of $\Om^1$. 

By construction, the maximal prolongation is a quotient of the tensor algebra $\mathcal{T}_B(\Omega^1(B))$ by an ideal generated in degree $2$. We denote this sub-bimodule of degree-two relations by $N^{(2)}$. Explicitly, this means that
\begin{align*}
\Omega^2(B) \simeq \Omega^1(B) \otimes_B \Omega^1(B)/N^{(2)}, & & \Omega^{\bullet}(B) \simeq \mathcal{T}_B(\Omega^1(B))/\langle N^{(2)}\rangle.
\end{align*}
Indeed, it is not difficult to show that 
$$
N^{(2)} = \left\{\sum_i \exd b_i \otimes \exd c_i \,|\, \textrm{ for } \sum_i b_i \exd c_i = 0\right\}.
$$
For a more detailed presentation of the maximal prolongation, see \cite[\textsection 1]{BM.Leabh} or \cite[\textsection 12.2.2]{KSLeabh}.

\subsection{Bimodule Connections and Torsion} \label{subsection:bimod.conn.torsion}

In this subsection we briefly recall some basic notions about connections. Let $\Omega^{1}(B)$ be a fodc over an algebra $B$, and let $\mathcal{F}$ be a left $B$-module. Then a \emph{left connection}, or simply a \emph{connection}, for $\mathcal{F}$ is a $\mathbb{C}$-linear map
$$
\nabla: \mathcal{F} \to \Omega^1(B) \otimes_B \mathcal{F}
$$
satisfying the identity 
\begin{align*}
\nabla(bf) = \exd b \otimes f + b\nabla f, & & \textrm{ for all } b \in B, \, f \in \mathcal{F}.
\end{align*}

Assuming that $\mathcal{F}$ is moreover a $B$-bimodule, a \emph{left bimodule connection} for $\F$ is a pair $(\nabla,\sigma)$, where $\nabla$ is a left connection and $\sigma$ is a generalised flip map $\sigma: \F \otimes_B \Omega^1(B) \rightarrow \Omega^1(B) \otimes_B \F $, satisfying the identity
\begin{align}\label{eqn:BiC}
\nabla (fb) = \nabla(f)b + \sigma (f\otimes db).
\end{align}
Note that in the commutative case $\sigma(f \otimes \exd b) = \exd b \otimes f$, that is to say, $\sigma$ reduces to the flip map. Moreover, since $ \sigma (f\otimes db)= \nabla (fb)-\nabla(f)b$, we see that $\sigma$ is uniquely determined by $\nabla$, so in particular being a bimodule connection is a \emph{property} of a connection.

Let us now restrict to the special case of $\FF = \Omega^1(B)$. For any such connection 
$$
\nabla: \Omega^1(B)\rightarrow \Omega^1(B)\otimes_B \Omega^1(B)
$$
the sub-bimodule $N^{(2)}$ is contained in $B\nabla(\exd B)$. Indeed, if $\sum_i b_i \exd c_i = 0$ in $\Omega^1(B)$, then 
$$
0 = \nabla\!\left(\sum_i b_i \exd c_i\right) = \sum_i \exd b_i \otimes \exd c_i + \sum_i b_i \nabla(\exd c_i).
$$
Thus we see that 
$$
\sum_i \exd b_i \otimes \exd c_i = - \sum_i b_i\nabla(\exd c_i) \in B \nabla(\exd B).
$$

The \emph{torsion} of $\nabla$ is the left $B$-module map 
\begin{align*}
T_{\nabla}:= \wedge \circ \nabla - \exd:\Omega^1(B) \to \Omega^2(B). 
\end{align*}
If $T_{\nabla} = 0$, then we say that $\nabla$ is \emph{torsion-free}. 

If $\nabla$ is a bimodule connection, then 
\begin{align*}
T_{\nabla}(\omega b) - T_{\nabla}(\omega)b = 
  \,  \wedge \big(\sigma(\omega \otimes \exd b) + \omega \otimes \exd b \big),
\end{align*}
see \cite[Lemma 3.72]{BM.Leabh} for details. Thus $\sigma(\omega \otimes \exd b) + \omega \otimes \exd b$ is in the kernel of $\wedge$, which is to say in $N^{(2)}$, if and only if $T_{\nabla}$ is a right $B$-module map. This immediately implies that
\begin{align}\label{eqn:wedge.sigma}
\wedge \circ \sigma = - \wedge.
\end{align} 
Even when these equivalent conditions hold, the inclusion of $\mathrm{im}(\sigma + \mathrm{id})$ in $N^{(2)}$ need not be an equality (see Example \ref{eg:graph}) which motivates Definition \ref{defn:sigma.ext} below.

If $\nabla$ is torsion-free, then since 
$$
\wedge(\nabla(\exd b)) = \exd^2 b = 0,
$$
we have that $\nabla(\exd b)$ is contained in $N^{(2)}$, and hence that
\begin{align} \label{eqn:torsion.free.N2}
N^{(2)} = B \nabla(\exd B) = B \nabla(\exd B)B.
\end{align}
Indeed, we see that this equality is equivalent to $\nabla$ being torsion-free. We finish with a short summarising table:



\begin{center}
  \renewcommand{\arraystretch}{1.6}
  \begin{tabular}{|r c l|}
    \hline
    $T_{\nabla}$ a bimodule map & $\Longleftrightarrow$ & $\mathrm{im}(\sigma + \id) \subseteq N^{(2)}$ \\
    \hline
    $T_{\nabla} = 0$ & $\Longleftrightarrow$ & $N^{(2)} = B\nabla(\exd B)$ \\
    \hline
  \end{tabular}
\end{center}


\section{Some General Results about Bimodule Connections} \label{section:torsionfreeextension}

In this section we establish some general results about bimodule connections that will be used throughout the paper. In particular, we investigate the special case of connections with invertible generalised flip maps, and then specialise to the case where this inverse can be expressed in terms of a $*$-map for the calculus.

Throughout this section, $\Omega^{1}(B)$ will be a fodc over an algebra $B$, and $N^{(2)}$ will denote the degree two relations of its maximal prolongation $\Omega^{\bullet}(B)$. Moreover, $\nabla$ will be a bimodule connection for $\Omega^1(B)$, with generalised flip map $\sigma$.

\subsection{Some Preliminary Results}

In this subsection we make some elementary observations about bimodule connections. Of particular importance is Corollary \ref{cor:the.projection}, which will be used in \textsection \ref{section:HK} to prove that the Heckenberger--Kolb calculi are $\sigma$-exterior (see Definition \ref{defn:sigma.ext}).

We begin with a simple identity that has three important consequences. Despite its simplicity, we find it instructive to present the identity in the form of a proposition.

\begin{prop} \label{prop:useful.id}
For any two elements $b,c \in B$, it holds that 
\begin{align*}
   \nabla(\exd(bc)) =  \nabla(\exd b)c + b\nabla(\exd c) + (\sigma + \id)(\exd b \otimes \exd c). 
\end{align*}
\end{prop}
\begin{proof}
This follows immediately from the calculation 
\begin{align*}
   \nabla(\exd(bc))& =  \,\nabla((\exd b)c) + \nabla(b\exd c)\\ 
    &= \, \nabla(\exd b)c + \sigma(\exd b \otimes \exd c) + \exd b \otimes \exd c + b\nabla(\exd c),
\end{align*}
where we have simply used the defining properties of a bimodule connection.
\end{proof}

\begin{cor} \label{cor:im.sigma.subseteq.BNabla(dB)B.}
It holds that $
\mathrm{im}(\sigma + \mathrm{id}) \subseteq B\nabla(\exd B)B$.
\end{cor}
\begin{proof}
Proposition \ref{prop:useful.id} says that 
\begin{align*}
(\sigma + \id)(\exd b \otimes \exd c)  = \nabla(\exd(bc)) - \nabla(\exd b)c - b\nabla(\exd c) \in B\nabla(\exd B)B.
\end{align*}
Since tensors of this form generate $\Omega^1(B)\otimes_B\Omega^1(B)$ as a \(B\)-bimodule, the claimed inclusion follows.
\end{proof}

\begin{cor} \label{cor:the.projection}
A derivation is given by the map 
$$
\mathfrak{d} := \mathrm{proj} \circ \nabla \circ \exd: B \to \Omega^1(B) \otimes_B \Omega^1(B)/\mathrm{im}(\sigma + \mathrm{id}),
$$
where the map 
$\mathrm{proj}$ is the canonical projection from $\Omega^1(B) \otimes_B \Omega^1(B)$ to the quotient space $\Omega^1(B) \otimes_B \Omega^1(B)/\mathrm{im}(\sigma + \mathrm{id})$.
\end{cor}
\begin{proof}
Applying the projection $\mathrm{proj}$ to the identity in Proposition \ref{prop:useful.id} we get the identity
$
\mathfrak{d}(bc) = \mathfrak{d}(b) c + b \mathfrak{d}(c), 
$
proving that $\mathfrak{d}$ is a derivation. 
\end{proof}


\begin{cor}
Let $X$ be a set of generators of $B$ as a unital algebra, and assume that $T_{\nabla} = 0$. Then it holds that $N^{(2)}$ is spanned as a $B$-bimodule by $\mathrm{im}(\sigma + \id)$ and the set of elements
$$
\Big \{\nabla(\exd x) \,|\,  x \in X \Big\}.
$$
\end{cor}
\begin{proof}
Let us denote by $S$ the $B$-bimodule span of $\mathrm{im}(\sigma + \id)$ and $\{\nabla(\exd x) \,|\, x \in X\}$. Assume that for any product of $k$ generators  $b = x_1 \cdots x_k$, we have that $\nabla(\exd b)$ is contained in $S$. Then for any $x_{k+1} \in X$,
\begin{align}
    \nabla(\exd(bx_{k+1})) = \nabla(\exd b)x_{k+1} + b\nabla(\exd x_{k+1}) + (\sigma + \id)(\exd b \otimes \exd x_{k+1}) \in S.
\end{align}
An inductive argument then implies that for any product $c$ of generators, we have that $\nabla(\exd c)$, and hence all of $B\nabla(\exd B)$, is contained in $S$. Thus, since $\nabla$ is torsion-free, meaning that $N^{(2)}$ is equal to $B\nabla(\exd B)$, we must have that $N^{(2)}$ is contained in $S$. In the other direction, the inclusion $S \subseteq N^{(2)}$ follows from the discussions of \textsection \ref{subsection:bimod.conn.torsion}.
\end{proof}

We finish with one last simple observation. This gives a type of relation that disappears in the classical setting since the flip map is an involution. 

\begin{prop} \label{prop:sigma.squared}
If the torsion $T_{\nabla}$ is a $B$-bimodule map, then for any element $\omega \otimes \nu$ in $\Omega^1(B) \otimes_B \Omega^1(B)$, it holds that 
$
\omega \otimes \nu  - \sigma^2(\omega \otimes \nu ) \in N^{(2)}.
$
\end{prop}
\begin{proof}
Note that the element
\begin{align*}
(\sigma + \mathrm{id})(\omega \otimes \nu - \sigma(\omega \otimes \nu)) = & \,\omega \otimes \nu - \sigma(\omega \otimes \nu) + \sigma(\omega \otimes \nu) - \sigma^2(\omega \otimes \nu) \\
= &\, \omega \otimes \nu - \sigma^2(\omega \otimes \nu)
\end{align*}
is contained in the image of $\sigma + \mathrm{id}$. Hence $\omega \otimes \nu - \sigma^2(\omega \otimes \nu)$ is a relation as claimed. 
\end{proof}


\subsection{The Case of an Invertible Generalised Flip Map}

In this subsection we examine the case where the bimodule map $\sigma$ is invertible. In particular, in Proposition \ref{prop:sigma.invariance} we show that in the torsion-free case the degree two relations $N^{(2)}$ are closed under the action of $\sigma$.

If the generalised flip map $\sigma$ is invertible, then 
\begin{align} \label{eqn:antiTor}
    \nabla^R := \sigma^{-1} \circ \nabla   
\end{align}
is a right bimodule connection and its generalised flip map is given by $\sigma^{-1}$, see \cite[Lemma 3.70]{BM.Leabh} for details. Moreover, if $T_{\nabla}$ is a $B$-bimodule map, then precomposing the identity \eqref{eqn:wedge.sigma} with $\sigma^{-1}$ gives the identity
\begin{align} \label{eqn:wedge.sigma.inv}
\wedge \circ \sigma^{-1} = - \wedge.
\end{align}

\begin{defn}
The \emph{anti-torsion operator} of $\nabla$ is the operator
$
\overline{T}_{\nabla} := \wedge \circ \nabla^R + \exd.
$
When $\overline{T}_{\nabla} = 0$, we say that $\nabla$ is \emph{anti-torsion-free}.
\end{defn}    
Analogously to the torsion operator, $\overline{T}_{\nabla}$ is automatically a right $B$-module map. The following lemma addresses the question of when it is a $B$-bimodule map.

\begin{lem} \label{lem:sigma.on.N2}
For a bimodule connection $\nabla$ with an invertible generalised flip map $\sigma$, the following are equivalent:
\begin{enumerate}
\item $T_{\nabla}$ is a $B$-bimodule map,
\item $\mathrm{im}(\sigma + \id) \subseteq N^{(2)}$,
\item $\overline{T}_{\nabla}$ is a $B$-bimodule map,
\item $\mathrm{im}(\sigma^{-1} + \id) \subseteq N^{(2)}$,
\item $T_{\nabla} = - \overline{T}_{\nabla}$.
\end{enumerate}
Thus, if these conditions hold, then $\nabla$ is torsion-free if and only if it is anti-torsion-free.
\end{lem}
\begin{proof}
We have already observed that 1. and 2. are equivalent. An analogous argument establishes that 3. and 4. are equivalent. Now if 2. holds, then it follows from \eqref{eqn:wedge.sigma.inv} that 
\begin{align*}
\overline{T}_{\nabla}(\omega) = \wedge \circ \nabla^{R}(\omega) + \exd \omega = \wedge \circ \sigma^{-1} \circ \nabla(\omega) + \exd \omega = - \wedge \circ \nabla(\omega) + \exd \omega = - T_{\nabla}(\omega).
\end{align*}
Thus we see that 2. implies 5. In the other direction, since $T_{\nabla}$ is a left $B$-module map, and $\overline{T}_{\nabla}$ is a right $B$-module map, the equality $T_{\nabla} = - \overline{T}_{\nabla}$ implies that $T_{\nabla}$ is a $B$-bimodule map. Thus 5. implies 1. The proof that 5. is equivalent to 3. and 4. is analogous. 
\end{proof}

If $\nabla$ is torsion-free, Lemma \ref{lem:sigma.on.N2} implies that it is anti-torsion-free. The right-connection analogue of \eqref{eqn:torsion.free.N2} therefore gives
\begin{align} \label{eqn:right.conn}
\nabla^R(\exd B)B   =  N^{(2)}.
\end{align}
The following proposition uses this fact to show that $N^{(2)}$ is closed under the action of $\sigma$ and $\sigma^{-1}$.

\begin{prop} \label{prop:sigma.invariance}
If $\nabla$ is torsion-free, then it holds that 
$$
\sigma(N^{(2)}) = \sigma^{-1}(N^{(2)}) = N^{(2)}.
$$
\end{prop}
\begin{proof}
Take a general element $\sum_i c_i \nabla(\exd b_i)$ in $N^{(2)}$, and observe that 
\begin{align*}
\sigma^{-1}\!\left(\sum_i c_i\nabla(\exd b_i)\right) = \sum_i c_i\nabla^R(\exd b_i).
\end{align*}
Since $\nabla^R(\exd b_i) \in N^{(2)}$, and $N^{(2)}$ is a $B$-bimodule, we must have that each summand  $c_i\nabla^R(\exd b_i)$ is contained in $N^{(2)}$. Thus we see that $\sigma^{-1}$ maps $N^{(2)}$ to itself. An analogous argument shows that $\sigma$ maps $N^{(2)}$ to itself. The claimed equalities now follow from the fact that $\sigma$ and $\sigma^{-1}$ are inverse to each other.
\end{proof}

\begin{cor} \label{cor:sigma.squared}
If $\nabla$ is torsion-free, then, for any $\omega,\nu\in\Omega^1(B)$, the following elements lie in  $N^{(2)}$:
\begin{enumerate}
\item $\sigma(\omega \otimes \nu) - \sigma^{-1}(\omega \otimes \nu)$,
\item $\omega \otimes \nu - \sigma^{-2}(\omega \otimes \nu)$.
\end{enumerate}
\end{cor}
\begin{proof}
It follows from \eqref{eqn:wedge.sigma} and \eqref{eqn:wedge.sigma.inv} that 
\begin{align*}
\wedge\big(\sigma(\omega \otimes \nu) - \sigma^{-1}(\omega \otimes \nu)\big) = - \omega \wedge \nu + \omega \wedge \nu = 0.
\end{align*}
Thus $\sigma(\omega \otimes \nu) - \sigma^{-1}(\omega \otimes \nu)$ is a relation as claimed. Operating on this relation by $\sigma^{-1}$ we now get the second relation.
\end{proof}

\begin{rem}
We note that applying $\sigma$ to the relation $\sigma(\omega \otimes \nu) - \sigma^{-1}(\omega \otimes \nu)$ gives the relation $\sigma^2(\omega \otimes \nu) - \omega \otimes \nu$. Thus, in the torsion-free case,  we have an alternative proof of Proposition \ref{prop:sigma.squared}.
\end{rem}

\subsection{Conjugate Connections from $*$-Invertible Generalised Flip Maps}

In this subsection we will assume that $B$ is a $*$-algebra and that $\Omega^{\bullet}(B)$ is a $*$-dc over $B$, and then we investigate how $\sigma$, $\nabla$, and the $*$-map interact. (For an alternative approach to conjugating bimodule connections, see \cite[\textsection 3.4.3]{BM.Leabh}.)

Consider first the usual \emph{lift} of the $*$-map of $\Omega^{2}(B)$ to the map
\begin{align*}
\dagger: \Omega^1(B) \otimes_B \Omega^1(B) \to \Omega^1(B) \otimes_B \Omega^1(B), & & \omega \otimes \nu \mapsto  \nu^* \otimes \omega^*.
\end{align*}
This gives an isomorphism of \emph{real} vector spaces. Moreover, we have the identity 
$$
* \circ \wedge = - \wedge \circ \, \dagger.
$$

\begin{defn} \label{defn:star.invt}
We say that $\sigma$ is \emph{$*$-invertible} if it is invertible and 
\begin{align}
\sigma \circ \dagger =   \dagger \circ \, \sigma^{-1}, & & \textrm{ or equivalently } & &  \dagger \circ \, \sigma \circ \dagger =  \sigma^{-1}.
\end{align}
\end{defn}

\begin{prop} \label{prop:star.spec}
Let $\sigma$ be a $*$-invertible generalised flip map. Then $\dagger$ permutes the set of eigenvectors of $\sigma$. Moreover, $\mathrm{spec}(\sigma)$, the set of eigenvalues of $\sigma$, admits the following permutation:
\begin{align*}
\mathrm{spec}(\sigma) \to \mathrm{spec}(\sigma), & & \lambda \mapsto \overline{\lambda^{-1}}.
\end{align*}
\end{prop}
\begin{proof}
Let $\lambda \in \mathrm{spec}(\sigma)$, and let $X \in \Omega^1(B) \otimes_B \Omega^1(B)$ be a $\lambda$-eigenvector. Then
\begin{align*}
\sigma(X^{\dagger}) =  \dagger \circ \sigma^{-1}(X) = \dagger(\lambda^{-1}X) = \overline{\lambda^{-1}}X^{\dagger}. 
\end{align*}
Thus $X^{\dagger}$ is an eigenvector of $\sigma$ with eigenvalue $\overline{\lambda^{-1}}$. Thus we see that $\dagger$ induces a map from the set of eigenvectors of $\sigma$ to itself, which is clearly involutive, and hence a permutation. Moreover, $\dagger$ induces a permutation of $\mathrm{spec}(\sigma)$.
\end{proof}

In the following proposition, we show that the assumption of $*$-invertibility allows us to construct a new connection from $\nabla$ and the $*$-map.

\begin{prop} \label{prop:conjugate.connection}
If $\sigma$ is $*$-invertible, then a left bimodule connection for the $*$-fodc $\Omega^1(B)$, with generalised flip map $\sigma$, is given by 
$$
\overline{\nabla} :=  \sigma \circ \dagger \circ \nabla \circ \ast.
$$
Moreover, it holds that $\overline{\overline{\nabla}} = \nabla$. We call $\overline{\nabla}$ the \emph{conjugate connection} of $\nabla$.
\end{prop}
\begin{proof}
Note first that, for $b \in B$, and $\omega \in \Omega^1(B)$, it holds that 
\begin{align*}
\overline{\nabla}(b \omega) =  \sigma \circ \dagger \circ \nabla \circ \ast(b\omega) 
=  \sigma \circ \dagger \circ \nabla(\omega^* b^*) 
=   \sigma \circ \dagger\Big(\nabla(\omega^*) b^* + \sigma(\omega^* \otimes \exd b^*)\Big).
\end{align*}
Noting that the definition of $\dagger$ implies that 
\begin{align} \label{eqn:dagger.nabla}
\dagger\big(\nabla(\omega^*) b^*\big) =  b\big(\dagger \circ \,\nabla(\omega^*)\big),
\end{align}
we see that our expression reduces to 
\begin{align*}
 b \big(\sigma \circ \dagger \circ \nabla(\omega^*)\big)  + \sigma \circ \dagger \circ \sigma \circ \dagger(\exd b \otimes \omega) = b \overline{\nabla}(\omega)  + \sigma \circ \sigma^{-1}(\exd b \otimes \omega) = b\overline{\nabla}(\omega) + \exd b \otimes \omega.
\end{align*}
Thus $\overline{\nabla}$ is a left connection as claimed. The fact that $\overline{\nabla}$ is a bimodule connection, with generalised flip map $\sigma$, follows from the calculation
\begin{align*}
\overline{\nabla}(\omega b) = & \,  \sigma \circ \dagger \circ \nabla \circ \ast(\omega b) \\
= & \, \sigma \circ \dagger \circ \nabla(b^*\omega^*) \\
=  & \, \sigma \circ \dagger\Big(\exd b^* \otimes \omega^* + b^*\nabla(\omega^*)\Big)\\
= & \, \sigma \Big(\omega \otimes \exd b +  \dagger \circ \nabla(\omega^*)b\Big) \\
= & \, \sigma(\omega \otimes \exd b) +  \overline{\nabla}(\omega)b.
\end{align*} 
Finally the calculation 
\begin{align*}
\overline{\overline{\nabla}} =  \sigma \circ \dagger \circ \overline{\nabla} \circ \ast = \sigma \circ \dagger \circ (\sigma \circ \dagger \,\circ \nabla \, \circ \ast) \circ \ast = \sigma \circ \sigma^{-1} \circ \nabla = \nabla 
\end{align*}
shows that conjugating a connection is an involutive operation.
\end{proof}

\begin{cor}
If $T_{\nabla}$ is a $B$-bimodule map, then 
\begin{align} \label{eqn:conj.torsion}
T_{\overline{\nabla}} = \ast \circ T_{\nabla} \circ \ast.
\end{align}
Thus  $\nabla$ is torsion-free if and only if $\overline{\nabla}$ is torsion-free.
\end{cor}
\begin{proof}
Expanding the torsion of $\overline{\nabla}$ we see that 
\begin{align*}
T_{\overline{\nabla}} =  \wedge \circ \overline{\nabla} - \exd =  \wedge \circ \, \sigma \circ \dagger \circ \nabla \circ \ast - \exd.
\end{align*}
Now since $\wedge \circ \sigma = - \wedge$, and $\exd$ is a $*$-map, this is equal to
\begin{align*}
-  \wedge \circ \dagger \circ \, \nabla \circ \ast - \ast \circ \exd \circ \ast.
\end{align*}
Recalling next that $\ast \circ \wedge = - \wedge \circ \, \dagger$, we see that this reduces to
\begin{align*}
 \ast \circ \wedge \circ \nabla \circ \ast - \ast \circ \exd \circ \ast 
=  \ast \circ ( \wedge \circ \nabla  - \exd) \circ \ast =  \ast \circ T_{\nabla} \circ \ast,
\end{align*}
giving the claimed identity. 
\end{proof}

\begin{defn}
We say that a bimodule connection $\nabla$ is \emph{self-conjugate} if $\nabla = \overline{\nabla}$.
\end{defn}

It is instructive to note that, for any self-conjugate connection $\nabla$ we have the identity 
\begin{align*}
\nabla \circ \ast = \sigma \circ \dagger \circ \nabla.
\end{align*}
If moreover $T_{\nabla}$ is a bimodule map, then 
\begin{align*}
T_{\nabla} = \ast \circ T_{\nabla} \circ \ast.
\end{align*}
Thus we see that, for this case, $T_{\nabla}$ is a $*$-map.


\section{The Quantum Homogeneous Space Case} \label{section:QHS}

We now treat the special case of a quantum homogeneous space, and show how the generalised flip map $\sigma$ associated to a bimodule connection admits a particularly simple form. Notably, this generalised flip map is independent of the choice of connection. We also show that in the $*$-calculus setting, $\sigma$ is always $*$-invertible.

\subsection{Some Preliminaries on Quantum Homogeneous Spaces}

We use Sweedler notation for Hopf algebra coproducts and coactions. We denote by $\Delta$, $\e$ and $S$ the coproduct, counit and antipode of a Hopf algebra respectively. We write $A^{\circ}$ for the finitary dual coalgebra of an algebra $A$, and moreover $A^{\circ}$ for the dual Hopf algebra of a Hopf algebra $A$. In both cases we denote the pairing between $A$ and $A^{\circ}$ by angular brackets. Moreover, all Hopf algebras are assumed to have bijective antipodes.

We begin by briefly recalling Takeuchi's equivalence for relative Hopf modules, see \cite{Takeuchi} or \cite[Appendix A]{GAPP} for more details. For $A$ a Hopf algebra, we say that a left coideal subalgebra $B \subseteq A$ is a \emph{quantum homogeneous \mbox{$A$-space}} if $A$ is faithfully flat as a right $B$-module and $B^+A = AB^+$, where $B^+ := \ker(\e|_B)$. This gives the canonical Hopf algebra projection $\pi_B:A \to A/B^+A$, with respect to which $B = \{b \in A \,|\, b_{(1)} \otimes \pi_B(b_{(2)}) = b \otimes 1\}$. We denote by ${}^{A}_B\mathrm{mod}_B$ the category of two-sided relative Hopf modules that are finitely generated as left $B$-modules. Moreover, we denote by ${}^{\pi_B}\mathrm{mod}_B$ the category whose objects are finite-dimensional left $\pi_B(A)$-comodules $\Delta_L: V \to \pi_B(A) \otimes V$, endowed with a right $B$-module structure, such that $\Delta_L(vb) = v_{(-1)}\pi_B(b_{(1)}) \otimes v_{(0)}b_{(2)}$, for all $v \in V, b \in B$, and whose morphisms are left $\pi_B(A)$-comodule maps and right $B$-module maps.

An equivalence of categories, known as Takeuchi's equivalence, is given by the functor $\Phi:{}^{A}_B\mathrm{mod}_B \to {}^{\pi_B}\mathrm{mod}_B$, where $\Phi(\F) = \F/B^+\F$, for any relative Hopf module $\F$, and the functor $\Psi:{}^{\pi_B}\mathrm{mod}_B \to {}^{A}_B\mathrm{mod}_B$ is defined using the cotensor product $\square_{\pi_B}$ over $\pi_B(A)$. A unit  $\unit: \F \to (\Psi \circ \Phi)(\F)$ for the equivalence is given by $\unit(f) =  f_{(-1)} \otimes [f_{(0)}]$, where $[f_{(0)}]$ denotes the coset of $f_{(0)}$ in $\Phi(\F)$. 

Consider next the full subcategory $^A_B\textrm{mod}_0$ of $^A_B\textrm{mod}_B$, whose objects $\mathcal{F}$ have the $B$-module structure 
\begin{align*}
fb = f_{(-2)}bS(f_{(-1)})f_{(0)}, & & \textrm{ for } f \in \F, \, b \in B. 
\end{align*}
Equivalently, the objects of ${}^A_B\mathrm{mod}_0$ can be identified as those objects $\FF$ in ${}^A_B\mathrm{mod}_B$ satisfying $\FF B^+ = B^+\FF$. The corresponding full subcategory ${}^{\pi_B}\mathrm{mod}_0$ of ${}^{\pi_B}\mathrm{mod}_B$ has as objects left $\pi_B(A)$-comodules $V$ with the right $B$-module structure $vb = \e(b)v$, for all $v \in V$, $b \in B$.

The category $^A_B\textrm{mod}_B$ has an evident monoidal structure $\otimes_B$, with respect to which  $^A_B\textrm{mod}_0$ is a monoidal subcategory. Moreover,  ${}^{\pi_B}\mathrm{mod}_0$ has the monoidal structure $\otimes$ given by the tensor product of $\pi_B(A)$-comodules. For $\F,\mathcal{G} \in {}^A_B\mathrm{mod}_0$, the natural isomorphism 
\begin{align*}
\mu_{\F,\mathcal{G}}:\Phi(\F) \otimes \Phi(\mathcal{G}) \to  \Phi(\F \otimes_B \mathcal{G}), & & [f] \otimes [g] \mapsto [f \otimes g]
\end{align*}
makes Takeuchi's equivalence into a monoidal equivalence between $^A_B\textrm{mod}_0$ and  ${}^{\pi_B}\mathrm{mod}_0$. See \cite[\textsection 4]{MMF2} for further details. In what follows we will tacitly identify $\Phi(\F) \otimes \Phi(\mathcal{G})$ and $\Phi(\F \otimes_B \mathcal{G})$.


A dc $\Omega^{\bullet}(B)$ over $B$ is said to be \emph{left $A$-covariant}, or simply \emph{covariant}, if the coaction $\Delta_L:B \to A \otimes B$ extends to a (necessarily unique) map $\Delta_L: \Omega^{\bullet}(B) \to A \otimes \Omega^{\bullet}(B)$ giving $\Omega^{\bullet}(B)$ the structure of a left $A$-comodule algebra, such that $\exd$ is a left $A$-comodule map. Moreover, we will always assume that covariant dc are finitely generated as left $B$-modules, and that their underlying covariant bimodules are objects in the subcategory ${}^A_B\mathrm{mod}_0$. Covariant fodc are defined similarly, and we also assume that they live in the category ${}^A_B\mathrm{mod}_0$. For an object $\F \in {}^A_B\mathrm{mod}_B$, we say that a connection  $\nabla:\FF \to \Omega^1(B) \otimes_B \FF$ is \emph{left $A$-covariant}, or simply \emph{covariant}, if $\nabla$  is a left $A$-comodule map. Note that if $\nabla$ is a bimodule connection with generalised flip map $\sigma$, then $\sigma$ is automatically a morphism in ${}^A_B\mathrm{mod}_B$.

\subsection{An Explicit Formula for $\sigma$}

Throughout the rest of this section, $A$ will denote a Hopf algebra, $B \subseteq A$ a quantum homogeneous $A$-space, and $\Omega^1(B)$ a left $A$-covariant fodc over $B$ that is contained in the subcategory ${}^A_B\mathrm{mod}_0$.  Moreover, $\F$ will be an object in ${}^A_B\mathrm{mod}_0$, endowed with a covariant bimodule connection $\nabla$. We see that these assumptions are sufficient to produce an explicit formula for the generalised flip map $\sigma$.

\begin{prop} \label{lem:covariant.sigma.FF}
For any object $\FF \in {}^A_B\mathrm{mod}_0$, and any covariant bimodule connection $\nabla:\FF \to \Omega^1(B) \otimes_B \FF$, the generalised flip map $\sigma$ of $\nabla$ satisfies
$$
\sigma(f \otimes \exd b) = \exd(f_{(-2)}bS(f_{(-1)})) \otimes f_{(0)} .
$$
Hence all the generalised flip maps associated to left $A$-covariant bimodule connections for $\FF$ coincide. 
\end{prop}
\begin{proof}
We begin by noting that  
\begin{align*}
\nabla(fb) =  \, \nabla(f_{(-2)}bS(f_{(-1)})f_{(0)}) 
=  \, \exd(f_{(-2)}bS(f_{(-1)}))\otimes f_{(0)} + f_{(-2)}bS(f_{(-1)})\nabla(f_{(0)}).
\end{align*}
Moreover, since $\nabla$ is a left $A$-comodule map, we must have that 
$$
\nabla(f)b = f_{(-2)}bS(f_{(-1)})\nabla(f_{(0)}).
$$
Taking these two expressions together we see that  
\begin{align*}
 \sigma(f \otimes \exd b) =  \nabla(fb) - \nabla(f)b = \exd(f_{(-2)}bS(f_{(-1)}))\otimes f_{(0)},
\end{align*}
giving us the claimed formula.
\end{proof}

Since $\sigma$ is a morphism of relative Hopf modules, we immediately have a corresponding morphism in the category of $\pi_B(A)$-comodules, which by abuse of notation we again denote by $\sigma$. 

\begin{cor} \label{cor:local.sigma}
From the monoidal version of Takeuchi’s equivalence, it follows that
\begin{align*}
\sigma([f] \otimes [\exd b]) = [\exd(f_{(-2)}bS(f_{(-1)}))] \otimes [f_{(0)}], & & \textrm{ for all } f \in \F, \, b \in B.
\end{align*}
\end{cor}


\begin{cor} \label{cor:FistheCalculus}
For the special case of $\F = \Omega^1(B)$, it holds that 
\begin{align*}
\sigma: \Omega^1(B) \otimes_B \Omega^1(B) \to \Omega^1(B) \otimes_B \Omega^1(B), & & \exd b \otimes \exd c \mapsto \exd(b_{(1)}cS(b_{(2)})) \otimes \exd b_{(3)}.
\end{align*}
Consequently, the map $\sigma: \Phi(\Omega^1(B)) \otimes \Phi(\Omega^1(B)) \to \Phi(\Omega^1(B)) \otimes \Phi(\Omega^1(B))$ satisfies  
\begin{align*}
\sigma([\exd b] \otimes [\exd c]) = [\exd(b_{(1)}cS(b_{(2)}))] \otimes [\exd b_{(3)}], & & \textrm{ for all } b,c \in B.
\end{align*}
\end{cor}

We finish with another consequence of our calculus being contained in the category ${}^A_B\mathrm{mod}_0$.

\begin{prop}
If $\F = \Omega^1(B) \in {}^A_B\mathrm{mod}_0$, then 
$
\mathrm{im}(\sigma + \id) \subseteq N^{(2)}.
$
\end{prop}
\begin{proof}
Since $\nabla, \exd$, and $\wedge$ are left $A$-covariant, $T_\nabla$ is a left $A$-comodule map. Moreover, it is also left $B$-linear. As both its domain and codomain belong to ${}_B^A\mathrm{mod}_0$, it automatically preserves their right $B$-actions, which is to say, $T_{\nabla}$ is a $B$-bimodule map. It now follows from the equivalence discussed immediately before \eqref{eqn:wedge.sigma} that $\mathrm{im}(\sigma + \id)$ is contained in $N^{(2)}$ as claimed.
\end{proof}

\subsection{$*$-Invertibility and Covariant Conjugate Connections}

The question of invertibility of the bimodule map can  be difficult to verify in practice. However, the following proposition says that if our fodc is a covariant $*$-fodc, then $*$-invertibility holds automatically.  Thus throughout this subsection we additionally assume that $A$ is a Hopf $*$-algebra, and that $B\subseteq A$ is a $*$-subalgebra.

\begin{prop} \label{prop:cov.star.invt}
Let $\Omega^1(B)$ be a left $A$-covariant $*$-fodc over $B$ contained in the category ${}^A_B\mathrm{mod}_0$. For any covariant bimodule connection $\nabla$, the generalised flip map $\sigma$ is $*$-invertible. Moreover, we have the identity
$$
\sigma^{-1} = \dagger \circ \sigma \circ \dagger, 
$$
as well as the explicit formula
\begin{align*}  
\sigma^{-1}(\exd b \otimes \exd c) = \exd c_{(3)} \otimes \exd\!\left(S^{-1}(c_{(2)})bc_{(1)}\right), & & \textrm{ for all } b,c \in B.
\end{align*}
\end{prop} 
\begin{proof}
Let us first produce an explicit expression for $\dagger \circ \sigma \circ \dagger$:
\begin{align*}
\dagger \circ \sigma \circ \dagger(\exd b \otimes \exd c) = & \, \dagger \circ \, \sigma(\exd c^* \otimes \exd b^*) \\
= & \, \dagger(\exd(c^*_{(1)}b^*S(c^*_{(2)})) \otimes \exd c^*_{(3)})\\
= & \, \exd c_{(3)} \otimes  \exd(c^*_{(1)}b^*S(c^*_{(2)}))^*\\
= & \, \exd c_{(3)} \otimes  \exd(S(c^*_{(2)})^*bc_{(1)})\\
= & \, \exd c_{(3)} \otimes  \exd(S^{-1}(c_{(2)})bc_{(1)}),
\end{align*}
where for the last identity we have used the standard identity $S(a^*)^* = S^{-1}(a)$, for any element $a$ of the Hopf $*$-algebra $A$. Using this identity, we now observe that 
\begin{align*}
\sigma \circ \dagger \circ \sigma \circ \dagger (\exd b \otimes \exd c) = & \, \sigma(\exd c_{(3)} \otimes \exd\!\left(S^{-1}(c_{(2)})bc_{(1)})\right) \\
= & \, \exd\!\left(c_{(3)}S^{-1}(c_{(2)})bc_{(1)}S(c_{(4)})\right) \otimes \exd c_{(5)}\\
= & \, \exd\!\left(bc_{(1)}S(c_{(2)})\right) \otimes \exd c_{(3)}\\
= & \, \exd b \otimes \exd c.
\end{align*}
An analogous calculation shows that $\dagger \circ \sigma \circ \dagger \circ \sigma = \mathrm{id}$, and so we see that $\dagger \circ \sigma \circ \dagger$ is indeed the inverse of $\sigma$.
\end{proof}

The following formula in the category ${}^{\pi_B}\mathrm{mod}_0$ now follows immediately. 

\begin{cor} \label{cor:sigma.formula.noantipode}
For any $b,c \in B$, it holds that 
$
\sigma^{-1}([\exd b] \otimes [\exd c])  = [\exd c_{(3)}] \otimes [\exd (S^{-1}(c_{(2)})bc_{(1)})].
$
\end{cor}

We next examine conjugate connections in the quantum homogeneous space setting. As before, $B \subseteq A$ is a quantum homogeneous space, and $\Omega^{\bullet}$ is a covariant differential $*$-calculus over $B$. 

\begin{prop} \label{prop:cov:conjugate.connection} 
For a left $A$-covariant bimodule connection $\nabla: \Omega^1 \to \Omega^1 \otimes_B \Omega^1$, its conjugate connection $\overline{\nabla}$ (as defined in Proposition \ref{prop:conjugate.connection}) is also left $A$-covariant.
\end{prop}
\begin{proof}
We note first that we have the identities
\begin{align*}
\Delta_L \circ \ast = (\ast \otimes \ast) \circ \Delta_L, & & \Delta_L \circ \dagger = (\ast \otimes \dagger) \circ \Delta_L.
\end{align*}
Using these identities we calculate that 
\begin{align*}
\Delta_L \circ \overline{\nabla} & =  \, \Delta_L(\sigma \circ \dagger \circ \nabla \circ \ast ) \\
& =\, (\mathrm{id} \otimes \sigma) \circ \Delta_L \circ (\dagger \circ \nabla \circ \ast) \\
& =\,  (\mathrm{id} \otimes \sigma)  \circ (\ast \otimes \dagger) \circ \Delta_L \circ \nabla \circ \ast \\
&= \, (\mathrm{id} \otimes \sigma) \circ (\ast \otimes \dagger) \circ (\mathrm{id} \otimes \nabla) \circ \Delta_L \circ \ast \\
& =\, (\mathrm{id} \otimes \sigma) \circ (\ast \otimes \dagger) \circ (\mathrm{id} \otimes \nabla) \circ (\ast \otimes \ast) \circ \Delta_L \\
&= \, (\mathrm{id} \otimes \sigma) \circ (\mathrm{id} \otimes \dagger) \circ (\mathrm{id} \otimes \nabla) \circ (\mathrm{id} \otimes \ast) \circ \Delta_L \\
& =\, (\mathrm{id} \otimes \overline{\nabla}) \circ \Delta_L.
\end{align*}
Thus we see that $\overline{\nabla}$ is a left $A$-comodule map as claimed.
\end{proof}

The following result is an immediate consequence of Proposition \ref{prop:conjugate.connection}. We find it instructive to state it as a formal corollary, since it will be used in \textsection \ref{section:HK}.

\begin{cor} \label{cor:self.conjugate}
If $\Omega^1(B)$ admits a unique left $A$-covariant bimodule connection $\nabla$, then $\nabla$ is self-conjugate.
\end{cor}

\subsection{The Local $*$-Map and Local $\dagger$-Map}

In this short subsection, we observe that, even though the $*$-map, and $\dagger$-map, of a covariant $*$-fodc are obviously not morphisms in ${}^A_B\mathrm{mod}_0$, they induce anti-linear automorphisms of $\Phi(\Omega^1(B))$, and $\Phi(\Omega^1(B)) \otimes \Phi(\Omega^1(B))$, respectively.

\begin{prop}
Let $\Omega^1(B)$ be a left $A$-covariant $*$-fodc over $B$ contained in the category ${}^A_B\mathrm{mod}_0$. Well-defined anti-linear real isomorphisms are given by 
\begin{flushleft}
\(
\begin{alignedat}{2}
~~~ \ast:\quad
& \Phi(\Omega^1(B)) \to \Phi(\Omega^1(B)),
&\qquad
& [\omega] \mapsto [\omega^*],\\
~~~~~~~~~~~~ \dagger:\quad
& \Phi(\Omega^1(B)) \otimes \Phi(\Omega^1(B))
  \to \Phi(\Omega^1(B)) \otimes \Phi(\Omega^1(B)),
&\qquad
& [X] \mapsto [X^\dagger].
\end{alignedat}
\)
\end{flushleft}
\end{prop}
\begin{proof}
The definition of the functor $\Phi$ tells us that
$$
\Phi(\Omega^1(B)) = \Omega^1(B)/B^+\Omega^1(B).
$$
It follows from the definition of ${}^A_B\mathrm{mod}_0$ that $\Omega^1(B) B^+ = B^+\Omega^1(B)$. Thus 
\begin{align*}
(\Omega^1(B)B^+)^* = B^+\Omega^1(B) = \Omega^1(B)B^+,
\end{align*}
meaning that the $\ast$-map descends to a well-defined map on the quotient $\Phi(\Omega^1(B))$. An analogous argument shows that $\dagger$ induces a well-defined map on $\Phi(\Omega^1(B)) \otimes \Phi(\Omega^1(B))$.
\end{proof}

The following corollary now follows immediately  from Proposition \ref{prop:star.spec}.

\begin{cor} \label{cor:dagger.eigen.value}
Let $\nabla:\Omega^1(B) \to \Omega^1(B) \otimes_B \Omega^1(B)$ be a covariant bimodule connection, with generalised flip map $\sigma$. Then for any $\sigma$-eigenvector 
$$
[X] \in \Phi(\Omega^1(B)) \otimes \Phi(\Omega^1(B)),
$$
with eigenvalue $\lambda$, a $\sigma$-eigenvector, with eigenvalue $\overline{\lambda}^{-1}$, is given by $[X^{\dagger}]$.
\end{cor}

As a consequence, we observe that when $\Phi(\Omega^1(B))$ is finite-dimensional, then the multiplicity of any eigenvalue $\lambda$ is equal to the multiplicity of the eigenvalue $\overline{\lambda}^{-1}$.


\section{Noncommutative Geometric Yang--Baxter Solutions} \label{section:YangBaxter}

In this section we explore the natural question of when the bimodule map of a bimodule connection satisfies the Yang--Baxter equation. We do so by passing to the dual formulation in terms of quantum tangent spaces, which in practice is often more manageable. We identify a distinguished class of tangent spaces for which the Yang--Baxter equation is always satisfied, thereby producing a corresponding class of noncommutative geometric Yang--Baxter solutions.

Throughout this section, so as to avoid tedious repetition, $A$ will denote a Hopf algebra, $B \subseteq A$ a quantum homogeneous space, $\Omega^1(B) \in {}^A_B\mathrm{mod}_B$ a left $A$-covariant fodc with tangent space $T$, and $\nabla$ a left $A$-covariant bimodule connection for $\Omega^1(B)$.


\subsection{The Wedge Yang--Baxter Equation}

We prove a weaker form of the Yang--Baxter equation, obtained by composing both sides with the wedge map. This extends Matassa's result for quantum projective space $\OO_q(\mathbb{CP}^n)$ \cite[\textsection 5]{MMBimodRels} to the general setting of \textsection \ref{section:QHS}. We highlight that we do not use this result to prove the actual Yang--Baxter solutions below. Instead it serves as a motivating observation.

\begin{defn}
A $B$-bimodule $\F$ with left bimodule connection $(\nabla, \sigma)$ is called \emph{extendable} if the generalised flip map
\[
\sigma : \F \otimes_B \Omega^1 \to \Omega^1 \otimes_B \F
\]
extends to a (necessarily unique) map
\begin{align*}
\sigma^{[n]} : \F \otimes_B \Omega^n \to \Omega^n \otimes_B \F, & & \textrm{ for } n \in \mathbb{Z}_{ > 1},
\end{align*}
satisfying the identity
\begin{align} \label{eqn:sigma.extend}
\sigma^{[m+n]} \circ (\mathrm{id} \otimes \wedge) = (\wedge \otimes \mathrm{id})
\circ (\mathrm{id} \otimes \sigma^{[n]})
\circ (\sigma^{[m]} \otimes \mathrm{id}), & & \textrm{ for all } m,n \in \mathbb{Z}_{\geq 1},
\end{align}
where domain and range of the two maps are
$$
\F \otimes_B \Omega^m \otimes_B \Omega^n \to \Omega^{m+n} \otimes_B \F,
$$
and we have denoted $\sigma^{[1]}:= \sigma$.
\end{defn}
It was shown by Beggs and Majid in \cite[Lemma 4.14]{BM.Leabh} that if $\Omega^{\bullet}$ is the maximal prolongation of $\Omega^1$, and the curvature $\nabla^2$ is a right $B$-module map, then $\nabla$ is necessarily extendable. 

\begin{prop} \label{prop:extendable}
If $\Omega^1(B) \in {}^A_B\mathrm{mod}_0$, then the curvature $\nabla^2$, with respect to the maximal prolongation of $\Omega^1(B)$, is a right $B$-module map and $\nabla$ is extendable. Moreover, $\sigma$ satisfies the identity
\begin{align*}
(\wedge \otimes \id) \circ (\sigma \otimes \id) \, \circ (\id \otimes \sigma) \, \circ (\sigma \otimes \id) = (\wedge \otimes \id) \circ (\id \otimes \sigma) \, \circ (\sigma \otimes \id) \circ (\id \otimes \sigma).
\end{align*}
\end{prop}
\begin{proof}
The curvature is a left $B$-module map, and moreover a left $A$-comodule map since $\nabla$ is a left $A$-comodule map. Thus it is a morphism in the category  ${}^A_B\mathrm{mod}_0$. In particular, this means that $\nabla^2$ is a right $B$-module map. It now follows from the result of Beggs and Majid that $\nabla$ is extendable.

For the case of $m = n = 1$, the identity \eqref{eqn:sigma.extend} reduces to 
\begin{align} \label{eqn:BM.extend.nabla}
\sigma^{[2]} \circ (\mathrm{id} \otimes \wedge)
=
(\wedge \otimes \mathrm{id})
\circ (\mathrm{id} \otimes \sigma)
\circ (\sigma \otimes \mathrm{id}).
\end{align}
Looking at the left hand side of this identity, and recalling that $\wedge \circ \sigma = - \wedge$, we see that  
$$
\sigma^{[2]} \circ (\mathrm{id} \otimes \wedge) = - \sigma^{[2]} \circ (\mathrm{id} \otimes \wedge) \circ (\id \otimes \sigma) =  - (\wedge \otimes \id) \circ (\mathrm{id} \otimes \sigma)
\circ (\sigma \otimes \mathrm{id})\circ (\id \otimes \sigma). 
$$
Looking at the right hand side we see that 
$$
(\wedge \otimes \mathrm{id}) \circ (\mathrm{id} \otimes \sigma) \circ (\sigma \otimes \mathrm{id}) = - (\wedge \otimes \mathrm{id}) \circ (\sigma \otimes \id) \circ (\mathrm{id} \otimes \sigma) \circ (\sigma \otimes \mathrm{id}).
$$
Comparing the two results gives us the claimed identity.
\end{proof}

The fact that $\sigma$ satisfies this weaker wedge form of the Yang–Baxter equation naturally raises the question of whether it also satisfies the full Yang–Baxter equation. In the following subsections, we introduce a distinguished subclass of calculi for which this is indeed true.


\subsection{Some Preliminaries on Quantum Tangent Spaces} \label{subsection:remarksQHTS}

In this subsection we recall the tangent space formulation of covariant differential calculi over quantum homogeneous spaces. Let $A$ be a Hopf algebra, and $L \subseteq A^{\circ}$ a Hopf subalgebra such that the space of invariants
\begin{align} \label{eqn:L.invariant.QHS}
B := \, {}^L\!A = \Big\{b \in A \,|\, b_{(1)} \langle l, b_{(2)}\rangle = \e(l)b, \textrm{ for all } l \in L \Big\}
\end{align}
is a quantum homogeneous $A$-space. We denote by $B^{\circ}$ its dual coalgebra. We note that a left $A^{\circ}$-module structure for $B^{\circ}$ is defined by 
\begin{align*}
XY(b) = \langle X, b_{(1)}\rangle \langle Y,b_{(2)}\rangle, & & \textrm{ for } X \in A^{\circ}, \, Y \in B^{\circ}, \, b \in B.
\end{align*}

A \emph{quantum tangent space}, or simply a \emph{tangent space}, for $B$ is a finite-dimensional subspace $T \subseteq \ker(1_B) \subseteq B^{\circ}$ such that $T \oplus \mathbb{C}1$ is a right coideal of $B^{\circ}$ and $LT \subseteq T$. For any tangent space $T$, a right $B$-ideal, and left $\pi_B(A)$-subcomodule,  of $B^{+} := B \cap \mathrm{ker}(\e)$ is given by 
\begin{align*}
I := \big\{ x \in B^+ \,|\, X(x) = 0, \textrm{ for all } X \in T \big\}.
\end{align*}
This means that the quotient $V^1 := B^+/I$ is naturally an object in the category ${}^{\pi_B}\mathrm{mod}_B$. We call $V^1$ the \emph{cotangent space} of $T$. Note that $V^1$ has a natural right $L$-module structure, coming from its left $\pi_B(A)$-comodule structure and the dual pairing between $L$ and $\pi_B(A)$. With respect to this $L$-module structure, we have the identity
\begin{align} \label{eqn:dagger}
\langle lX,v\rangle  = \langle X,vl\rangle, & & \textrm{ for all } X \in T, v \in V^1, l \in L. 
\end{align}

Consider now the object 
\begin{align*}
\Omega^1(B) := A \, \square_{\pi_B} V^1.
\end{align*}
If $\{X_i\}_{i=1}^n$ is a basis for $T$, and $\{e_i\}_{i=1}^n$ is the dual basis of $V^1$, then the map 
\begin{align*}
\exd: B \to \Omega^1(B), & & b \mapsto \sum_{i=1}^n (X_i \triangleright b) \otimes e_i :=  \sum_{i=1}^n b_{(1)} \langle X_i,b_{(2)} \rangle \otimes e_i
\end{align*}
is a derivation, and the pair $(\Omega^1(B),\exd)$ is a left $A$-covariant fodc over $B$. This gives a bijective correspondence between isomorphism classes of finite-dimensional tangent spaces and left $A$-covariant fodc \cite{HKTangent}. Explicitly, an isomorphism between $V^1$ and $\Phi(\Omega^1(B))$ is given by mapping the coset $[b]$ to the coset $[\exd b]$, for $b \in B^+$.


\subsection{Primitive Tangent Spaces}

Let $B$ be a quantum homogeneous space of the form given in \eqref{eqn:L.invariant.QHS}. In this subsection we give a tangent space characterisation of when a covariant first-order differential calculus $\Omega^1(B)$ is an object in the subcategory ${}^A_B\mathrm{mod}_0$. We begin with a simple lemma followed by a definition.

\begin{lem} \label{lem:L.mod.prim}
The space $\mathrm{Prim}(B^{\circ})$, that is the space of primitive elements of $B^{\circ}$, is a left $L$-submodule of $B^{\circ}$.
\end{lem}
\begin{proof}
Let $X$ be an element of $\mathrm{Prim}(B^{\circ})$ and $l \in L$. We need to show that $lX$ is again a primitive element. However, this is clearly true, since
\begin{align*}
\Delta(lX) = l_{(1)}X \otimes \rho_B(l_{(2)}) + \rho_B(l_{(1)}) \otimes l_{(2)} X = lX \otimes 1 + 1 \otimes lX,
\end{align*}
where we have used the fact that $\rho_B(l) = \e(l)1_B^{\circ}$, for all $l \in L$.
\end{proof}

\begin{defn}
A \emph{primitive tangent space} for $B$ is a finite-dimensional $L$-submodule $T \subseteq \mathrm{Prim}(B^{\circ})$.
\end{defn}

Since primitive elements are always contained in the kernel of counit, we observe that a primitive tangent space is, in particular, a tangent space. In fact, as we will establish in the remainder of this subsection, a covariant fodc has a primitive tangent space if and only if the fodc lives in the subcategory ${}^{A}_B\mathrm{mod}_0$.

For any tangent space $T$ with basis $\{X_i\}_i$, we have that 
\begin{align*}
\Delta(X_i) = 1 \otimes X_i + \sum_{k} X_k \otimes Y_{ki}, & & \textrm{ for some elements } Y_{ki} \in B^{\circ}.
\end{align*}
We call $\{Y_{ki}\}_{k,i}$ the \emph{$Y$-matrix} of $T$, with respect to $\{X_i\}$. The following elementary, but instructive, lemma describes the right $B$-module structure of $V^1$ in terms of the $Y$-matrix.

\begin{lem}
Let  $\{e_i\}_i$ be the basis of $V^1$ dual to $\{X_i\}_i$. Then for any $c \in B$, it holds that 
\begin{align} \label{eqn:Y.matrix.B.action}
e_ic = \sum_j \langle Y_{ij},c \rangle e_j, & & \textrm{ for all } i.
\end{align} 
\end{lem}
\begin{proof}
For each $i$, let $b_i$ be an element of $B^+$ satisfying $e_i = [b_i]$. We denote 
$$
[b_i]c = [b_ic] =: \sum_j \lambda_{ij}[b_j].
$$
Duality of the bases $\{X_i\}_i$ and $\{e_i\}_i$ implies that the coefficient $\lambda_{ij}$ is given by 
\begin{align*}
\lambda_{ij} = \langle X_j,[b_ic] \rangle = \langle 1,b_i \rangle\langle X_j, c\rangle + \sum_k \langle X_k,b_i \rangle \langle Y_{kj}, c \rangle = \langle Y_{ij}, c\rangle,
\end{align*}
giving the claimed identity.
\end{proof}

\begin{prop} \label{prop:primitive.tangent}
It holds that $\Omega^1(B) \in {}^A_B\mathrm{mod}_0$ if and only if its tangent space $T$ is a primitive tangent space.
\end{prop}
\begin{proof}
Assume that, for some $j \neq i$, the element $Y_{ij}$ is a non-zero element of $B^{\circ}$, and let $c \in B$ pair non-trivially with $Y_{ij}$. Then in the expansion of $e_ic$ given in \eqref{eqn:Y.matrix.B.action}, the coefficient of $e_j$ is non-zero. Thus $e_ic$ cannot be equal to $\e(c)e_i$, meaning that $\Omega^1(B)$ is not contained in the subcategory ${}^A_B\mathrm{mod}_0$. Hence if $\Omega^1(B) \in {}^A_B\mathrm{mod}_0$, then it holds that
\begin{align}
\Delta(X_i) = 1 \otimes X_i +  X_i \otimes Y_{ii}.
\end{align}
It now follows from \eqref{eqn:Y.matrix.B.action} that 
\begin{align*}
e_ic = \langle Y_{ii},c \rangle e_i, & & \textrm{ for all } c \in B.
\end{align*}
Thus when $\Omega^1(B) \in {}^A_B\mathrm{mod}_0$, we must have that 
\begin{align*}
 \langle Y_{ii},c \rangle = \e(c), & & \textrm{ for all } c \in B,
\end{align*}
and so, $Y_{ii} = 1$. In other words, the calculus can be contained in ${}^A_B\mathrm{mod}_0$ only if $X_i$ is primitive. 

In the other direction, let us assume that each element $X_i$ is primitive. For any $c \in B$, it follows from  \eqref{eqn:Y.matrix.B.action} that 
\begin{align*}
e_i c = \langle 1, c \rangle e_i = \e(c) e_i, & & \textrm{ for all } i.
\end{align*}
Thus $\Omega^1(B)$ is contained in the subcategory ${}^A_B\mathrm{mod}_0$.
\end{proof}


\subsection{The Tangent Space Dual of $\sigma$}

Rather than working with $\sigma$ directly, considering the dual map, at the level of the tangent space, yields a useful alternative approach. Extend the non-degenerate pairing between $T$ and $V^1$ in the usual manner
\begin{align*}
T \otimes T \times V^1 \otimes V^1 \to \mathbb{C}, & & (X \otimes Y, \, [b] \otimes [c]) \mapsto \langle X, b\rangle \langle Y, c\rangle. 
\end{align*}
Then since this is a non-degenerate pairing of finite-dimensional vector spaces, the map $\sigma$ dualises to a map 
$$
\widecheck{\sigma}: T \otimes T \to T \otimes T,
$$
uniquely defined by the identity
\begin{align*}
\langle \widecheck{\sigma}(X \otimes Y), v \otimes w\rangle = \langle X \otimes Y, \sigma(v \otimes w)\rangle, & &  \textrm{ for } X, Y \in T, \, \textrm{ and } v,w \in V^1.
\end{align*}
Moreover, it follows from \eqref{eqn:dagger} that $\widecheck{\sigma}$ is a left $L$-module map.

In the next proposition we observe that when one of the tangent vectors is of a particular form, then the map $\widecheck{\sigma}$ admits an explicit presentation. (As we will see in \textsection \ref{section:HK}, the tangent vectors of the Heckenberger--Kolb tangent spaces are of this particular form.) In what follows, we use $\rho_B:A^{\circ} \to B^{\circ}$ to denote the obvious restriction of domains.

\begin{defn}
We say that an element $X \in A^{\circ}$ is an \emph{$L$-twisted primitive element} if it  satisfies $\Delta(X) = X \otimes K_1 + K_2 \otimes X$, for two grouplike elements $K_1,K_2 \in L$. 
\end{defn}

\begin{prop} \label{prop:THEFORMULA}
Let $X$ be an element of $A^{\circ}$ such that $\rho_B(X) \in T$, then for any $l \in L$, it holds that $\rho_B(lX) \in T$. Moreover, if $X$ is $L$-twisted primitive, and $Y$ is another element of $A^{\circ}$ satisfying $\rho_B(Y) \in T$, then it holds that   
\begin{align} \label{eqn:LTprimitive formula}
\widecheck{\sigma}\big(\rho_B(lX) \otimes \rho_B(Y)\big) = l_{(1)}K_2K_1^{-1}S(l_{(3)})\rho_B(Y) \otimes l_{(2)}\rho_B(X).
\end{align} 
\end{prop}
\begin{proof}
Dualising the right $L$-module structures of $A$ and $B$ to left $L$-module structures on $A^{\circ}$ and $B^{\circ}$, we see that the restriction map $\rho_B$ is a left $L$-module map. Thus, since $T$ is a left $L$-submodule of $B^{\circ}$, we must have that $\rho_B(lX) = l\rho_B(X)$ is an element of $T$. 

It follows from the description of $\sigma$ given in Corollary \ref{cor:FistheCalculus} that, for $b,c \in B^+$,
\begin{align*}
\langle \rho_B(lX) \otimes \rho_B(Y), \, \sigma([b] \otimes [c]) \rangle = & \, \langle \rho_B(lX) \otimes \rho_B(Y), \, [b_{(1)}cS(b_{(2)})] \otimes [b_{(3)}] \rangle \\
= & \,  \langle lX, b_{(1)}cS(b_{(2)})\rangle \langle Y, b_{(3)}\rangle.
\end{align*}
Using the properties of the dual pairing we  expand this expression to 
\begin{align*}
 \langle (lX)_{(1)}, b_{(1)} \rangle \langle (lX)_{(2)}, c \rangle \langle S((lX)_{(3)}), b_{(2)}\rangle \langle Y, b_{(3)} \rangle,
\end{align*}
which in turn reduces to 
\begin{align} \label{eqn:pre.coproduct.sigma.Check}
 \langle (lX)_{(1)}S((lX)_{(3)})Y, b \rangle \langle (lX)_{(2)}, c\rangle.
\end{align}
Using our assumption that $\Delta(X) = X \otimes K_1 + K_2 \otimes X$, we now calculate the coproduct 
\begin{align*} 
(\id \otimes \Delta) \circ \Delta(lX) = & \, (\id \otimes \Delta)(l_{(1)}X \otimes l_{(2)}K_1 + l_{(1)}K_2 \otimes l_{(2)}X)\\
= & \, l_{(1)}X \otimes l_{(2)}K_1 \otimes l_{(3)}K_1 + l_{(1)}K_2 \otimes l_{(2)} X \otimes l_{(3)}K_1 \\
& \, ~~~~ ~~ + l_{(1)}K_2 \otimes l_{(2)}K_2 \otimes l_{(3)}X.
\end{align*}
Since for any $k \in L$ the term $\langle k,c\rangle = \e(k)\e(c)$ vanishes, it holds that 
\begin{align*} 
\langle  l_{(1)}X S(l_{(3)}K_1)Y, b\rangle \langle l_{(2)}K_1,c\rangle  
= \langle l_{(1)}K_2S(l_{(3)}X)Y,b\rangle \langle l_{(2)}K_2, c\rangle = 0. 
\end{align*}
Thus inserting our expression for $(\id \otimes \Delta) \circ \Delta(lX)$ into the identity \eqref{eqn:pre.coproduct.sigma.Check}, gives
\begin{align} \label{eqn:penultimate.sigma.Check}
\langle l_{(1)}K_2S(l_{(3)}K_1) Y, b\rangle \langle l_{(2)}X,c \rangle = \langle l_{(1)}K_2K_1^{-1}S(l_{(3)}) Y, b\rangle \langle l_{(2)}X,c \rangle.
\end{align}
Now since $T$ is a left $L$-module, we see that $l_{(1)}K_2K_1^{-1}S(l_{(3)})\rho_B(Y)$ and $l_{(2)}\rho_B(X)$ are both elements of $T$. Hence \eqref{eqn:penultimate.sigma.Check} can be written as 
\begin{align*}
\langle l_{(1)}K_2K_1^{-1}S(l_{(3)})\rho_B(Y) \otimes l_{(2)}\rho_B(X), \, [b] \otimes [c]\rangle.
\end{align*}
The claimed formula now follows from non-degeneracy of the pairing between $T \otimes T$ and $V^1 \otimes V^1$.
\end{proof}

We finish with a crucial generating condition for a tangent space, and then a simple but important observation. 

\begin{defn}
We say that a tangent space $T$ is \emph{generated by $L$-twisted primitive elements} if it is generated as a left $L$-module by elements of the form $\rho_B(X)$, where $X \in A^{\circ}$ is an $L$-twisted primitive element.
\end{defn}

\begin{prop} \label{prop:L.twisted.implies.primitive}
If $T$ is generated by $L$-twisted primitive elements, then it is a primitive tangent space and hence the associated calculus $\Omega^1(B)$ is contained in ${}^A_B\mathrm{mod}_0$. 
\end{prop}
\begin{proof}
If $X \in A^{\circ}$ is an $L$-twisted primitive element,  then by definition there exist grouplike elements $K_1,K_2 \in L$, such that 
$
\Delta(X) = X \otimes K_1 + K_2 \otimes X.
$
The fact that $K_1$ and $K_2$ are contained in $L$ means that their restriction to $B$ is equal  to $\e(K_i)1 = 1$, which is to say $\rho_B(X)$ is primitive in $B^{\circ}$. Since we know from Lemma \ref{lem:L.mod.prim} that $\mathrm{Prim}(B^{\circ})$ is an $L$-submodule of $B^{\circ}$, it must hold that $L\rho_B(X) \subseteq \mathrm{Prim}(B^{\circ})$. Thus $T$ must be contained in $\mathrm{Prim}(B^{\circ})$, which is to say, $T$ is primitive. It now follows from Proposition \ref{prop:primitive.tangent} that $\Omega^1(B)$ is contained in ${}^A_B\mathrm{mod}_0$. 
\end{proof}


\subsection{Noncommutative Geometric Yang--Baxter Solutions}

In this subsection we prove one of the most important results of the paper, namely we give sufficient conditions for the dual bimodule map $\widecheck{\sigma}$ to satisfy the Yang--Baxter equation. We then observe that for any such $\widecheck{\sigma}$, the dual map $\sigma$ also satisfies the Yang--Baxter equation. 

\begin{thm} \label{thm.THE.THM}
If the tangent space $T$ is generated by $L$-twisted primitive elements, then the bimodule map $\widecheck{\sigma}$ satisfies the Yang--Baxter equation
$$
(\widecheck{\sigma} \otimes \id) \, \circ (\id \otimes \widecheck{\sigma}) \, \circ (\widecheck{\sigma} \otimes \id) = (\id \otimes \widecheck{\sigma}) \, \circ (\widecheck{\sigma} \otimes \id) \circ (\id \otimes \widecheck{\sigma}).
$$
\end{thm}
\begin{proof}
Let $X$ and $Y$ be two $L$-twisted primitive elements in $A^{\circ}$, such that $\rho_B(X), \, \rho_B(Y) \in T$. Explicitly 
\begin{align*}
\Delta(X) = X \otimes J_1 + J_2 \otimes X, & & \Delta(Y) = Y \otimes K_1 + K_2 \otimes Y, 
\end{align*}
for $J_1,J_2,K_1$, and $K_2$ some grouplike elements of $L$. For any $Z \in T$, and $l,m \in L$, we denote 
$$
\Theta :=  \rho_B(lX) \otimes  \rho_B(mY) \otimes Z.
$$
Since our tangent space is generated by $L$-twisted primitive elements, Proposition \ref{prop:L.twisted.implies.primitive} implies that our calculus is contained in ${}^A_B\mathrm{mod}_0$, and so, we can use Proposition \ref{prop:THEFORMULA}. Acting now on $\Theta$ by $\widecheck{\sigma}$ gives 
$$
l_{(1)}J_2J_1^{-1}S(l_{(3)})m\rho_B(Y) \otimes l_{(2)}\rho_B(X) \otimes Z.
$$
Acting next by $\id \otimes \widecheck{\sigma}$, we get the term
$$
l_{(1)}J_2J_1^{-1}S(l_{(5)})m\rho_B(Y) \otimes l_{(2)}J_2J_1^{-1}S(l_{(4)})Z \otimes l_{(3)}\rho_B(X).
$$
Finally, acting by $\widecheck{\sigma} \otimes \id$, we get the term
\begin{align*}
l_{(1)}J_2J_1^{-1}S(l_{(9)})m_{(1)}K_2K_1^{-1}S\big(l_{(3)}J_2J_1^{-1}S(l_{(7)})m_{(3)}\big)l_{(4)}J_2J_1^{-1}S(l_{(6)})Z \\
 \otimes \, l_{(2)}J_2J_1^{-1}S(l_{(8)})m_{(2)}\rho_B(Y) \otimes l_{(5)}\rho_B(X).~~~~~~~~~~~~~~~~~~~~~~~~
\end{align*}
Using the fact that $S$ is an anti-algebra map, and performing the obvious cancellations, this reduces to the expression
\begin{align*}
l_{(1)}J_2J_1^{-1}S(l_{(5)})m_{(1)}K_2K_1^{-1}S(m_{(3)})Z
 \otimes \, l_{(2)}J_2J_1^{-1}S(l_{(4)})m_{(2)}\rho_B(Y) \otimes l_{(3)}\rho_B(X).
\end{align*}
Looking instead at the right hand side of the Yang--Baxter equation, we act on $\Theta$ by $\id \otimes \widecheck{\sigma}$ to get
\begin{align*}
l\rho_B(X) \otimes m_{(1)}K_2K_1^{-1}S(m_{(3)})Z \otimes m_{(2)}\rho_B(Y).
\end{align*}
We next act by $\widecheck{\sigma} \otimes \id$ to get 
\begin{align*}
l_{(1)}J_2J_1^{-1}S(l_{(3)})m_{(1)}K_2K_1^{-1}S(m_{(3)})Z \otimes l_{(2)}\rho_B(X) \otimes m_{(2)}\rho_B(Y).
\end{align*}
Finally we act by $\id \otimes \widecheck{\sigma}$ to get 
\begin{align*}
l_{(1)}J_2J_1^{-1}S(l_{(5)})m_{(1)}K_2K_1^{-1}S(m_{(3)})Z
 \otimes \, l_{(2)}J_2J_1^{-1}S(l_{(4)})m_{(2)}\rho_B(Y) \otimes l_{(3)}\rho_B(X).
\end{align*}
Comparing the two terms, we see that 
$$
(\widecheck{\sigma} \otimes \id) \, \circ (\id \otimes \widecheck{\sigma}) \, \circ (\widecheck{\sigma} \otimes \id)(\Theta) = (\id \otimes \widecheck{\sigma}) \, \circ (\widecheck{\sigma} \otimes \id) \circ (\id \otimes \widecheck{\sigma})(\Theta).
$$
Since $T$ is generated as a left $L$-module by $L$-twisted primitive elements by assumption, it now follows that $\widecheck{\sigma}$ satisfies the Yang--Baxter equation.
\end{proof}


We finish by considering the dual matrix description of $\sigma$ with respect to a choice of basis $\{X_i\}_{i=1}^d$ of $T$. Denote 
\begin{align*}
\widecheck{\sigma}(X_i \otimes X_j) =: \sum_{k,l=1}^d c^{ij}_{kl} X_k \otimes X_l, & & \textrm{ for } i,j = 1, \dots, d.
\end{align*}
Then since $\widecheck{\sigma}$ is the adjoint of $\sigma$ with respect to the dual pairing of $T \otimes T$ and $V^1 \otimes V^1$, we see that
\begin{align*}
\sigma(e_i \otimes e_j) = \sum_{k,l=1}^d c^{kl}_{ij} e_k \otimes e_l, & & \textrm{ for } i,j = 1, \dots, d,
\end{align*}
where $\{e_i\}_{i=1}^d$ is the dual basis of $V^1$. 
From this we get the following immediate corollary to Theorem \ref{thm.THE.THM}.

\begin{cor} \label{cor.THE.COR}
If the tangent space $T$ is generated  by $L$-twisted primitive elements, then the bimodule map $\sigma$ satisfies the Yang--Baxter equation.
\end{cor}


\section{Almost Complex Structures and $\sigma$-Exterior Calculi} \label{section:ALC}

In this section, we investigate how the torsion-free bimodule connection description of the maximal prolongation interacts with a chosen complex structure on the differential calculus. In the $\sigma$-exterior case, we derive sufficient conditions under which a first-order complex structure extends to an almost complex structure. Moreover, we identify sufficient conditions ensuring that the resulting structure is factorisable. We also present specific results for the special case of quantum homogeneous spaces.

Throughout this section, so as to avoid tedious repetition, $B$ will denote an algebra, $\Omega^{\bullet}(B)$ a dc over $B$, and $\nabla$ a bimodule connection for $\Omega^{1}(B)$ with bimodule map $\sigma$. In the last two subsections, we specialise to the case where $B \subseteq A$ is a quantum homogeneous space of a Hopf algebra $A$.


\subsection{$\sigma$-Exterior Differential Calculi}

As discussed in \textsection \ref{subsection:bimod.conn.torsion}, the $B$-sub-bimodule  $\mathrm{im}(\sigma + \mathrm{id})$ is contained in the degree two relations of a prolongation if and only if the torsion operator $T_{\nabla}$ is a right $B$-module map. In this subsection we look at dc whose relations are quadratic and whose space of degree two relations is precisely $\mathrm{im}(\sigma + \mathrm{id})$. Such calculi will be used extensively in this section. 

\begin{defn} \label{defn:sigma.ext}
We say that $\Omega^{\bullet}(B)$ is \emph{$\sigma$-exterior} with respect to a bimodule connection $\nabla$, with associated $B$-bimodule map $\sigma$, if
\begin{align*}
\Omega^{\bullet}(B) = \mathcal{T}(\Omega^1(B))/\langle \mathrm{im}(\sigma + \mathrm{id}) \rangle.
\end{align*}
\end{defn}

An immediate observation is that a dc is $\sigma$-exterior with respect to $\nabla$ only if $T_{\nabla}$ is a right $B$-module map.

\begin{eg}
Let $M$ be a smooth manifold, and let $\Gamma^{\infty}(\Omega^1(M))$ be the space of smooth sections of its cotangent bundle $\Omega^{1}(M)$. Any connection $\nabla$ for $\Gamma^{\infty}(\Omega^1(M))$ is a bimodule connection whose bimodule map is the flip. Thus the space of sections of the de Rham complex of $M$ is $\sigma$-exterior with respect to any connection $\nabla$. In particular, we see that it is possible to have a $\sigma$-exterior calculus, where $\sigma$ is the bimodule map of a connection with torsion. 
\end{eg}

\begin{eg} 
Let $M$ be a compact Riemannian manifold endowed with a free isometric action of an $n$-torus $\mathbb{T}^n$, for some $n \geq 2$. The associated action of $\mathbb{T}^n$ on $C^{\infty}(M)$, the algebra of smooth functions of $M$, allows one to define a deformed algebra $C^{\infty}_{\theta}(M)$ known as the \emph{$\theta$-deformation} of $C^{\infty}(M)$ \cite{theta}. Moreover, a torus equivariant Dirac operator on $M$ can be $\theta$-deformed to a spectral triple over $C^{\infty}_{\theta}(M)$, that is to say, a noncommutative Riemannian manifold in the sense of Connes \cite{Connes.Leabh.Dearg}. Associated to the spectral triple is a fodc $\Omega^{1}_{\theta}(M)$.

It follows from  \cite[\textsection 7]{ff} and \cite[\textsection 7]{Koszul.Kolkata.Landi} that the maximal prolongation $\Omega^{\bullet}_{\theta}(M)$ of $\Omega^{1}_{\theta}(M)$ admits a bimodule connection $\nabla$, with respect to which $\Omega^{\bullet}_{\theta}(M)$ is $\sigma$-exterior. Subsequently, it was shown in \cite{Mesland.Rennie.Adv} that the freeness condition can be dropped for the special case of $n=2$. 
\end{eg}

\begin{eg} \label{eg:graph}
Let $X$ be a finite set and let  $B=\mathbb{C}(X)$ denote the algebra of complex-valued functions on $X$. 
First-order differential calculi on $B$ correspond to directed graphs $E$ (without loops or multiple edges) whose vertices are labelled by the elements of $X$, see \cite[Proposition 1.24]{BM.Leabh}. Explicitly, $\Omega^{1}(X)$ is the vector space spanned by the basis elements $\omega_{x\rightarrow y}$, which are indexed by the edges $x \rightarrow y$ of the graph. We give it a $B$-bimodule structure by defining 
\begin{align*}
f.\,\omega_{x\rightarrow y} := f(x)\omega_{x\rightarrow y},
& & 
\omega_{x\rightarrow y}.f := f(y)\omega_{x\rightarrow y}, && \textrm{ for } f \in \mathbb{C}(X).
\end{align*}
Finally, its differential is defined by 
\begin{align*}
\exd f &:= \sum_{x\rightarrow y \, \in \, E^1} (f(y)-f(x))\omega_{x\rightarrow y},
\end{align*}
where $E^1$ denotes the edge set of the graph $E$. 

In \cite{Sigma.Extendability.Graphs}, it is shown that these calculi always admit a connection with respect to which the maximal prolongation of $\Omega^1(X)$ is $\sigma$-exterior. It was also shown that a large class of graphs admit connections with respect to which the maximal prolongation is not $\sigma$-exterior.
\end{eg}

\begin{eg}
In \cite[Theorem 3.5]{MMBimodRels} it was shown that the Heckenberger--Kolb dc $\Omega^{\bullet}_q(\mathbb{CP}^n)$ of quantum projective space is $\sigma$-exterior with respect to the unique left $\OO_q(\mathrm{SU}_{n+1})$-covariant connection of $\Omega^{1}_q(\mathbb{CP}^n)$ \cite[Theorem 4.5]{HVBQFM}. This result is extended to all of the Heckenberger--Kolb calculi in \textsection \ref{subsection:HKsigmeExt} below.
\end{eg}


\subsection{Preliminaries on Complex and Almost Complex Structures}

In this subsection, we briefly recall some preliminaries about almost complex and complex structures. See \cite[\textsection 7]{BM.Leabh} or \cite{BS,KLVSCP1,MMF2} for a more detailed discussion of complex structures.

An {\em almost complex structure} $\Om^{(\bullet,\bullet)}$, for a $*$-dc $(\Om^{\bullet},\exd)$, is a $\mathbb{Z}^2_{\geq 0}$-algebra grading 
of $\Om^{\bullet}$ such that
\begin{align*} 
\Om^k(B) = \bigoplus_{a+b = k} \Om^{(a,b)}, \textrm{ for all } k \in \mathbb{Z}_{\geq 0}, & & \big(\Om^{(a,b)}\big)^* = \Om^{(b,a)}, & & \textrm{ for all } (a,b) \in \mathbb{Z}^2_{\geq 0}. 
\end{align*}
If the multiplication maps
\begin{align*}
\wedge: \Omega^{(a,0)} \otimes_B \Omega^{(0,b)} \to \Omega^{(a,b)}, & & \wedge: \Omega^{(0,b)} \otimes_B \Omega^{(a,0)} \to \Omega^{(a,b)}, & & \textrm{  for all } (a,b) \in \mathbb{Z}^2_{\geq 0},
\end{align*}
are isomorphisms, then we say that the almost complex structure is \emph{factorisable}. As observed in \cite[Proposition 6.7]{MMF2}, if $\Omega^\bullet$ is the maximal prolongation of its fodc, factorisability is equivalent to the wedge multiplications being isomorphisms for the special case of $(a,b) = (1,1)$. If the exterior derivative decomposes into a sum $\exd = \del + \adel$, for $\del$ a (necessarily unique) degree $(1,0)$-map, and $\adel$ a (necessarily unique) degree $(0,1)$-map, then we say that $\Om^{(\bullet,\bullet)}$ is a \emph{complex structure}. 

A \emph{first-order complex structure}, or \emph{focs}, for a $*$-fodc $\Omega^1(B)$ over a $*$-algebra $B$ is a direct sum decomposition of $B$-bimodules
\begin{align} \label{eqn:foacs}
\Omega^1(B) = \Omega^{(1,0)} \oplus \Omega^{(0,1)}
\end{align}
such that $(\Omega^{(1,0)})^* = \Omega^{(0,1)}$ or equivalently such that $(\Omega^{(0,1)})^* = \Omega^{(1,0)}$. We say that a focs on $\Omega^{1}(B)$ \emph{extends} to an almost complex structure $\Omega^{(\bullet,\bullet)}$ on a prolongation of $\Omega^1(B)$ if the decomposition of $1$-forms given by $\Omega^{(\bullet,\bullet)}$ is the decomposition of the focs. As is easily seen, there is at most one extension of a focs for any given prolongation \cite[Proposition 6.1]{MMF2}.

\subsection{Complex-Compatible Bimodule Maps}

In this subsection we introduce a natural compatibility condition between the bimodule map $\sigma$ and a focs $\Omega^1(B) = \Omega^{(1,0)} \oplus \Omega^{(0,1)}$. This requires us to consider the following decomposition 
\begin{align} \label{eqn:degreetwodecomposition}
\Omega^1(B) \otimes_B \Omega^1(B) = \Omega^{\otimes (2,0)} \oplus \Omega^{\otimes (1,1)} \oplus \Omega^{\otimes (0,2)},
\end{align}
where we have denoted
\begin{align*}
\Omega^{\otimes (2,0)} := \Omega^{(1,0)} \otimes_B \Omega^{(1,0)}, & & \Omega^{\otimes (0,2)} := \Omega^{(0,1)} \otimes_B \Omega^{(0,1)},
\end{align*}
\begin{align*}
\Omega^{\otimes (1,1)} :=  \big(\Omega^{(1,0)} \otimes_B \Omega^{(0,1)}\big) \oplus \big(\Omega^{(0,1)} \otimes_B \Omega^{(1,0)}\big).
\end{align*}
As observed in \cite[\textsection 6.1]{MMF2}, a focs on $\Omega^1(B)$ extends to an almost complex structure on the maximal prolongation of $\Omega^1(B)$ if and only if the space of degree two relations $N^{(2)}$ is homogeneous with respect to the decomposition given in \eqref{eqn:degreetwodecomposition}.

\begin{defn}
We say that $\sigma$ is \emph{weakly compatible} with a focs for $\Omega^1(B)$ if the decomposition \eqref{eqn:degreetwodecomposition} is invariant under the action of $\sigma$, that is, if
\begin{align*}
\sigma\big(\Omega^{\otimes (a,b)}\big) \subseteq \Omega^{\otimes (a,b)}, & & \text{ for } (a,b) = (2,0),(1,1),(0,2).
\end{align*}
We say that $\sigma$ is \emph{weakly compatible} with an almost complex structure $\Omega^{(\bullet,\bullet)}$ if it is weakly compatible with the induced focs $\Omega^{1}(B)$.
\end{defn}

As a first easy result about weakly compatible  bimodule maps we have the following proposition.

\begin{prop}\label{prop:FOACS.Extension} If the maximal prolongation $\Omega^{\bullet}(B)$ of $\Omega^1(B)$ is $\sigma$-exterior with respect to $\nabla$, and $\sigma$ is weakly compatible with a focs $\Omega^1(B) = \Omega^{(1,0)} \oplus \Omega^{(0,1)}$, then the focs extends to an almost complex structure $\Omega^{(\bullet,\bullet)}$ on $\Omega^{\bullet}(B)$.
\end{prop}
\begin{proof}
We need to show that $N^{(2)}$ is homogeneous with respect to the decomposition \eqref{eqn:degreetwodecomposition}. Since we are assuming that $\Omega^{\bullet}(B)$ is $\sigma$-exterior, every element of $N^{(2)}$ is of the form $\sigma(X) + X$, for $X$ an element of $\Omega^1(B) \otimes_B \Omega^1(B)$. Now we have the decomposition 
\begin{align*}
X = X_1 + X_2 + X_3, & & \textrm{ where } X_1 \in \Omega^{\otimes(2,0)}, \, X_2 \in  \Omega^{\otimes(1,1)}, \, X_3 \in \Omega^{\otimes(0,2)}.
\end{align*}
Thus the relation $\sigma(X) + X$ is of the form 
$$
\sigma(X) + X  = (\sigma(X_1) + X_1) + (\sigma(X_2) + X_2) + (\sigma(X_3) + X_3). 
$$
Now since $\sigma$ is weakly compatible with the focs, each of the three bracketed summands is homogeneous with respect to the decomposition \eqref{eqn:degreetwodecomposition}. Thus we see that the focs extends.
\end{proof}

We now consider a stronger form of compatibility of $\sigma$ with a focs.

\begin{defn}
We say that $\sigma$ is \emph{compatible} with a focs $\Omega^1(B) = \Omega^{(1,0)} \oplus \Omega^{(0,1)}$ if it is weakly compatible with the focs and moreover if it holds that
\begin{align*}
\sigma\Big(\Omega^{(1,0)} \otimes_B \Omega^{(0,1)}\Big) \subseteq \Omega^{(0,1)} \otimes_B \Omega^{(1,0)},  \textrm{  and }  \sigma\Big(\Omega^{(0,1)} \otimes_B \Omega^{(1,0)}\Big) \subseteq \Omega^{(1,0)} \otimes_B \Omega^{(0,1)}.
\end{align*}
We say that $\sigma$ is \emph{compatible} with an almost complex structure $\Omega^{(\bullet,\bullet)}$ if it is compatible with the focs associated to $\Omega^{(\bullet,\bullet)}$ .
\end{defn}

We say that an almost complex structure is \emph{surjective} if the maps
\begin{align*}
\wedge: \Omega^{(1,0)} \otimes_B \Omega^{(0,1)} \to \Omega^{(1,1)}, & & \wedge: \Omega^{(0,1)} \otimes_B \Omega^{(1,0)} \to \Omega^{(1,1)}
\end{align*}
are both surjective. \emph{Injectivity} of an almost complex structure is defined similarly. 

\begin{prop} \label{prop:FOACS.surjective}
Let $\sigma$ be the bimodule map of a bimodule connection whose torsion is a right $B$-module map. Assume that $\sigma$ is compatible with a focs and let  $\Omega^{(\bullet,\bullet)}$ be its extension to an almost complex structure. Then $\Omega^{(\bullet,\bullet)}$ is surjective.
\end{prop}
\begin{proof}
An arbitrary element $X$ of $\Omega^{(1,1)}$ is of the form 
$$
X := \sum_i \omega_i \wedge \nu_i + \sum_i \mu_i \wedge \theta_i,
$$
where $\sum_i \omega_i \otimes \nu_i$ is an element of $\Omega^{(1,0)} \otimes_B \Omega^{(0,1)}$, and where $\sum_i \mu_i \otimes \theta_i$ is an element of $\Omega^{(0,1)} \otimes_B \Omega^{(1,0)}$. We note that $X$ is an element of the image of the map 
\begin{align} \label{eqn:restricted.map}
\wedge:\Omega^{(1,0)} \otimes_B \Omega^{(0,1)} \to \Omega^{(1,1)},
\end{align}
if and only if $\sum_i \mu_i \wedge \theta_i$ is also contained in the image of \eqref{eqn:restricted.map}. Now it follows from \eqref{eqn:wedge.sigma} that 
$$
\sum_i \mu_i \wedge \theta_i = - \sum_i \wedge \circ \sigma(\mu_i \otimes \theta_i). 
$$
Since by assumption $\sigma$ is compatible with the focs, we have that $\sigma(\mu_i \otimes \theta_i)$ is contained in $\Omega^{(1,0)} \otimes_B \Omega^{(0,1)}$. Thus $\sum_i \mu_i \wedge \theta_i$ is contained in the image of the wedge map. The second surjectivity condition is established similarly, implying that the almost complex structure is surjective.
\end{proof}

\begin{prop} \label{prop:sigma.cross}
If $\Omega^{\bullet}(B)$ is $\sigma$-exterior and the maximal prolongation of a fodc, and $\sigma$ is compatible with an almost complex structure $\Omega^{(\bullet,\bullet)}$, then the following are equivalent:
\begin{enumerate}

\item factorisability of $\Omega^{(\bullet,\bullet)}$,

\item injectivity of  $\Omega^{(\bullet,\bullet)}$,

\item $\sigma^2(X) = X$, for all $X \in \Omega^{\otimes (1,1)}$.

\end{enumerate}
\end{prop}
\begin{proof}
Since any factorisable almost complex structure is injective by definition, it is obvious that \textbf{1. implies 2}. The fact that \textbf{2. implies 1}. follows from the fact that our almost complex structure is automatically surjective, as established in Proposition \ref{prop:FOACS.surjective}. 

Let us next show that \textbf{2. implies 3}. Take an element $X \in \Omega^{(1,0)} \otimes_B \Omega^{(0,1)}$. We know from Proposition \ref{prop:sigma.squared} that $X - \sigma^2(X)$ is an element of $N^{(2)}$. Now since $\sigma$ is compatible with the focs, the element $X - \sigma^2(X)$ must be contained in $\Omega^{(1,0)} \otimes_B \Omega^{(0,1)}$. Injectivity of the wedge map 
$
\wedge: \Omega^{(1,0)} \otimes_B \Omega^{(0,1)} \to \Omega^{(1,1)}
$
means that this element must be zero, which is to say $\sigma^2(X) = X$. An analogous argument shows that $\sigma^2(Y) = Y$, for all $Y  \in \Omega^{(0,1)} \otimes_B \Omega^{(1,0)}$. Thus we have established that 2. implies 3.

We now prove the other implication, that is, that \textbf{3. implies 2}. The multiplication map $\wedge: \Omega^{(1,0)} \otimes_B \Omega^{(0,1)} \to \Omega^{(1,1)}$ is injective if and only if 
\begin{align} \label{eqn:decomp.Factor}
N^{(1,1)} \cap \big(\Omega^{(1,0)} \otimes_B \Omega^{(0,1)}\big) = 0. 
\end{align}
Now since the focs is compatible with $\sigma$, we see that 
\begin{align*}
N^{(2)} =  N^{(2,0)} \oplus N^{(1,1)} \oplus N^{(0,2)},
\end{align*}
where we have denoted
\begin{align*}
N^{(a,b)} := (\sigma + \mathrm{id})(\Omega^{\otimes(a,b)}), & & \textrm{ for } (a,b) = (2,0), (1,1), (0,2).
\end{align*}
Now any element of $N^{(1,1)}$ is of the form 
\begin{align*}
X + Y + \sigma(X) + \sigma(Y), & & \textrm{ for } X \in \Omega^{(1,0)} \otimes_B \Omega^{(0,1)}, \textrm{ and } Y \in \Omega^{(0,1)} \otimes_B \Omega^{(1,0)}.
\end{align*} 
Thus we see that $N^{(1,1)} \cap \big(\Omega^{(1,0)} \otimes_B \Omega^{(0,1)}\big)$ is non-trivial if and only if there exist elements $X$ and a $Y$ such that
\begin{align} \label{eqn:case1}
X + \sigma(Y) \neq 0, & & \textrm{ and } Y + \sigma(X) = 0.
\end{align}
However, if $Y + \sigma(X) = 0$ then operating by $\sigma$ gives us
$$
0 = \sigma(Y + \sigma(X)) = \sigma(Y) + \sigma^2(X) = X + \sigma(Y). 
$$
To avoid contradiction we must conclude that $N^{(1,1)}$ and $\Omega^{(1,0)} \otimes_B \Omega^{(0,1)}$ have trivial intersection. A similar argument shows that $N^{(1,1)}$ and $\Omega^{(0,1)} \otimes_B \Omega^{(1,0)}$ have trivial intersection. Thus we see that the almost complex structure is injective, and hence that 2. and 3. are equivalent.
\end{proof}

We finish this subsection with a short table to summarise the results established so far relating almost complex structures on a $\sigma$-exterior maximal prolongation $\Omega^{\bullet}(B)$.
\[
\begin{array}{|l|l|}
\hline
& \\
\textbf{Condition on } \sigma & \textbf{Complex Structure Property} \\
& \\
\hline
& \\
\text{weakly compatible} & \text{the focs extends} \\
& \\
\hline
& \\
\text{compatible} & \text{$\Omega^{(\bullet,\bullet)}$ is surjective} \\
& \\
\hline
& \\
\rule{0pt}{3ex}\text{compatible}
& \text{$\Omega^{(\bullet,\bullet)}$ is injective } \Longleftrightarrow \sigma^2|_{\Omega^{\otimes(1,1)}} = \mathrm{id} \rule[-2ex]{0pt}{0pt}\\
& \\
\hline
\end{array}
\]

\subsection{Covariant First-Order Complex Structures}

In this subsection we restrict to the case of a covariant $*$-dc $\Omega^{\bullet}(B)$, over a quantum homogeneous space $B \subseteq A$, with tangent space $T$. Moreover, we assume that our bimodule connection $\nabla$ is a covariant connection, and hence that $\sigma$ is a morphism in ${}^A_B\mathrm{mod}_0$. 

A complex structure $\Omega^{(\bullet,\bullet)}$ for $\Omega^{\bullet}$ is said to be \emph{covariant} if the $\mathbb{Z}^2_{\geq 0}$-decomposition of the dc is a decomposition in the category ${}^A_B\mathrm{mod}_0$. A \emph{covariant} focs is defined similarly. We begin with a lemma, which describes direct sum decompositions of fodc in the dual setting of quantum tangent spaces. For this we need the evident notion of a \emph{decomposition} of a tangent space $T \subseteq B^{\circ}$, which is a choice of two tangent spaces $T_1,T_2 \subseteq T$ such that $T = T_1 \oplus T_2$, with respect to the evident definition of direct sum for tangent spaces.

\begin{lem} \label{lem:tangent.FOCS}
We have a bijective correspondence between the following structures:
\begin{enumerate}

\item decompositions of $\Omega^1(B) \simeq \Gamma_1 \oplus \Gamma_2$ in the category ${}^A_B\mathrm{mod}_0$,

\item decompositions $V^1 = V_1 \oplus V_2$ in the category ${}^{\pi_B}\mathrm{mod}_0$,

\item tangent space decompositions $T = T_1 \oplus T_2$.

\end{enumerate}
\end{lem}
\begin{proof}
Takeuchi's equivalence implies that decompositions of $\Omega^1(B)$ in the category ${}^A_B\mathrm{mod}_0$ correspond to decompositions of $V^1$ in the category ${}^{\pi_B}\mathrm{mod}_0$. Since $T$ and $V^1$ are dually paired, a vector space decomposition $V^1 = V_1 \oplus V_2$ has an associated vector space decomposition $T = T_1 \oplus T_2$, where explicitly $X \in T_1$ if, for all $v \in V_2$, we have $\langle X,v\rangle = 0$, and $Y \in T_2$ if, for all $w \in V_1$, we have $\langle Y,w\rangle = 0$. We need to check that this is a decomposition of quantum tangent spaces. Consider first elements $l \in L$, $X \in T_1$, and $[b] \in V_2$, and note that
\begin{align*}
\langle lX, [b] \rangle = \langle lX, b \rangle = \langle l, b_{(1)} \rangle \langle X, b_{(2)} \rangle = \langle l, \pi_B(b_{(1)}) \rangle \langle X, [b_{(2)}] \rangle = 0
\end{align*} 
since $V_2$ is a left $\pi_B(A)$-subcomodule of $V^1$. Thus we see that $T_1$ is a left $L$-module. Next we note that  $\Delta(X) \in (T_1 \oplus \mathbb{C}) \otimes B^{\circ}$ if and only if $\langle X_{(1)},b \rangle X_{(2)}$ is equal to zero, for all $[b] \in V_2$. Pairing $\langle X_{(1)},b \rangle X_{(2)}$ with some $c \in B$, we get the element
\begin{align*}
\langle X_{(1)},b \rangle \langle X_{(2)}, c\rangle = \langle X, bc\rangle = \langle X, [bc]\rangle = \e(c)\langle X, [b]\rangle = 0.
\end{align*}
Non-degeneracy of the pairing between $B^{\circ}$ and $B$ now implies that $\Delta(X) \in (T_1 \oplus \mathbb{C}) \otimes B^{\circ}$, and hence that $T_1$ is a tangent space. The proof that $T_2$ is a tangent space is completely analogous.

In the other direction, let $T = T_1 \oplus T_2$ be a decomposition of tangent spaces, and denote by $V^1 = V_1 \oplus V_2$ the corresponding decomposition of $V^1$. Let us first show that this is a decomposition of left $\pi_B(A)$-comodules. Consider  
\begin{align*}
\Delta_L([b]) = \pi_B(b_{(1)}) \otimes [b_{(2)}], & & \textrm{ for some } [b] \in V_1.
\end{align*}
For any $l \in L$, and $X \in T_2$, we have
\begin{align*}
\langle l,\pi_B(b_{(1)})\rangle \langle X,b_{(2)} \rangle = \langle l,b_{(1)}\rangle \langle X,b_{(2)} \rangle = \langle lX, b\rangle = \langle lX, [b]\rangle = 0, 
\end{align*}
since $lX$ is an element of $T_2$. Thus we have a decomposition of $V^1$ in the category ${}^{\pi_B}\mathrm{mod}_0$, which gives us the equivalence between 2. and 3.
\end{proof}

With this lemma in hand, we can now introduce a number of formulations of extendability of a covariant focs in terms of weak compatibility. Analogous to the decomposition of $\Omega^1(B) \otimes_B \Omega^1(B)$, we have the decomposition 
$$
V^1 \otimes V^1 \simeq V^{\otimes(2,0)} \oplus V^{\otimes(1,1)} \oplus V^{\otimes(0,2)},
$$
and moreover, the decomposition 
$
T \otimes T \simeq T^{\otimes(2,0)} \oplus T^{\otimes(1,1)} \oplus T^{\otimes(0,2)}.
$

\begin{prop} \label{prop:sigma.compt.tangent}
For a choice of covariant focs $\Omega^1(B) = \Omega^{(1,0)} \oplus \Omega^{(0,1)}$, the following are equivalent:

\begin{enumerate}


\item the focs is weakly $\sigma$-compatible,

\item  
$
V^1 \otimes V^1 \simeq V^{\otimes(2,0)} \oplus V^{\otimes(1,1)} \oplus V^{\otimes(0,2)}
$
is invariant under $\sigma$.

\item 
$
T \otimes T \simeq T^{\otimes(2,0)} \oplus T^{\otimes(1,1)} \oplus T^{\otimes(0,2)}
$
is invariant under $\widecheck{\sigma}$.
\end{enumerate}
\end{prop}
\begin{proof}
It follows directly from Takeuchi's equivalence that the decomposition in 2. is $\sigma$-invariant if and only if $\sigma$ is weakly compatible with the focs.

Looking at the third formulation, we observe that Lemma \ref{lem:tangent.FOCS} implies that 
\begin{align*}
\langle V^{\otimes(a,b)}, T^{\otimes(c,d)}\rangle = 0, & & \textrm{ if }  (a,b) \neq (c,d),
\end{align*}
and moreover that the restricted dual pairing
\begin{align*}
V^{\otimes(a,b)} \times T^{\otimes(a,b)} \to \mathbb{C}
\end{align*}
is non-degenerate. Thus $\sigma(V^{\otimes(a,b)}) \subseteq V^{\otimes(a,b)}$ if and only if 
\begin{align*}
\langle \sigma(V^{\otimes(a,b)}) , T^{\otimes(c,d)}\rangle = 0, & & \textrm{ for all } (c,d) \neq (a,b).
\end{align*}
However, by definition of $\widecheck{\sigma}$, we have that 
$$
\langle \sigma(V^{\otimes(a,b)}) , T^{\otimes(c,d)}\rangle = \langle V^{\otimes(a,b)}, \widecheck{\sigma}(T^{\otimes(c,d)})\rangle, ~~~~~ \textrm{ for } (a,b) \neq (c,d),
$$
which is necessarily zero if the decomposition of $T^{\otimes 2}$ is invariant under the action of $\widecheck{\sigma}$. The other direction of the equivalence is proved analogously. 
\end{proof}

\subsection{$L$-Twisted Primitive Elements}

In this subsection we show how $L$-twisted primitive elements can be used to establish $\sigma$-compatibility and as a consequence focs extendability.

\begin{prop} \label{cor:LTP.focs.extend}
Let $T = T^{(1,0)} \oplus T^{(0,1)}$ be the decomposition of $T$ associated to a covariant focs on $\Omega^1(B)$. If $T^{(1,0)}$ and $T^{(0,1)}$ are both generated by $L$-twisted primitive elements, then $\sigma$ is compatible with the focs.
\end{prop}
\begin{proof}
Consider two pairs
$$
(a,b), \, (c,d) \in \big\{(1,0), \, (0,1)\big\}.
$$
By assumption, $T^{(a,b)}$ is linearly spanned by elements of the form $\rho_B(lX)$, for some $L$-twisted primitive element $X \in A^\circ$, and some $l \in L$. For any $Y \in T^{(c,d)}$, Proposition \ref{prop:THEFORMULA} tells us that 
\begin{align*}
\widecheck{\sigma}(\rho_B(lX) \otimes Y) = l_{(1)}K_2K_1^{-1}S(l_{(3)})Y \otimes l_{(2)}\rho_B(X).
\end{align*}
Since we know from Lemma \ref{lem:tangent.FOCS} that $T^{(1,0)}$ and $T^{(0,1)}$ are left $L$-modules, we have 
$$
\widecheck{\sigma}(\rho_B(lX) \otimes Y) \in T^{(c,d)} \otimes T^{(a,b)}.
$$
Thus we see that $\sigma$ is compatible with the focs.
\end{proof}

Proposition \ref{prop:FOACS.Extension} now implies the following corollary.

\begin{cor} \label{cor:L.Twisted.focs.extend}
If $\Omega^{\bullet}(B)$ is $\sigma$-exterior, then the focs extends to an almost complex structure for $\Omega^{\bullet}(B)$.
\end{cor}

\section{The Irreducible Quantum Flag Manifolds $\OO_q(G/L_S)$} \label{section:HK}

In this section, the general framework of the paper is applied to our motivating family of examples: the cominuscule quantum flag manifolds equipped with their Heckenberger–Kolb calculi. The calculi are shown to be $\sigma$-exterior with respect to their unique covariant connection, thus yielding a simple and uniform description of their degree-two relations. Moreover, we give a new direct proof of the existence of the Heckenberger--Kolb almost complex structures, and observe that the bimodule map $\sigma$ squares to the identity on the space of mixed forms $\Omega^{\otimes (1,1)}$. 

\subsection{Drinfeld--Jimbo Quantum Groups}

Let $\mathfrak{g}$ be a finite-dimensional complex simple Lie algebra of rank $n$. We fix a Cartan subalgebra $\mathfrak{h}$ with corresponding root system $\Delta \subseteq \mathfrak{h}^{\ast}$, where $\mathfrak{h}^{\ast}$ denotes the
linear dual of $\mathfrak{h}$. Choose a set of simple roots $\Pi= \{ \alpha_1, \alpha_2, \cdots,\alpha_n\}$ and denote the associated set of positive roots by $\Delta^+$. We write $\prec$ for the standard partial order on $\Delta$, that is, $\beta \prec \gamma$ if $\gamma - \beta$ is a non-zero sum of positive roots. We denote by $(-,-)$ the symmetric bilinear form induced on $\mathfrak{h}^{\ast}$ by the Killing form of $\mathfrak{g}$, normalised
so that any shortest simple root $\alpha_i$ satisfies $(\alpha_i, \alpha_i)=2$. The coroot $\alpha_i^\vee$ of a simple root $\alpha_i$ is defined by
\begin{align*}
    \alpha^\vee_i:=\frac{ \alpha_i}{d_i}= 2\frac{\alpha_i}{(\alpha_i, \alpha_i)}, \quad\text{where }  d_i=\frac{(\alpha_i, \alpha_i)}{2}.
\end{align*}
The Cartan matrix $\mathcal{A} = (a_{i j})_{i j}$ of $\mathfrak{g}$ is the $(n \times n)$-matrix defined by 
$
a_{ij}:= (\alpha^\vee_i, \alpha_j).
$
Let $\{\varpi_1,\cdots, \varpi_n\}$ denote the corresponding set of fundamental weights of $\mathfrak{g}$, which is to say, the dual basis of the coroots. For any root $\alpha$ we denote its height by $\mathrm{ht}(\alpha)$, and we denote the highest root of $\Delta$ by $\theta$. 

Let $q \in  \mathbb{R}_{> 0} \backslash \{1\}$ and denote $q_i:=q^{d_i}$. The quantised universal enveloping algebra $U_q(\mathfrak{g})$ is the noncommutative associative algebra generated by the
elements $E_i, F_i, K_i$, and $K_i^{-1}$, for $i = 1,\cdots,n$, 
subject to the relations 
\begin{align*}
 K_iE_j = q_i^{a_{ij}} E_j K_i, ~~~~ K_iF_j= q_i^{-a_{ij}} F_j K_i, ~~~~ K_i K_j = K_j K_i, ~~~~ K_iK_i^{-1} = K_i^{-1}K_i = 1,\\
 E_iF_j - F_jE_i = \d_{ij}\frac{K_i - K\inv_{i}}{q_i-q_i^{-1}}, ~~~~~~~~~~~~~~~~~~~~~~~~~~~~~~~~~~~~~~~~~
\end{align*}
and the quantum Serre relations 
\begin{align*}
 \sum\nolimits_{s=0}^{1-a_{ij}} (-1)^s \begin{bmatrix} 1 - a_{ij} \\ s \end{bmatrix}_{q_i}
 E_i^{1-a_{ij}-s} E_j E_i^s = 0,\quad \textrm{ for } i\neq j,\\
 \sum\nolimits_{s=0}^{1-a_{ij}} (-1)^s \begin{bmatrix} 1 - a_{ij} \\ s \end{bmatrix}_{q_i}
 F_i^{1-a_{ij}-s} F_j F_i^s = 0,\quad \textrm{ for } i\neq j,
\end{align*}
where we have used the $q$-binomial coefficient
\begin{align*}
\begin{bmatrix} n \\ r \end{bmatrix}_q := \frac{[n]_q!}{[r]_q! \, [n-r]_q!},
\end{align*}
defined in terms of the quantum integer
\begin{align*}
[m]_{q} := \frac{q^{m} - q^{-m}}{q-q^{-1}}.
\end{align*}
A Hopf algebra structure is defined on $U_q(\mathfrak{g})$ by
\begin{gather*}
  \Delta E_i = E_i\otimes K_i + 1\otimes E_i,\quad
  \Delta F_i = F_i\otimes 1 + K^{-1}_i \otimes F_i,\quad
  \Delta K^{\pm}_i = K^{\pm}_i \otimes K^{\pm}_i,\\
  S(E_i) = - E_iK^{-1}_i,\quad
  S(F_i) = - K_iF_i,\quad
  S(K^{\pm}_i) = K^{\mp}_i,\\
  \e (K_i) = 1,\quad \e (E_i)=\e(F_i)=0.
\end{gather*}
A Hopf $\ast$-algebra structure, called the \emph{compact real form} of $U_q (\mathfrak{g})$, is defined by
\begin{align*}
    K^{\ast}_i = K_i, \quad E^{\ast}_i = K_iF_i, \quad F^{\ast}_i = E_i K^{-1}_i.
\end{align*}
Let $\mathcal{P}\subset\fh^\ast$ be the weight lattice of $\mathfrak{g}$, and $\mathcal{P}_+\subset\mathcal{P}$ its set of dominant integral weights.
For each $\mu \in \mathcal{P}_+$ there exists an irreducible finite-dimensional $U_q (\mathfrak{g})$-module~$V_{\mu}$, uniquely defined by the existence of a \emph{highest weight vector} $v_{\mu}\in  V_{\mu}$ satisfying
\begin{align*}
    E_i \triangleright v_{\mu}=0, \quad K_i\triangleright v_{\mu}= q_i^{(\mu,\alpha^\vee_i)}v_{\mu}, \quad \text{for all $i=1, \cdots , n$}.
\end{align*}
Moreover, $v_{\mu}$ is the unique such element up to scalar multiple. We call any finite direct
sum of such $U_q (\mathfrak{g})$-modules a \emph{type-$1$ representation}.  For further details on Drinfeld--Jimbo quantised enveloping algebras, we refer the reader to the standard text~\cite{KSLeabh}.

\subsection{Quantum Coordinate Algebras and Cominuscule Quantum Flag Manifolds} 

In this subsection we recall some necessary material about quantum coordinate algebras.
Let $V$ be a finite-dimensional left $U_q(\mathfrak{g})$-module, $v \in V$, and $f \in V^*$, the $\mathbb{C}$-linear dual of $V$, endowed with its right \mbox{$U_q(\mathfrak{g})$-module} structure.

Consider the function $c^{\tiny{V}}_{f,v}:U_q(\mathfrak{g}) \to \mathbb{C}$ defined by $c^{\tiny{V}}_{f,v}(X) := f\big(X \triangleright v\big)$. The \emph{space of matrix coefficients} of $V$ is the subspace
\begin{align*}
C(V) := \mathrm{span}_{\mathbb{C}}\!\left\{ c^{\tiny{V}}_{f,v} \,| \, v \in V,  f \in V^*\right\} \subseteq U_q(\mathfrak{g})^*.
\end{align*}

Let $U_q(\mathfrak{g})^\circ$ denote the Hopf dual of $U_q(\mathfrak{g})$. It is easily checked that a Hopf subalgebra of $U_q(\mathfrak{g})^{\circ}$ is given by
\begin{equation}\label{eq:PeterWeyl}
\mathcal{O}_q(G) := \bigoplus_{\mu \in \mathcal{P}^+} C(V_{\mu}).
\end{equation}
We call $\mathcal{O}_q(G)$ the {\em quantum coordinate algebra of~$G$}, where~$G$ is the compact, connected, simply-connected, simple Lie group having~$\mathfrak{g}$ as its complexified Lie algebra.
We note that $\mathcal{O}_q(G)$ is a cosemisimple Hopf algebra by construction. We denote the dual pairing between $\OO_q(G)$ and $U_q(\mathfrak{g})$ by $\langle -,- \rangle$.

For any subset $S \subseteq \Pi$ of simple roots, we denote 
\begin{align*}
\Delta_S^{+} := \mathbb{Z}S \cap \Delta^{+}, && \overline{\Delta^+_S} := \Delta^+ \setminus \Delta_S^+.
\end{align*}
For any such $S$, consider the Hopf $*$-subalgebras
\begin{align*}
U_q(\mathfrak{l}_S) := \big< K_i^{\pm 1}, E_j, F_j \,|\, i = 1, \ldots, n; \alpha_j \in S \big>,\\
U_q(\mathfrak{l}^{\mathrm{\,s}}_S) := \big< K_j^{\pm 1}, E_j, F_j \,|\, \alpha_j \in S \big>.~~~~~~~~~~~~~~~~\,
\end{align*} 
The category of type-$1$ finite-dimensional $U_q(\mathfrak{l}_S)$-modules is  semisimple, as can be concluded from semisimplicity of the category of type-$1$ finite-dimensional $U_q(\mathfrak{g})$-modules. 
The {\em quantum flag manifold associated} to $S$ is the space of invariants
\begin{align*}
\mathcal{O}_q\big(G/L_S\big) := & \, {}^{U_q(\mathfrak{l}_S)}\mathcal{O}_q(G) \\
= & \Big\{ b \in \OO_q(G) \,\big|\, X \triangleright b = b_{(1)} \langle X, b_{(2)}\rangle = \e(X)b, \textrm{ for all } X \in U_q(\mathfrak{l}_S)\Big\}.
\end{align*} 
It follows from semisimplicity of the category of type-$1$ finite-dimensional  $U_q(\mathfrak{l}_S)$-modules that $\mathcal{O}_q(G/L_S)$ is a quantum homogeneous $\OO_q(G)$-space (see \cite[\textsection 4.3]{HVBQFM} for details).
Since the category of type-$1$ finite-dimensional $U_q(\mathfrak{l}_S)$-modules is semisimple, ${}^A_B\mathrm{mod}_0$ is necessarily semisimple.

A quantum flag manifold is called \emph{cominuscule}, or \emph{irreducible}, if the defining subset $S$ of simple roots is of the form
$$
S = \{\alpha_1, \dots, \alpha_n \} \setminus \{\alpha_x\}
$$
where $\alpha_x$ is a \emph{cominuscule root}, that is to say, it has coefficient $1$ in the expansion of the highest root of $\mathfrak{g}$. See Appendix \ref{section:Appendix} for a graphical representation of these roots.

\subsection{The Heckenberger--Kolb Calculus}

We present the calculus in two steps, beginning with the Heckenberger--Kolb classification of first-order differential calculi over the cominuscule quantum flag manifolds $\OO_q(G/L_S)$. Then we discuss the maximal prolongation of the direct sum of the two calculi identified. 

\begin{thm} \cite[\textsection 2]{HK} \label{HKClass}
There exist exactly two non-isomorphic finite-dimensional irreducible left-covariant first-order differential calculi over $\mathcal{O}_q(G/L_S)$. 
\end{thm}

We denote these two calculi by $\Omega^{(1,0)}$ and $\Omega^{(0,1)}$. We call the direct sum 
$$
\Omega^1_q(G/L_S) := \Omega^{(1,0)} \oplus \Omega^{(0,1)}
$$
the {\em Heckenberger--Kolb first-order differential calculus} of $\OO_q(G/L_S)$.

For $\alpha_x$ the cominuscule root corresponding to a cominuscule quantum flag manifold, tangent spaces for $\OO_q(G/L_S)$ are given by
\begin{align*}
T^{(1,0)} := U_q(\mathfrak{l}_S)F_x, & & T^{(0,1)} := U_q(\mathfrak{l}_S)E_x,
\end{align*} 
where to lighten notation we have written $F_x$, and $E_x$, for $\rho_{\OO_q(G/L_S)}(F_x)$ and $\rho_{\OO_q(G/L_S)}(E_x)$ respectively. Moreover, they are non-isomorphic and irreducible as $U_q(\mathfrak{l}_S)$-modules, meaning that the associated fodc are non-isomorphic and irreducible. Hence by Theorem \ref{HKClass} they are the tangent spaces of $\Omega^{(1,0)}$ and $\Omega^{(0,1)}$. Thus $T := T^{(1,0)} \oplus T^{(0,1)}$ is the tangent space of the Heckenberger--Kolb calculus $\Omega^1_q(G/L_S)$. 

The tangent space $T$ admits an evident lift to a tangent space for $\OO_q(G)$ (see \cite{ROB.AB} for details) which by abuse of notation we again denote by $T$. Noting that $T^* = T$, we see that the corresponding calculus on $\OO_q(G)$ is a $*$-calculus, see \cite[\textsection 2.4]{ACROBJR} for more details. Hence the restriction to $\OO_q(G/L_S)$, which is to say the Heckenberger--Kolb fodc, is a $*$-calculus. This result was originally proved by Matassa \cite{MarcoConj} using a different method. 

From this it is easily concluded that the decomposition $\Omega^1_q(G/L_S) = \Omega^{(1,0)} \oplus \, \Omega^{(0,1)}$ is a covariant focs for the Heckenberger--Kolb fodc. (In the $q=1$ limit, this decomposition reduces to the complex manifold decomposition of the space of complexified $1$-forms on $G/L_S$.) We denote the maximal prolongation of $\Omega^1_q(G/L_S)$ by $\Omega^{\bullet}_q(G/L_S)$, and call it the \emph{Heckenberger--Kolb calculus}.

\subsection{Torsion-Free Bimodule Connections}

In this subsection we recall some necessary information about covariant connections for the Heckenberger--Kolb calculi. Firstly, we recall from \cite[Theorem 4.5]{HVBQFM} that for each cominuscule quantum flag manifold $\O_q(G/L_S)$, its Heckenberger--Kolb calculus $\Omega^1_q(G/L_S)$ admits a unique covariant connection 
$$
\nabla: \Omega^1_q(G/L_S) \to \Omega^1_q(G/L_S) \otimes_{\OO_q(G/L_S)} \Omega^1_q(G/L_S).
$$
As observed in \cite[Remark 4.8]{HVBQFM}, $\nabla$ is also torsion-free. (See \cite[Example 4.3]{ACROBJR} for a different perspective on the proof of torsion-freeness.) Moreover, it was established in \cite[Corollary 6.7]{JBBGAKROB}, in the general setting of complex structures and Chern connections, that $\nabla$ is a bimodule connection. Thus, it follows from the discussions of \textsection \ref{subsection:bimod.conn.torsion} that 
$$
\mathrm{im}(\sigma + \id) \subseteq \OO_q(G/L_S) \nabla(\exd \OO_q(G/L_S)) = N^{(2)}.
$$

\begin{rem}
It was shown in \cite[\textsection 6.4]{JBBGAKROB} that $\Omega^1_q(G/L_S)$ admits a covariant symmetric metric $g$ in the sense of Beggs and Majid \cite[\textsection 1.3]{BM.Leabh}. With respect to $g$, the connection $\nabla$ is a Levi-Civita connection, again in the sense of Beggs and Majid \cite[\textsection 8.1]{BM.Leabh}.
\end{rem}

\subsection{Applications to the Heckenberger--Kolb Calculi} \label{subsection:HKisSigma.Ext}

In this subsection we apply the general results of the paper to the Heckenberger--Kolb calculi. We start with the following easy proposition, which recovers the well-known fact \cite[\textsection 3.2]{HKdR} that the Heckenberger--Kolb calculus is an object in the subcategory ${}^{~~~~~\OO_q(G)}_{\OO_q(G/L_S)}\mathrm{mod}_0$. 

\begin{prop} \label{prop:qflag.primitive.tangent}
The tangent space $T$ is primitive, and hence $\Omega^1_q(G/L_S)$ is an object in the subcategory ${}^{~~~~~\OO_q(G)}_{\OO_q(G/L_S)}\mathrm{mod}_0$.
\end{prop}
\begin{proof}
Since $T$ is generated by the $L$-twisted primitive elements $F_x$ and $E_x$, it follows from Proposition \ref{prop:L.twisted.implies.primitive} that $T$ is a primitive tangent space and hence that $\Omega^1_q(G/L_S)$ is an object in the subcategory.
\end{proof}

The fact that the calculus lives in the subcategory has the following straightforward corollaries. 

\begin{cor}
The bimodule map $\sigma$ of $\nabla$ is $*$-invertible.
\end{cor}
\begin{proof}
Since the calculus is contained in the subcategory ${}^{~~~~~\OO_q(G)}_{\OO_q(G/L_S)}\mathrm{mod}_0$, and moreover is a $*$-fodc, it follows from Proposition \ref{prop:cov.star.invt} that $\sigma$ is $*$-invertible.
\end{proof}

\begin{cor}
The connection $\nabla$ is extendable.
\end{cor}
\begin{proof}
This follows directly from Proposition \ref{prop:qflag.primitive.tangent} and Proposition \ref{prop:extendable}.
\end{proof}

We finish with the following theorem which provides a family of noncommutative geometric Yang--Baxter solutions, one for each cominuscule root of a complex simple Lie algebra $\mathfrak{g}$.

\begin{thm}
For each cominuscule quantum flag manifold $\OO_q(G/L_S)$, the bimodule map $\sigma$ satisfies the Yang--Baxter equation.
\end{thm}
\begin{proof}
Since $T$ is evidently generated by the $L$-twisted primitive elements $F_x$ and $E_x$, the fact that $\sigma$ satisfies the Yang--Baxter equation follows directly from Corollary \ref{cor.THE.COR}.
\end{proof}


\subsection{Heckenberger--Kolb as a $\sigma$-Exterior Calculus} \label{subsection:HKsigmeExt}

We now prove one of the main results of the paper, namely that the Heckenberger--Kolb calculus is $\sigma$-exterior with respect to its unique covariant connection $\nabla$. This follows from the Heckenberger--Kolb classification of fodc and Corollary \ref{cor:the.projection}.

\begin{thm} \label{thm.THE.HK.THM}
The Heckenberger--Kolb calculi are $\sigma$-exterior. Hence the space of degree two relations $N^{(2)}$ is the $\OO_q(G/L_S)$-bimodule generated by 
\begin{align*}
\Big\{ \exd b \otimes \exd c + \exd \big(b_{(1)}cS(b_{(2)})\big) \otimes \exd b_{(3)} ~\big|~ b,c \in \OO_q(G/L_S)\Big\}.
\end{align*}
\end{thm}
\begin{proof}
Corollary \ref{cor:the.projection} tells us that the map 
$$
\mathfrak{d} = \mathrm{proj} \circ \nabla \circ \exd: \OO_q(G/L_S) \to \Omega^{1}_q(G/L_S) \otimes_{\OO_q(G/L_S)} \Omega^{1}_q(G/L_S)/\mathrm{im}(\sigma + \mathrm{id})
$$
is a derivation. Thus the pair 
$$
\Big(\Gamma := \OO_q(G/L_S)\mathfrak{d}\big(\OO_q(G/L_S)\big), ~ \mathfrak{d}\Big)
$$
is a fodc. If $\Gamma$ is the zero calculus, then  
$$
\nabla \circ \exd(\OO_q(G/L_S)) \subseteq  \mathrm{im}(\sigma + \mathrm{id}),
$$
and so, $\Omega^{\bullet}_q(G/L_S)$ is clearly $\sigma$-exterior with respect to $\nabla$. 

Let us now assume that $\Gamma$ is a non-zero calculus. It is clear that $(\Gamma,\mathfrak{d})$ is  left $A$-covariant and moreover that it is an object in ${}^{~~~~\mathcal{O}_q(G)}_{\OO_q(G/L_S)}\mathrm{mod}_0$. Consider the decomposition $\Gamma = \bigoplus_{i}^m \Gamma_i$ into non-zero irreducible subobjects in ${}^{~~~~\mathcal{O}_q(G)}_{\OO_q(G/L_S)}\mathrm{mod}_0$. Since the canonical projection
$$
\mathrm{proj}_i: \Gamma \to \Gamma_i,
$$
is a morphism, a finite-dimensional irreducible left $\OO_q(G)$-covariant fodc is given by each pair
$$
\Big(\Gamma_i,\mathfrak{d}_i := \mathrm{proj}_i \circ \mathfrak{d}\Big).
$$
Thus each $(\Gamma_i,\mathfrak{d}_i)$ is isomorphic to $\Omega^{(1,0)}$ or $\Omega^{(0,1)}$. Thus we have an embedding of $\Omega^{(1,0)}$ or $\Omega^{(0,1)}$ into $\Omega^{\otimes 2}_q(G/L_S)/\mathrm{im}(\sigma + \mathrm{id})$. Since the category ${}^{~~~~\mathcal{O}_q(G)}_{\OO_q(G/L_S)}\mathrm{mod}_0$ is semisimple, the canonical projection map 
\begin{align} \label{eqn:splitting}
\Omega^1_q(G/L_S) \otimes_{\OO_q(G/L_S)} \Omega^1_q(G/L_S) \to \Omega^1_q(G/L_S) \otimes_{\OO_q(G/L_S)} \Omega^1_q(G/L_S) /\mathrm{im}(\sigma + \mathrm{id})
\end{align}
splits. This implies the existence of an embedding of $\Omega^{(1,0)}$ or $\Omega^{(0,1)}$ into the left hand side of \eqref{eqn:splitting}. However, as observed in \cite[Lemma 6.8]{JBBGAKROB}, no such embedding exists. Thus we are forced to conclude that $\Gamma$ is zero and hence that the Heckenberger--Kolb calculus is $\sigma$-exterior with respect to $\nabla$.
\end{proof}

The unique covariant focs on $\Omega^1_q(G/L_S)$ is well known to extend to a complex structure on $\Omega^{\bullet}_q(G/L_S)$, see for example \cite[Proposition 4.4]{PLMIQF}. This essentially follows from the construction of Heckenberger and Kolb \cite{HKdR}. In the following corollary, we show that extendability can also be deduced directly from the tangent space description of the calculus.

\begin{cor}
It holds that the focs 
$$
\Omega^1_q(G/L_S) =  \Omega^{(1,0)} \oplus \Omega^{(0,1)}
$$
extends to an almost complex structure on the maximal prolongation $\Omega^{\bullet}_q(G/L_S)$.
\end{cor}
\begin{proof}
The tangent spaces $T^{(1,0)}$ and $T^{(0,1)}$ are generated by $L$-twisted primitive elements. Thus extendability follows from Theorem \ref{thm.THE.HK.THM}, Proposition \ref{cor:LTP.focs.extend}, and Corollary  \ref{cor:L.Twisted.focs.extend}.
\end{proof}

It follows from the Heckenberger--Kolb construction of the maximal prolongation of $\Omega^1_q(G/L_S)$ that it is factorisable \cite[\textsection 3.3.4]{HKdR}. The result below now follows immediately from Proposition \ref{prop:sigma.cross}.

\begin{cor}
It holds that \,
$
\sigma^2|_{\Omega^{\otimes(1,1)}} = \mathrm{id}.
$
\end{cor}


\section{The Spectrum of $\sigma$ and its Relationship to Other Generalised Flip Maps} \label{section:spectrum}

In this section, we explicitly calculate the spectrum of the bimodule map $\sigma$. This allows us to compare $\sigma$ to other generalised flip maps in the literature. In particular, we show that the restriction of $\sigma$ to the purely holomorphic and anti-holomorphic forms is a rescaling of the quasitriangular braiding $\sigma_R$ of the category of $U_q(\mathfrak{l}_S^s)$-modules. Moreover, we discuss the relationship to Drinfeld's normalisation of $\sigma_R$, and to the Nichols algebra braiding of the quantum Grassmannians.

Note that we actually calculate the eigenvalues of the dualised map $\widecheck{\sigma}$. However, since the spectrum of the two operators clearly coincides, this is equivalent.

\subsection{Some General Eigenvalue Calculations} \label{subsection.general.eigen}

In this subsection we produce two families of eigenvectors for a general cominuscule quantum flag manifold. Moreover, we calculate their explicit eigenvalues. 

The first family of eigenvectors consists of the lowest weight vectors of $T^{(0,1)} \otimes T^{(0,1)}$, given by the tensor product of the lowest weight vector of $T^{(0,1)}$ with itself. 

\begin{lem}
For a cominuscule root $\alpha_x$, with Heckenberger--Kolb anti-holomorphic tangent space $T^{(0,1)} = U_q(\mathfrak{l}_S)E_x$, a lowest weight vector of $T^{(0,1)} \otimes T^{(0,1)}$  is given by 
$
E_x \otimes E_x.
$
\end{lem}
\begin{proof}
Since $F_iE_x = E_xF_i$, for all $i \neq x$, we see that $E_x$ is a lowest weight vector of $T^{(0,1)}$. Hence it follows from the coproduct formula for $F_i$, that $E_x \otimes E_x$ is a lowest weight vector of $T^{(0,1)} \otimes T^{(0,1)}$. 
\end{proof}

\begin{prop}
The element $E_x \otimes E_x$ is an eigenvector of $\widecheck{\sigma}$, explicitly
$$
\widecheck{\sigma}(E_x \otimes E_x) = q^{-(\alpha_x,\alpha_x)}E_x \otimes E_x.
$$
\end{prop}
\begin{proof}
The action of $\widecheck{\sigma}$ on $E_x \otimes E_x$ is given by 
\begin{align}
\widecheck{\sigma}(E_x \otimes E_x) = S(K_x)E_x \otimes E_x = K_x^{-1} E_x \otimes E_x = q^{-(\alpha_x,\alpha_x)} E_x \otimes E_x,
\end{align}
giving the claimed identity.
\end{proof}

The second family of eigenvectors is again given by lowest weight vectors. Their form varies somewhat across the series, but they have a clearly similar form.
    
\begin{lem} \label{lem:Valtice}
Let $\alpha_x$ be a cominuscule root, with Heckenberger--Kolb anti-holomorphic tangent space $T^{(0,1)} = U_q(\mathfrak{l}_S)E_x$.
\begin{enumerate}

\item   If $\alpha_x = \alpha_1$, then a lowest weight vector of $T^{(0,1)} \otimes T^{(0,1)}$ is given by 
$$
E_2E_1 \otimes E_1 - q^{(\alpha_1,\alpha_2)}E_1 \otimes E_2E_1.
$$

\item  If $\alpha_x = \alpha_n$, then a lowest weight vector of $T^{(0,1)} \otimes T^{(0,1)}$ is given by 
$$
E_{m}E_n \otimes E_n - q^{(\alpha_{m},\alpha_{n})}E_n \otimes E_{m}E_n,
$$
where $\alpha_m$ is the unique simple root \emph{adjacent} to $\alpha_n$, that is $(\alpha_m,\alpha_n) \neq 0$.

\item  For the quantum Grassmannian case $\OO_q(\mathrm{Gr}_{n+1,r})$, taking $\alpha_x = \alpha_r$, for $r \neq 1,n$, two lowest weight vectors of $T^{(0,1)} \otimes T^{(0,1)}$ are given by
\begin{align*}
E_{r + 1}E_r \otimes E_r - q^{(\alpha_{r+1},\alpha_{r})}E_r \otimes E_{r + 1}E_r, \\
E_{r-1}E_r \otimes E_r - q^{(\alpha_{r-1},\alpha_r)}E_r \otimes E_{r - 1}E_r.
\end{align*}

\end{enumerate}
\end{lem}
\begin{proof}
 We demonstrate the proof for the first case, the proof for the other two cases being completely analogous. Note first that, for any $j \neq 2$, it is evident that 
 $$F_{j}(E_2E_1 \otimes E_1 - q^{(\alpha_1, \alpha_2)}E_1 \otimes E_2E_1)=0.$$
We move next to the case of $F_2$
 \begin{align*}
F_2(E_2E_1 \otimes E_1 - q^{(\alpha_1, \alpha_2)}E_1 \otimes E_2E_1) = F_2E_2E_1\otimes E_1 -q^{(\alpha_1, \alpha_2)} K_{2}^{-1}E_{1} \otimes F_2E_2E_1.
\end{align*}
It follows from the commutation relations of $U_q(\mathfrak{g})$ that this reduces to  
\begin{align*}
\dfrac{(K_{2}^{-1}-K_{2})E_1}{q_2-q^{-1}_2} \otimes E_{1} - q^{(\alpha_1, \alpha_2)}K_{2}^{-1}E_{1} \otimes \dfrac{(K_{2}^{-1}-K_{2})E_1}{q_{2}-q^{-1}_2}
\end{align*}
which in turn reduces to zero. Thus we have a lowest weight vector.
\end{proof}

\begin{prop}
All lowest weight vectors appearing in Lemma \ref{lem:Valtice} are eigenvectors of $\widecheck{\sigma}$ with eigenvalue $-1$. 
\end{prop}
\begin{proof}
We again prove the statement for the first eigenvector, as the proof for the other three is completely analogous. 

For the first summand of $E_2E_1 \otimes E_1 - q^{(\alpha_1,\alpha_2)}E_1 \otimes E_2E_1$ we see that 
\begin{align*}
\widecheck{\sigma}(E_2E_1 \otimes E_1) = E_{2}K_{1}^{-1}K_{2}^{-1}E_{1} \otimes K_{2}E_{1} + K_{1}^{-1} K_{2}^{-1}E_{1} \otimes E_{2}E_{1} - K_{1}^{-1} E_{2}K_{2}^{-1}E_{1} \otimes E_{1},
\end{align*}
which by the commutation relations of $U_q(\mathfrak{g})$ reduces to 
\begin{align*}
 (q^{-(\alpha_1,\alpha_1)}-q^{-2(\alpha_1, \alpha_2)-(\alpha_1, \alpha_1)}) E_{2}E_{1} \otimes E_{1} + q^{-(\alpha_1, \alpha_2)-(\alpha_1, \alpha_1)} E_{1} \otimes E_{2}E_{1}.
\end{align*}
For the second summand we see that 
\begin{eqnarray*}
\widecheck{\sigma}(E_1 \otimes E_{2}E_1)=K_{1}^{-1}E_{2}E_{1}\otimes E_{1}= q^{-(\alpha_1, \alpha_2)-(\alpha_1, \alpha_1)} E_{2}E_{1} \otimes E_{1}     
\end{eqnarray*}
Collecting these two terms together we arrive at the expression
\begin{align*}
-q^{-2(\alpha_1, \alpha_2)-(\alpha_1, \alpha_1)} (E_{2}E_{1} \otimes E_{1}- q^{(\alpha_1, \alpha_2)}E_1 \otimes E_2E_1).
\end{align*}
For the cominuscule roots occurring in the cases under consideration, one can easily verify that 
$$
-2(\alpha_1, \alpha_2)-(\alpha_1, \alpha_1) = 0.
$$
Thus we have an eigenvector of eigenvalue $-1$ as claimed.
\end{proof}

Since the eigenvectors of Lemma \ref{lem:Valtice} all have eigenvalues equal to $-1$, we see that they give the $q$-deformed anti-symmetric subspace of $T^{(0,1)} \otimes T^{(0,1)}$.

We say that a Yang--Baxter solution $\sigma$ is \emph{unitary} if 
$$
\sigma \circ \sigma^{\dagger} = \sigma^{\dagger} \circ \sigma = \mathrm{id},
$$
where $\sigma^{\dagger}$ denotes the conjugate transpose of $\sigma$ with respect to some choice of basis. 

\begin{cor}
The Yang--Baxter solution $\sigma$ is not unitary.
\end{cor}
\begin{proof}
A unitary operator must have eigenvalues that lie on the unit circle. Since by assumption $q \in \mathbb{R}_{> 0}$, and $(\alpha_x,\alpha_x) \neq 0$, the eigenvalue $q^{-(\alpha_x,\alpha_x)}$ cannot lie on the unit circle. Hence $\sigma$ is not unitary.
\end{proof}

\subsection{The Third Eigenvalue of the Quantum Grassmannians}

In this subsection we calculate the third and final anti-holomorphic eigenvalue of $\widecheck{\sigma}$ for the $A$-series cominuscule quantum flag manifolds, that is, for the quantum Grassmannians $\OO_q(\mathrm{Gr}_{n+1,r})$, for $r \neq 1, n$. This of course excludes the case of the quantum projective space $\OO_q(\mathbb{CP}^n)$, which has only two eigenvalues, namely those already determined in \textsection \ref{subsection.general.eigen}. We begin by determining a fourth and final lowest weight vector of $T^{(0,1)} \otimes T^{(0,1)}$.

\begin{lem} \label{lem:Grass.LW}
A lowest weight vector of $T^{(0,1)} \otimes T^{(0,1)}$ is given by 
$$
E_r\otimes E_{r-1}E_{r+1}E_r + q^{2} E_{r-1}E_{r+1}E_r \otimes E_r -q E_{r-1}E_r \otimes E_{r+1}E_r -q E_{r+1}E_r \otimes E_{r-1}E_r. 
$$
\end{lem}
\begin{proof}
Let us denote our proposed lowest weight vector by $v$. We note first that, for any $y\notin \{r+1, r-1, r\}$, the commutation relations of $U_q(\mathfrak{sl}_{n+1})$ directly imply that the action of $F_{y}$ on $v$ is zero. Now, for $y = r-1$, the diagonal action of $U_q(\mathfrak{l}_S)$ gives
\begin{align*}
F_{r-1}v & = K_{r-1}^{-1}E_r \otimes F_{r-1}E_{r-1}E_{r+1}E_r + q^{2} F_{r-1}E_{r-1}E_{r+1}E_r \otimes E_r\\
& -q F_{r-1}E_{r-1}E_r \otimes E_{r+1}E_r -qK_{r-1}^{-1}E_{r+1}E_r \otimes F_{r-1}E_{r-1}E_r.
\end{align*}
The commutation relations of $U_q(\mathfrak{sl}_{n+1})$ allow us to reduce this to 
\begin{align*}
K_{r-1}^{-1}E_r \otimes \Big(\dfrac{K_{r-1}^{-1}-K_{r-1}}{q-q^{-1}}\Big) E_{r+1}E_r + q^{2} \Big(\dfrac{K_{r-1}^{-1}-K_{r-1}}{q-q^{-1}}\Big)E_{r+1}E_r \otimes E_r\\
 -q \Big(\dfrac{K_{r-1}^{-1}-K_{r-1}}{q-q^{-1}}\Big)E_r \otimes E_{r+1}E_r -qK_{r-1}^{-1}E_{r+1}E_r \otimes \Big(\dfrac{K_{r-1}^{-1}-K_{r-1}}{q-q^{-1}}\Big)E_r.
\end{align*}
This in turn reduces to 
\begin{align*}
qE_r \otimes E_{r+1}\Big(\dfrac{K_{r-1}^{-1}-K_{r-1}}{q-q^{-1}}\Big)E_r + q^{2} E_{r+1}\Big(\dfrac{K_{r-1}^{-1}-K_{r-1}}{q-q^{-1}}\Big)E_r \otimes E_r\\
 -q \Big(\dfrac{K_{r-1}^{-1}-K_{r-1}}{q-q^{-1}}\Big)E_r \otimes E_{r+1}E_r -q^2E_{r+1}E_r \otimes \Big(\dfrac{K_{r-1}^{-1}-K_{r-1}}{q-q^{-1}}\Big)E_r,
\end{align*}
which finally cancels to $0$. A similar argument establishes that $F_{r+1}v = 0$. Thus, since $v$ is clearly a weight vector, we can conclude that it is a lowest weight vector as claimed.
\end{proof}

Since $T^{(0,1)} \otimes T^{(0,1)}$ is multiplicity-free as a $U_q(\mathfrak{l}_S)$-module, the lowest weight vector is necessarily an eigenvector of $\widecheck{\sigma}$. We now calculate its eigenvalue.

\begin{prop}
The third anti-holomorphic eigenvalue of $\OO_q(\mathrm{Gr}_{n+1,r})$ is $q^{2}$.
\end{prop}
\begin{proof}
We act on the lowest weight vector $v$, given in Lemma \ref{lem:Grass.LW} above, by $\widecheck{\sigma}$ on a term by term basis. For the first term we see that  
\begin{eqnarray*}
\widecheck{\sigma}(E_r \otimes E_{r-1}E_{r+1}E_r)=K_{r}^{-1}E_{r-1}E_{r+1}E_r \otimes E_r= E_{r-1}E_{r+1}E_r \otimes E_r.
\end{eqnarray*}
For the second term it can be calculated directly that 
\begin{eqnarray*}
\widecheck{\sigma}(E_{r-1}E_{r+1}E_r \otimes E_r)& = & E_r \otimes E_{r-1}E_{r+1}E_r - \nu E_{r-1}E_r \otimes E_{r+1}E_r \\
&~& -\nu E_{r+1}E_r \otimes E_{r-1}E_r + \nu^2 E_{r-1}E_{r+1}E_r \otimes E_r,
\end{eqnarray*}
where we have denoted $\nu := q - q^{-1}$. Similarly, for the third term it can be calculated directly that 
\begin{eqnarray*}
\widecheck{\sigma}(E_{r-1}E_r\otimes E_{r+1}E_r) &=& E_{r+1}E_r \otimes E_{r-1}E_r-\nu E_{r-1}E_{r+1}E_r \otimes E_r. 
\end{eqnarray*}
Finally, for the fourth term we have that 
\begin{eqnarray*}
\widecheck{\sigma}(E_{r+1}E_r\otimes E_{r-1}E_r)  &=& E_{r-1}E_r \otimes E_{r+1}E_r-\nu E_{r-1}E_{r+1}E_r \otimes E_r. 
\end{eqnarray*}
Collecting together these terms, we see that 
\begin{eqnarray*}
\widecheck{\sigma}(v)&=& E_{r-1}E_{r+1}E_r \otimes E_r + q^{2} E_r \otimes E_{r-1}E_{r+1}E_r +q^{2} \nu ^{2} E_{r-1}E_{r+1}E_r \otimes E_r \\
&~& - q^{2} \nu E_{r-1}E_r \otimes E_{r+1}E_r - q^{2} \nu E_{r+1}E_r \otimes E_{r-1}E_r -q E_{r+1}E_r \otimes E_{r-1}E_r \\
&~& +q \nu E_{r-1}E_{r+1}E_r \otimes E_r -q E_{r-1}E_r \otimes E_{r+1}E_r +q \nu E_{r-1}E_{r+1}E_r \otimes E_{r}\\
&=& q^{2}v.
\end{eqnarray*}
Thus we have calculated the claimed eigenvalue.
\end{proof}

\subsection{The Third Eigenvalue of the Quantum Lagrangian Grassmannians}

In this subsection we calculate the third anti-holomorphic eigenvalue of $\widecheck{\sigma}$ for the $C$-series cominuscule quantum flag manifolds, that is, for the quantum Lagrangian Grassmannians $\OO_q(\mathbf{LG}_{2n})$. As usual, we begin by determining a lowest weight vector of $T^{(0,1)} \otimes T^{(0,1)}$.

\begin{lem}
For $n \geq 2$, a lowest weight vector of $T^{(0,1)} \otimes T^{(0,1)}$ is given by 
$$
E_{n} \otimes E_{n-1}^{2}E_{n} + q^{2} E_{n-1}^{2}E_{n} \otimes E_{n} -q^{2} E_{n-1}E_{n} \otimes E_{n-1}E_{n}. 
$$
\end{lem}
\begin{proof}
Let us denote our proposed lowest weight vector by $v$. Note first that, for any $y \in \{1, \dots, n-2\}$, the commutation relations of $U_q(\mathfrak{sp}_{2n})$ immediately imply that the action of $F_{y}$ on $v$ is zero. Now for $y = n-1$, the diagonal action of $U_q(\mathfrak{l}_S)$ gives
\begin{align*}
F_{n-1}v = & \, K_{n-1}^{-1}E_{n} \otimes F_{n-1}E_{n-1}^{2}E_{n} + q^{2} F_{n-1}E_{n-1}^{2}E_{n}\otimes E_{n}\\
& ~ -q^{2} F_{n-1}E_{n-1}E_{n} \otimes E_{n-1}E_{n} -q^{2}K_{n-1}^{-1}E_{n-1}E_{n} \otimes F_{n-1}E_{n-1}E_{n}.
\end{align*}
Consider the tensor factor $F_{n-1}E_{n-1}^{2}E_{n}$. Using the commutation relations of $U_q(\mathfrak{sp}_{2n})$, we see that 
\begin{align*}
F_{n-1}E_{n-1}^{2}E_{n} = & \, E_{n-1}F_{n-1}E_{n-1}E_n + \Big(\dfrac{K_{n-1}^{-1}-K_{n-1}}{q-q^{-1}}\Big) E_{n-1}E_{n}.
\end{align*}
Since the second summand on the right hand side vanishes, this reduces to 
$$
E_{n-1}F_{n-1}E_{n-1}E_n.
$$
Another application of the commutation relations of $U_q(\mathfrak{sp}_{2n})$ reduces this to the expression 
\begin{align*}
E_{n-1}^2F_{n-1}E_n + E_{n-1}\Big(\dfrac{K_{n-1}^{-1}-K_{n-1}}{q-q^{-1}}\Big)E_n 
=  [2]_q E_{n-1}E_n.
\end{align*}
Considering next the tensor factor $F_{n-1}E_{n-1}E_{n}$, an analogous argument shows that 
$$
F_{n-1}E_{n-1}E_{n} = [2]_q E_n.
$$
Returning to the main calculation, we now see that it reduces to 
\begin{align*}
q^{2} [2]_q E_{n} \otimes E_{n-1}E_{n} + q^{2}[2]_q E_{n-1}E_{n} \otimes E_{n} -q^{2} [2]_q E_{n} \otimes E_{n-1}E_{n} -q^{2}[2]_q E_{n-1}E_{n}\otimes E_{n},
\end{align*}
which we see cancels to $0$. Thus, since $v$ is clearly a weight vector, we can conclude that it is a lowest weight vector.
\end{proof}

Since $T^{(0,1)} \otimes T^{(0,1)}$ is multiplicity-free as a $U_q(\mathfrak{l}_S)$-module, the lowest weight vector is necessarily an eigenvector of $\widecheck{\sigma}$. We now calculate its eigenvalue, which just as for the quantum Grassmannians, is $q^2$.

\begin{prop}
The third anti-holomorphic eigenvalue of $\OO_q(\mathbf{LG}_{2n})$ is $q^{2}$.
\end{prop}
\begin{proof}
We act on the lowest weight vector $v$ by $\widecheck{\sigma}$ on a term by term basis. For the first term we see that  
\begin{eqnarray*}
\widecheck{\sigma}(E_{n}\otimes E_{n-1}^{2}E_{n})= K_{n}^{-1}E_{n-1}^{2}E_{n}\otimes E_{n} = E_{n-1}^{2}E_{n}\otimes E_{n}.  
\end{eqnarray*}
For the second term it can be directly calculated that $\widecheck{\sigma}(E_{n-1}^{2}E_{n}\otimes E_{n})$ is equal to
\begin{eqnarray*}
(q^{-4}-q^{-2}-1+q^{2})E_{n-1}^{2}E_{n} \otimes E_{n} +(q^{-2}-q^{2})E_{n-1}E_{n}\otimes E_{n-1}E_{n} +E_{n}\otimes E_{n-1}^{2}E_{n}.
\end{eqnarray*}
Finally, for the third term we have that 
\begin{eqnarray*}
\widecheck{\sigma}(E_{n-1}E_{n}\otimes E_{n-1}E_{n}) = q^{-2} E_{n-1}E_{n}\otimes E_{n-1}E_{n} + (q^{-4}-1) E_{n-1}^{2}E_{n}\otimes E_{n}.
\end{eqnarray*}
Collecting together these terms we see that 
\begin{align*}
\widecheck{\sigma}(v) = \, q^2E_{n} \otimes E_{n-1}^{2}E_{n} + q^{4} E_{n-1}^{2}E_{n} \otimes E_{n} - q^4 E_{n-1}E_{n} \otimes E_{n-1}E_{n}  = q^2 v.
\end{align*} 
Thus $v$ is indeed an eigenvector of $\widecheck{\sigma}$ with eigenvalue $q^2$.
\end{proof}

\subsection{The Third Eigenvalue for the Remaining Non-Exceptional Cases} \label{subsection:third.eigenvalue}

For each of the remaining $B$ and $D$ series cominuscule quantum flag manifolds, the calculation of their remaining lowest weight vectors and their eigenvalues is analogous to the cases of the quantum Grassmannians and the quantum Lagrangian Grassmannians. Thus we will state the results without proof. We note that all the eigenvalues of $\widecheck{\sigma}$ are collected in a table in Appendix \ref{sub.appendix.DJ.Spectrum} below. The two exceptional examples $\OO_q(\mathbb{OP}^2)$ and $\O_q(\mathbf{F})$ will be treated in later work.

We begin with the odd quantum quadric, and remark that its third eigenvalue depends on the rank of the Lie algebra $\mathfrak{so}_{2n+1}$. As we will see in the next section, this behaviour reflects the properties of the braiding of the category of $U_q(\mathfrak{so}_{2n-1})$-modules. In the statement of the proposition we follow the conventions of Appendix \ref{Appendix.quadrics}.

\begin{prop}
For the odd quantum quadric $\mathcal{O}_q(\mathbf{Q}_{2n+1})$, a lowest weight vector of $T^{(0,1)} \otimes T^{(0,1)}$ is given by the alternating sum 
$$
\sum_{\alpha \in \overline{\Delta_{S}^{+}}, \, \alpha \leq \e_1 ~~}
(-q^{2})^{\mathrm{ht}(\alpha)-1}
E_{\alpha}\otimes E_{\theta -\alpha + \alpha_1}  ~~ -  \sum_{\alpha \in \overline{\Delta_{S}^{+}}, \, \alpha > \e_1}
(-q^{2})^{\mathrm{ht}(\alpha)-2}
E_{\alpha}\otimes E_{\theta -\alpha + \alpha_1}.
$$
Moreover, this element is an eigenvector of $\widecheck{\sigma}$ with eigenvalue $q^{2(2n-3)}$.
\end{prop}

Next we come to the $D$-series, beginning with the even quantum quadrics. We note that, just as for the odd quantum quadrics, the eigenvalue depends on the rank of the Lie algebra $\mathfrak{so}_{2n}$. Again, this behaviour reflects the properties of the braiding of the category of $U_q(\mathfrak{so}_{2n-2})$-modules. In the statement of this proposition we again follow the conventions of Appendix \ref{Appendix.quadrics}.

\begin{prop} \label{prop:third.EV.even.Q}
For the even quantum quadric $\mathcal{O}_q(\mathbf{Q}_{2n})$, a lowest weight vector of $T^{(0,1)} \otimes T^{(0,1)}$ is given by the alternating sum
$$
\sum_{\alpha \in \overline{\Delta^+_S}}
(-q)^{\mathrm{ht}(\alpha)-1}
E_{\alpha} \otimes E_{\theta -  \alpha + \alpha_1}.
$$
Moreover, this element is an eigenvector of $\widecheck{\sigma}$ with eigenvalue $q^{2(n-2)}$.
\end{prop}

We finish with the remaining $D$-series cominuscule quantum flag manifold, the quantum spinor variety. Unlike the previous two examples, the eigenvalue is independent of the rank of the Lie algebra $\mathfrak{sl}_{n}$.

\begin{prop}
For the quantum spinor variety $\mathcal{O}_q(\mathbf{S}_{2n})$,  a lowest weight vector of $T^{(0,1)} \otimes T^{(0,1)}$ is given by  
\begin{align*}
& \qquad ~~ E_{n} \otimes E_{n-2} E_{n-1}E_{n-3}E_{n-2}E_{n}  ~ - ~ q\,E_{n-2}E_{n} \otimes E_{n-1}E_{n-3}E_{n-2}E_{n} \\
 & \qquad + q^{2}E_{n-3}E_{n-2}E_{n} \otimes E_{n-1}E_{n-2}E_{n}  ~ + ~ q^{2}E_{n-1}E_{n-2}E_{n} \otimes E_{n-3}E_{n-2}E_{n} \\
 & \qquad  - ~q^{3}E_{n-1}E_{n-3}E_{n-2}E_{n}\otimes E_{n-2}E_{n}  + q^{4}E_{n-2}E_{n-1}E_{n-3}E_{n-2}E_{n} \otimes E_{n}. ~~~~~~~~~~~~~~~~~
\end{align*}
Moreover, this element is an eigenvector of $\widecheck{\sigma}$ with eigenvalue  $q^4$.
\end{prop}

\subsection{Relationship to the Quasitriangular Braiding} \label{subsection:quasi.tri.flips}

The category of finite-dimensional type-$1$ representations comes with an evident monoidal structure $\otimes$. This monoidal category has a well-known braiding $\sigma_{R}$, that comes from a \emph{quasitriangular type} structure $R$ for $U_q(\mathfrak{g})$. (Note that $R$ is not an element of $U_q(\mathfrak{g}) \otimes U_q(\mathfrak{g})$, instead it is contained in a completion of this space. See \cite[\textsection 8.3.3]{KSLeabh} for details.)

In order to compare $\sigma_{R}$ with our braiding, we need the following result of Drinfeld  \cite[\S8.4.3, Corollary~23]{KSLeabh}: Let $V_\lambda$ be the irreducible type-$1$ representation of~$U_q(\mathfrak{g})$ with highest weight $\lambda$, and let $V_{\mu}$ be an irreducible component of $V_{\lambda} \otimes V_{\lambda}$ with highest weight $\mu$. Then the braiding acts on $V_\mu$ as multiplication by the scalar
\begin{equation}\label{eq:RmatEigen}
  \pm q^{-(2(\lambda,\lambda+2\rho)-(\mu,\mu+2\rho))/2},
\end{equation}
where $\rho$ is the half-sum of positive roots in $\mathfrak{g}$. We note that the eigenvalues of the braiding 
$
\sigma_R: V_{\lambda} \otimes V_{\lambda} \to V_{\lambda} \otimes V_{\lambda}, 
$
coincide with the eigenvalues of 
$
\sigma_R: V_{\lambda}^{\vee} \otimes V_{\lambda}^{\vee} \to V_{\lambda}^{\vee} \otimes V_{\lambda}^{\vee},
$
for any $\lambda \in \mathcal{P}^+$. Indeed, this fact can be concluded from \eqref{eq:RmatEigen} and the Weyl group invariance of the inner product of the root system.

The eigenvalues for the braiding
$
\sigma_{R}: V^{\otimes (0,2)} \to V^{\otimes (0,2)},
$
for the special case of $U_q(\mathfrak{l}_S^s)$-modules, can be calculated using this formula, and we present the explicit values, for the non-exceptional cases,  in Appendix \ref{sub.appendix.DJ.Spectrum}. 

\begin{thm}
For every non-exceptional cominuscule quantum flag manifold, there exist non-zero scalars $c_-$ and $c_+$ such that 
\begin{align} \label{eqn:braid.scale}
\sigma|_{V^{\otimes(2,0)}} = c_- \, \sigma_{R}, & & \sigma|_{V^{\otimes(0,2)}} = c_+ \, \sigma_{R}^{-1}.
\end{align} 
\end{thm}
\begin{proof}
Looking at Table \ref{tab:low-rank-decomposition}, we see that the decomposition of $T^{(0,1)} \otimes T^{(0,1)}$ into irreducible $U_q(\mathfrak{l}_S^s)$-modules is multiplicity-free, and hence the decomposition of $V^{(0,1)} \otimes V^{(0,1)}$ into irreducible $U_q(\mathfrak{l}_S^s)$-modules is multiplicity-free. Thus it follows from Schur's lemma that to compare the maps $\sigma|_{V^{\otimes (0,2)}}$ and $\sigma_R^{-1}$, we just need to compare their eigenvalues for each irreducible summand. Moreover, multiplicity-freeness, along with non-degeneracy of the pairing between $T^{(0,1)} \otimes T^{(0,1)}$ and $V^{(0,1)} \otimes V^{(0,1)}$, implies that the eigenvalue of $\sigma$ for any irreducible summand $V_{\mu}$ of $V^{(0,1)} \otimes V^{(0,1)}$ equals the eigenvalue of $\widecheck{\sigma}$ for the dual irreducible summand $V_{\mu}^{\vee}$ of $T^{(0,1)} \otimes T^{(0,1)}$. Thus it suffices to compare the eigenvalues of $\widecheck{\sigma}|_{T^{\otimes (0,2)}}$ and $\sigma_R^{-1}$ on $V_{\mu}^{\vee}$.

Comparing the explicit eigenvalues of $\widecheck{\sigma}|_{T^{\otimes(0,2)}}$, as calculated above, with the eigenvalues of $\sigma^{-1}_R$, as presented in Appendix \ref{sub.appendix.DJ.Spectrum}, we see that for every non-exceptional cominuscule quantum flag manifold, there exists a scalar $c_+$ satisfying the second identity in \eqref{eqn:braid.scale}. 

For the first identity, we first recall the map 
$$
\dagger: \Phi(\Omega^1(B)) \otimes \Phi(\Omega^1(B)) \to \Phi(\Omega^1(B)) \otimes \Phi(\Omega^1(B)).
$$
It follows from the properties of the $*$-map that $\dagger$ maps each irreducible $U_q(\mathfrak{l}_S)$-submodule $V_{\mu} \subseteq \Phi(\Omega^{\otimes (0,2)})$ to a copy of the corresponding dual module $V_{\mu}^{\vee} \subseteq \Phi(\Omega^{\otimes (2,0)})$. Thus if we denote the $\sigma$-eigenvalue of $V_{\mu}$ by $\lambda$, then it follows from Corollary \ref{cor:dagger.eigen.value} that the $\sigma$-eigenvalue of $V_{\mu}^{\vee}$ will be $\lambda^{-1}$, where we have used the fact that every $\sigma$-eigenvalue is real. The first identity now follows from the fact that the braiding has the same eigenvalues for a module and its dual.
\end{proof}

Thus we see that on the holomorphic and anti-holomorphic forms, our braiding is just a rescaling of the quasitriangular braiding of the Drinfeld--Jimbo quantum groups. Thus we have found a noncommutative geometric realisation of these standard Yang--Baxter solutions. This result can be seen explicitly for some low rank examples in \textsection \ref{section:examples} below.

\subsection{A Comparison with Other Generalised Flip Maps} \label{subsection:other.flips}

In \cite{Drinfeld.normalised}, Drinfeld introduced a normalisation  $\widetilde{\sigma}_{R,q}$ of the braiding $\sigma_{R}$ such that its eigenvalues are $\pm 1$. It no longer gives a representation of the braid group, but instead a representation of the cactus group. Comparing eigenvalues we see that $\widetilde{\sigma}_{R,q}$ is clearly distinct from $\sigma$ on  $V^{\otimes(2,0)}$ and $V^{\otimes(0,2)}$, even up to rescaling. 

\bigskip

In \cite{Nichols.MMF} a quantum principal bundle approach to constructing Yang--Baxter solutions was introduced. It was then applied to the anti-holomorphic Heckenberger--Kolb forms of the quantum Grassmannians. This produced a diagonalisable map
$$
\varsigma: V^{(0,1)} \otimes V^{(0,1)} \to V^{(0,1)} \otimes V^{(0,1)},
$$
whose set of eigenvalues is presented in the table below:
\begin{table}[h]
\centering
\renewcommand{\arraystretch}{2}
\setlength{\tabcolsep}{12pt}
\begin{tabular}{|c|c|c|c|c|}
\hline
$V^{\otimes (0,2)} $ & $V_{2\varpi_1} \boxtimes V_{2\varpi_1}^{\vee}$ & $V_{\varpi_2} \boxtimes V_{\varpi_2}^{\vee}$ & $V_{2\varpi_1} \boxtimes V_{\varpi_2}^{\vee}$ & $V_{\varpi_2} \boxtimes V_{2\varpi_1}^{\vee}$ \\
\hline
Eigenvalue of $\varsigma$ & $1$ & $1$ & $-q^{-2}$ & $-q^{2}$ \\
\hline
\end{tabular}
\caption{The spectrum of the operator $\varsigma$ for the quantum Grassmannians $\OO_q(\mathrm{Gr}_{n+1,r})$, for $r \geq 2$.}
\label{tab:example}
\end{table}

Comparing eigenvalues, we see that we cannot rescale $\sigma$ to produce $\varsigma$. This is reflected in the fact that
$$
\ker(\varsigma - \id) = \mathrm{im}\big((\sigma + \id)_{V^{\otimes(0,2)}}\big).
$$
One operator gives the anti-holomorphic relations as a kernel, while the other gives them as an image.


\section{Some Low Rank Examples} \label{section:examples}

In this section we show how our formula for the map $\widecheck{\sigma}$ can be used to produce explicit descriptions of the degree two relations for some low rank examples of cominuscule quantum flag manifolds. This is interesting in that it concretely presents the peculiarities that distinguish the low rank examples across the different series. 

The examples in question are those of dimension less than or equal to $4$. We present all such cominuscule quantum flags in the following table.

\begin{table}[h]
\centering
\renewcommand{\arraystretch}{1.4}
\setlength{\tabcolsep}{10pt}
\begin{tabular}{|c|c|c|c|c|}
\hline
$\mathrm{dim}|T^{(0,1)}|$ & $1$ & $2$ & $3$ & $4$ \\
\hline
$\mathcal{O}_q(G/L_S)$ & $\mathcal{O}_q(S^2)$ & $\mathcal{O}_q(\mathbb{CP}^2)$ & $\mathcal{O}_q(\mathbb{CP}^3), \, \mathcal{O}_q(\mathbf{LQ})$ & $\mathcal{O}_q(\mathbb{CP}^4), \, \mathcal{O}_q(\mathrm{Gr}_{4,2})$ \\
\hline
\end{tabular}
\caption{Cominuscule $\mathcal{O}_q(G/L_S)$ of complex dimension $\leq 4$}
\label{tab:example}
\end{table}


\subsection{The Standard Podle\'s Sphere}

Let us start with the simplest example, namely the \emph{quantum projective line} $\O_q(\mathbb{CP}^1)$, or as it is better known, the \emph{standard Podle\'s sphere} $\OO_q(S^2)$. This is the single irreducible quantum flag manifold for $\OO_q(\mathrm{SU}_2)$. We cross the single node of the $A_1$ Dynkin diagram, that is, we choose $S = \emptyset$. The two tangent spaces have the following bases
\begin{align}
T^{(1,0)} = \mathbb{C} F_1 \subseteq \mathcal{O}_q(S^2)^{\circ}, & & T^{(0,1)} = \mathbb{C}E_1 \subseteq \mathcal{O}_q(S^2)^{\circ}.
\end{align}
In terms of this basis, the operator $\widecheck{\sigma}$ acts as follows
\begin{align*}
\widecheck{\sigma}(F_1 \otimes F_1) = q^{2} F_1 \otimes F_1, & & \widecheck{\sigma}(E_1 \otimes F_1) = q^{2} F_1 \otimes E_1\\
\widecheck{\sigma}(F_1 \otimes E_1) = q^{-2} E_1 \otimes F_1, & & \widecheck{\sigma}(E_1 \otimes E_1) = q^{-2} E_1 \otimes E_1.
\end{align*}
With respect to the basis 
\[
\mathcal B
=
\Big\{
F_1 \otimes F_1,\;
F_1 \otimes E_1,\;
E_1 \otimes F_1,\;
E_1 \otimes E_1
\Big\}
\]
the operator takes the form: 
\[
\widecheck{\sigma}
=
\begin{pmatrix}
q^{2} & 0      & 0      & 0 \\
0     & 0      & q^{2} & 0 \\
0     & q^{-2}  & 0      & 0 \\
0     & 0      & 0      & q^{-2}
\end{pmatrix}
\]
From this we see that an eigenbasis of $T \otimes T$ is given by 
\begin{align*}
F_1 \otimes F_1, & & E_1 \otimes E_1, & & E_1 \otimes F_1 + q^2 F_1 \otimes E_1,  & & E_1 \otimes F_1 - q^2 F_1 \otimes E_1,
\end{align*}
and their eigenvalues are respectively $q^{2}$, $q^{-2}$, $1$, and $-1$. 

We now dualise our formula for $\widecheck{\sigma}$  to the dual basis $\{e,f\}$ of $V^1$, noting that since the matrix of $\widecheck{\sigma}$ is not symmetric, care needs to be taken in producing the formula for $\sigma$. Doing so, we arrive at the following  relations in $V^{\bullet}$
\begin{align}
f \wedge f = 0, & & e \wedge f = -q^{-2} f \wedge e, & & e \wedge e = 0,
\end{align}
which of course form a full set of relations.


\subsection{The Quantum Projective Plane $\mathcal{O}_q(\mathbb{CP}^2)$}

We now proceed to the next simplest cominuscule example, the \emph{quantum projective plane} $\OO_q(\mathbb{CP}^2)$, that is, the irreducible quantum flag manifold for $\OO_q(\mathrm{\mathrm{SU}}_3)$. We can cross the first or the second node of the $A_2$ Dynkin diagram. We choose to cross the first:
\begin{center}
\begin{tikzpicture}[scale=1.2]
  \draw[thick] (0,0) -- (1,0);
  \filldraw (0,0) circle (0.08); 
  \draw (1.08,0) circle (0.08);     
\end{tikzpicture}
\end{center}
The holomorphic and anti-holomorphic tangent spaces have the following bases
\begin{align} \label{eqn:orderedbasis}
T^{(1,0)} = \mathrm{span}_{\mathbb{C}}\{F_1, \, F_2F_1\} \subseteq \O_q(\mathbb{CP}^2)^{\circ}, & & T^{(0,1)} = \mathrm{span}_{\mathbb{C}}\{E_1, \, E_2E_1\} \subseteq \O_q(\mathbb{CP}^2)^{\circ}.
\end{align}

\subsubsection{\textbf{The Anti-Holomorphic Corelations}}

The action of $\widecheck{\sigma}$ on pairs of anti-holomorphic elements is given by 
\begin{gather*}
\widecheck{\sigma}(E_1 \otimes E_1) = q^{-2} E_1 \otimes E_1, \quad \quad \widecheck{\sigma}(E_2E_1 \otimes E_2E_1) = q^{-2} E_2E_1 \otimes E_2E_1, \\
\widecheck{\sigma}(E_1 \otimes E_2E_1) = q^{-1} E_2E_1 \otimes E_1,
\end{gather*}
and the single action that is not a $q$-scaling of the flip map
\begin{align*}
\widecheck{\sigma}(E_2E_1 \otimes E_1) = q^{-1} E_1 \otimes E_2E_1  - q^{-1} \nu E_2E_1 \otimes E_1,
\end{align*}
where we have denoted  $\nu := q-q^{-1}$.

From this we see that $T^{(0,1)} \otimes T^{(0,1)}$ admits a basis of $\widecheck{\sigma}$-eigenvectors. We have two eigenvalues, namely $-1$ and $q^{-2}$. The $q^{-2}$-eigenspace is spanned by the elements
\begin{align*}
E_1 \otimes E_1, & &  E_2E_1 \otimes E_1 + q E_1 \otimes E_2E_1, & &  E_2E_1 \otimes E_2E_1,
\end{align*}
and the $-1$-eigenspace is spanned by the single eigenvector 
$$
E_2E_1 \otimes E_1 - q^{-1} E_1 \otimes E_2E_1.
$$ 
In particular, we see that $\widecheck{\sigma}$ satisfies the Hecke relation
$$
(\widecheck{\sigma} + 1)(\widecheck{\sigma} - q^{-2}) = 0.
$$
With respect to the choice of basis
$$
\Big\{E_1 \otimes E_1,  ~ E_1 \otimes E_2E_1, ~ ~ E_2E_1 \otimes E_1, ~ E_{2}E_1 \otimes E_2E_1  \Big \},
$$
we have the following matrix presentation of the operator
\[
[\widecheck{\sigma}_{T^{\otimes(0,2)}}] =
\begin{pmatrix}
q^{-2} & 0 & 0 & 0 \\
0 & 0 & q^{-1} & 0 \\
0 & q^{-1} & - q^{-1}\,\nu & 0 \\
0 & 0 & 0 & q^{-2}
\end{pmatrix} = q^{-1} \widecheck{R}^{-1},
\]
where we have denoted
\begin{align*}
\widecheck{R}^{ji}_{kl} := R^{ij}_{kl} :=  q^{\delta_{ij}}\delta_{ik}\delta_{jl} + \Theta(i-j)\mkern2mu\nu \mkern2mu \delta_{il}\delta_{kj}, & & \textrm{ for } i,j,k,l = 1,2,
\end{align*}
with $\Theta$ the standard Heaviside step function. Note that $R$ is the standard $(4 \times 4)$-$R$ matrix underlying the FRT-construction of $\OO_q(\mathrm{SU}_2)$, see \cite[\textsection 9.2]{KSLeabh} for details.

Turning to the holomorphic corelations, analogous calculations show that, with respect to the basis
$$
\Big\{F_1 \otimes F_1, ~ F_2F_1 \otimes F_1, ~ F_1 \otimes F_2F_1, ~ F_{2}F_1 \otimes F_2F_1  \Big \},
$$
the matrix of $\widecheck{\sigma}|_{T^{\otimes (2,0)}}$ is given by 
\[
[\widecheck{\sigma}_{T^{\otimes(2,0)}}]  =
\begin{pmatrix}
q^{2} & 0 & 0 & 0 \\
0 & q \mkern2mu \nu & q & 0 \\
0 & q & 0 & 0 \\
0 & 0 & 0 & q^{2}
\end{pmatrix}
 = q \widecheck{R} = [\widecheck{\sigma}_{T^{\otimes(0,2)}}]^{-1}.
\]


\subsubsection{\textbf{The Cross Terms}}

We start with an ordered basis for $T^{\otimes (1,1)}$\,:
\begin{align*}
\Bigg\{
\begin{aligned}
F_1 \otimes E_2E_1, ~ F_1 \otimes E_1, ~ F_2F_1 \otimes E_2E_1, ~ F_2F_1 \otimes E_1, ~~ \\
~~ E_2E_1 \otimes F_1, ~ E_1 \otimes F_1, ~ E_2E_1 \otimes F_2F_1, ~ E_1 \otimes F_2F_1
\end{aligned}
\Bigg\}
\end{align*}
Since $\widecheck{\sigma}$ is compatible with the complex decomposition, it interchanges $T^{(1,0)} \otimes T^{(0,1)}$ and $T^{(0,1)} \otimes T^{(1,0)}$. Moreover, since the complex structure is factorisable,  $\widecheck\sigma^2=\mathrm{id}$ on the mixed component. Hence, relative to our ordered basis its matrix has the form
\begin{align*}
\begin{pmatrix}
0 & A^{-1} \\
A & 0
\end{pmatrix}\!,  & & \textrm{ for some } A \in \mathrm{GL}_4(\mathbb{C}).
\end{align*}
For the action  of $\widecheck{\sigma}$ on $T^{(1,0)} \otimes T^{(0,1)}$, we have the $q$-scalings of the flip map
\begin{align*}
\widecheck{\sigma}(F_1 \otimes E_2E_1) = q^{-1} E_2E_1 \otimes F_1, & & \widecheck{\sigma}(F_2F_1 \otimes E_1) = q^{-1} E_1 \otimes F_2F_1,
\end{align*}
\begin{align*}
\widecheck{\sigma}(F_1 \otimes E_1) = q^{-2} E_1 \otimes F_1,
\end{align*}
and a single non-trivial action
\begin{align*}
\widecheck{\sigma}(F_2F_1 \otimes E_2E_1) = q^{-2} E_2E_1 \otimes F_2F_1 + q^{-2}\, \nu E_1 \otimes F_1.
\end{align*}
Thus we see that 
$$
A = \begin{pmatrix}
~\, q^{-1} & 0 & 0 & 0 \\
0 & q^{-2} & q^{-2}\nu & 0 \\
0 & 0 & q^{-2} & 0 \\
0 & 0 & 0 & q^{-1}
\end{pmatrix} = q^{-2} \begin{pmatrix}
~\, q & 0 & 0 & 0 \\
0 & 1 & \nu & 0 \\
0 & 0 & 1 & 0 \\
0 & 0 & 0 & q~
\end{pmatrix} = q^{-2}R.
$$
In other words, with respect to our choice of basis, the matrix of the restriction of $\widecheck{\sigma}$ to $T^{\otimes(1,1)}$ is given by 
\begin{align*}
\left(
\begin{array}{@{\hspace{8pt}}c@{\hspace{18pt}}c@{\hspace{8pt}}}
\text{\large $0$} & q^{2}R^{-1} \\[8pt]
q^{-2}R & \text{\large $0$}
\end{array}
\right)\!, & & \textrm{ where } R^{-1} = \begin{pmatrix}
~\, q^{-1} & 0 & 0 & 0 \\
0 & 1 & -\nu & 0 \\
0 & 0 & 1 & 0 \\
0 & 0 & 0 & q^{-1}~
\end{pmatrix}
\end{align*}

The $1$-eigenspace decomposes into a direct sum of a copy of the $3$-dimensional $U_q(\mathfrak{sl}_2)$ \emph{quantum adjoint} representation $V_{2\varpi_1}$, and a copy of the trivial representation $V_0$. The copy of $V_{2\varpi_1}$ is spanned by the elements
\begin{align*}
F_1 \otimes E_2E_1 + q^{-1} E_2E_1 \otimes F_1, ~~~~~~~~~~~~~~~~~~~~~~~~~~~~~ F_2F_1 \otimes E_1 + q^{-1} E_1 \otimes F_2F_1, \\
F_2F_1 \otimes E_2E_1 + q^{-1} F_1 \otimes E_1 + q^{-1} E_1 \otimes F_1 + q^{-2}E_2E_1 \otimes F_2F_1.~~~~~~~~~~~
\end{align*}
The copy of the trivial representation is spanned by the element
\begin{align*}
F_2F_1 \otimes E_2E_1 - q F_1 \otimes E_1 - q^{-3} E_1 \otimes F_1 + q^{-2} E_2E_1 \otimes F_2F_1.
\end{align*}

Analogously, the $-1$-eigenspace decomposes into a direct sum of a copy of the $3$-dimensional $U_q(\mathfrak{sl}_2)$ \emph{quantum adjoint} representation $V_{2\varpi_1}$, and a copy of the trivial representation $V_0$. The copy of $V_{2\varpi_1}$ is spanned by the elements
\begin{align*}
F_1 \otimes E_2E_1 - q^{-1} E_2E_1 \otimes F_1, ~~~~~~~~~~~~~~~~~~~~~~~~~~~~~ F_2F_1 \otimes E_1 - q^{-1} E_1 \otimes F_2F_1, \\
F_2F_1 \otimes E_2E_1 + q^{-1} F_1 \otimes E_1 - q^{-1} E_1 \otimes F_1 - q^{-2}E_2E_1 \otimes F_2F_1.~~~~~~~~~~~
\end{align*}
The copy of the trivial representation is spanned by the element
\begin{align*}
F_2F_1 \otimes E_2E_1 - q F_1 \otimes E_1 + q^{-3} E_1 \otimes F_1 - q^{-2} E_2E_1 \otimes F_2F_1.
\end{align*}

\subsubsection{\textbf{The Relations}} We denote the corresponding dual bases by 
\begin{align*}
V^{(1,0)} = \mathrm{span}_{\mathbb{C}}\{f_1, \, f_{2} \}, & & V^{(0,1)} = \mathrm{span}_{\mathbb{C}}\{e_1, \, e_{2}\}.
\end{align*}
Dualising the formulae for $\widecheck{\sigma}$ to the corresponding formulae for $V^1$, we arrive at the following relations in $V^{\bullet}$: First are the holomorphic relations
\begin{align*}
f_2 \wedge f_1 = -q^{-1} f_1 \wedge f_2, & & f_1 \wedge f_1 = f_2 \wedge f_2 = 0.
\end{align*}
Next we have the anti-holomorphic relations
\begin{align*}
e_2 \wedge e_1 = -q e_1 \wedge e_2, & & e_1 \wedge e_1 = e_2 \wedge e_2 = 0.
\end{align*}
Finally, we have the cross relations
\begin{align*}
e_1 \wedge f_1 =  - q^{-2} f_1 \wedge e_1 - q^{-2}\nu f_2 \wedge e_2, & & e_2 \wedge f_2 = - q^{-2} f_2 \wedge e_2.
\end{align*}
\begin{align*}
e_1 \wedge f_2 = - q^{-1} f_2 \wedge e_1, & &   e_2 \wedge f_1 = - q^{-1} f_1 \wedge e_2.
\end{align*}
Together these give a full set of relations for the quantum exterior algebra $V^{\bullet}$. 

We finish by noting that the form of the generalised flip map $\sigma$ and the relations for $\OO_q(\mathbb{CP}^2)$ generalise directly to the higher rank quantum projective spaces $\OO_q(\mathbb{CP}^n)$. For an explicit description of these relations (albeit in somewhat different conventions) see \cite[Proposition 5.8]{MMF1}.


\subsection{The Quantum Lagrangian Quadric $\O_q(\mathbf{LQ})$}

For the case of $C_2$, we have two simple roots
$\Pi = \{\alpha_1, \alpha_2\}$. We consider the subset of simple roots $S=\{\alpha_1\}$,  corresponding to the Dynkin diagram:
\begin{center}
\begin{tikzpicture}[scale=1.2]
  \draw[thick] (0.12,1.05) -- (0.88,1.05);
  \draw[thick] (0.12,0.95) -- (0.88,0.95);

  \draw[thick] (0.52,1.15) -- (0.37,1);  
  \draw[thick] (0.52,0.85) -- (0.37,1);  

  \draw (0,1) circle (0.08);  
  \filldraw (1,1) circle (0.08);      
\end{tikzpicture}
\end{center}
This gives the first quantum Lagrangian Grassmannian $\OO_q(\mathbf{LG}_{4})$. Since the root systems $B_2$ and $C_2$ are isomorphic, this quantum homogeneous space is also the $B_2$-irreducible quantum flag manifold, that is, the first odd quantum quadric $\OO_q(\mathbf{Q}_{5})$. 
\begin{center}
\begin{tikzpicture}[scale=1.2]
  \draw[thick] (0.12,1.05) -- (0.88,1.05);
  \draw[thick] (0.12,0.95) -- (0.88,0.95);

  \draw[thick] (0.48,1.15) -- (0.63,1);  
  \draw[thick] (0.48,0.85) -- (0.63,1);  

  \filldraw (0,1) circle (0.08);
  \draw (1,1) circle (0.08);
\end{tikzpicture}
\end{center}
We call this quantum homogeneous space, which is both a quantum Lagrangian Grassmannian and a quantum quadric, the \emph{quantum Lagrangian quadric} and denote it by
$$
\mathcal{O}_q(\mathbf{LQ}).
$$
We content ourselves with calculating the holomorphic and anti-holomorphic braidings and relations. However, cross terms can be treated as for previous examples.

\subsubsection{\textbf{The Holomorphic and Anti-Holomorphic Corelations}}

Since this example is somewhat more involved than $\OO_q(S^2)$ and $\OO_q(\mathbb{CP}^2)$, we find it helpful to use the fact that $\widecheck{\sigma}$ is a weight map. We note that the weights of the basis elements of $T^{(0,1)}$, as a $U_q(\mathfrak{l}_S)$-module, are as follows 
\begin{align*}
|E_2| = \alpha_2, & & |E_1E_2| = \alpha_1 + \alpha_2, & & |E_1^2E_2| = 2\alpha_1 + \alpha_2 =: \theta,
\end{align*}
where as usual we use $\theta$ to denote the highest root. The weights of the tensor product of basis elements are of course given by taking the sums of the weights of the factors, as presented in the following table:

\[
\renewcommand{\arraystretch}{1.4}
\setlength{\arraycolsep}{10pt}
\begin{array}{|c|c|c|c|}
\hline
   & E_2 & E_1E_2 & E_1^2E_2 \\ \hline
E_2
  & 2\alpha_2
  & {\color{blue}\alpha_1+2\alpha_2}
  & {\color{red}2\alpha_1+2\alpha_2} \\ \hline
E_1E_2
  & {\color{blue}\alpha_1+2\alpha_2}
  & {\color{red}2\alpha_1+2\alpha_2}
  & {\color{blue}3\alpha_1+2\alpha_2} \\ \hline
E_1^2E_2
  & {\color{red}2\alpha_1+2\alpha_2}
  & {\color{blue}3\alpha_1+2\alpha_2}
  & 4\alpha_1+2\alpha_2 \\ \hline
\end{array}
\]

We see that the space $T^{(0,1)} \otimes T^{(0,1)}$ decomposes into five weight subspaces. Coloured in black we have the two weight subspaces of multiplicity one, coloured in blue we have the two weight subspaces of multiplicity two, and coloured in red we have the single weight subspace of multiplicity three.

The weight table shows that, with respect to an appropriate choice of basis, the matrix of $\widecheck{\sigma}|_{T^{\otimes (0,2)}}$ can be written in block matrix form. Direct calculation using formula \eqref{eqn:LTprimitive formula} shows that one such block form is given by 
\[
[\widecheck{\sigma}|_{T^{\otimes(0,2)}}] =
\begin{pmatrix}
q^{-4} & 0 & 0 & 0 & 0 \\
0 & A(2) & 0 & 0 & 0 \\
0 & 0 & A(3) & 0 & 0 \\
0 & 0 & 0 & A(2) & 0 \\
0 & 0 & 0 & 0 & q^{-4}
\end{pmatrix}
\]
where the $(2 \times 2)$-block matrix is given by 
\begin{align*}
A(2) = 
\begin{pmatrix}
 0  & q^{-2}  \\
 q^{-2} & q^{-4}-1
\end{pmatrix} 
= q^{-2} \begin{pmatrix}
 0  & 1 \\
 1 & - [2]_q\,\nu
\end{pmatrix}\!,
\end{align*}
and the $(3 \times 3)$-block matrix is given by 
\begin{align*}
A(3) \;=\;
\left(
\begin{array}{ccc}
0 & 0 & 1 \\[8pt]
0 & q^{-2} & q^{-2}-q^2 \\[8pt]
1 & q^{-4} -1 & q^{-4} - q^{-2} - 1 + q^2 
\end{array}
\right)  = 
\left(
\begin{array}{ccc}
0 & 0 & 1 \\[8pt]
0 & q^{-2} & - [2]_q \, \nu \\[8pt]
1 &  - q^{-2} \, [2]_q \, \nu & q^{-1}\, [2]_q \, \nu^2
\end{array}
\right)
\end{align*}
We note that unlike for the $A$-series examples, this matrix is not symmetric, meaning that we have a non-trivial transpose to take into account when calculating the dual relations. This recalls the fact that the $A$-series Cartan matrix is symmetric, while the $C$-series matrix is not.

\subsubsection{\textbf{The Anti-Holomorphic Relations}}

Coming now to the cotangent space $V^{(0,1)}$, we denote the basis dual to the basis of $T^{(0,1)}$ by  
\begin{align*}
V^{(0,1)} = \mathrm{span}_{\mathbb{C}}\big\{e_{\alpha_2}, \, e_{\alpha_1+\alpha_2}, \, e_{\theta}\big\}.
\end{align*}

For pairs of distinct basis elements, we have the relations
\begin{align*}
e_{\alpha_1 + \alpha_2} \wedge e_{\alpha_2}= -q^{2} e_{\alpha_2} \wedge e_{\alpha_1 + \alpha_2}, & & e_{\theta} \wedge e_{\alpha_1 + \alpha_2} = -q^{2} e_{\alpha_1 + \alpha_2} \wedge e_{\theta},
\end{align*}
$$
e_{\alpha_2} \wedge e_{\theta} = - e_{\theta} \wedge e_{\alpha_2}.
$$
Note that these relations can be uniformly expressed as
\begin{align*}
e_{\beta} \wedge e_{\alpha} = - q^{(\alpha,\beta)} e_{\alpha} \wedge e_{\beta}, & & \textrm{ for } \alpha < \beta,
\end{align*}
where $<$ denotes the restriction of the partial order $\prec $ of $\Delta$ to a total order on $\overline{\Delta^+_S}$. 

The anti-holomorphic relations for matching pairs of basis elements are
\begin{align*}
e_{\alpha_2}\wedge e_{\alpha_2} = e_{\theta} \wedge e_{\theta} = 0, & &e_{\alpha_1 + \alpha_2} \wedge e_{\alpha_1 + \alpha_2} = - q \nu e_{\alpha_2} \wedge e_{\theta}.
\end{align*}
To derive the last identity one needs to be careful to use the transpose of the matrix $A(3)$. Moreover, the derivation uses the elementary identity $(q^{-2}-q^2)(1+q^{-2})^{-1} = - q \nu$ and the relation $e_{\theta} \wedge e_{\alpha_2} = - e_{\alpha_2} \wedge e_{\theta}.$

Together these give a full set of relations for the anti-holomorphic quantum exterior algebra $V^{(0,\bullet)}$.

\subsection{The Quantum Grassmannian $\OO_q(\mathrm{Gr}_{4,2})$}

We finish with the first cominuscule quantum flag manifold of non-projective space type that has a four-dimensional anti-holomorphic tangent space: the \emph{$(4,2)$-quantum Grassmannian}. This is the cominuscule quantum flag manifold of $\OO_q(\mathrm{SU}_4)$ given by crossing the middle node of the $A_3$ Dynkin diagram:
\begin{center}
\begin{tikzpicture}[scale=1.2]
  \draw[thick] (0.08,0) -- (0.92,0);
  \draw[thick] (1.08,0) -- (1.92,0);

  \draw (0,0) circle (0.08);
  \fill (1,0) circle (0.08);
  \draw (2,0) circle (0.08);
\end{tikzpicture}
\end{center}
In other words $S = \{\alpha_1, \alpha_3\} \subseteq \Pi$. We content ourselves with the anti-holomorphic braiding and relations.

As an ordered  basis for the tangent space $T^{(0,1)}$ we choose 
$$ 
\Big\{E_2, ~ E_1E_2, ~ E_3E_2, ~ E_1E_3E_2\Big\}.
$$
The weights of these vectors are given by the ordered set
$$
\Big\{\alpha_2, ~ \alpha_1 + \alpha_2, ~ \alpha_2 + \alpha_3, ~ \theta := \alpha_1 + \alpha_2 + \alpha_3\Big\}.
$$
The table of weights of the tensor products of basis elements is given by:
\[
\renewcommand{\arraystretch}{1.4}
\setlength{\arraycolsep}{8pt}
\begin{array}{|c|c|c|c|c|}
\hline
 & E_2 & E_1E_2 & E_3E_2 & E_1E_3E_2 \\ \hline
E_2 & 2\alpha_2 & {\color{blue}\alpha_1 + 2\alpha_2} & {\color{purple}2\alpha_2+\alpha_3} & {\color{red}\alpha_1+2\alpha_2+\alpha_3} \\ \hline
E_1E_2 & {\color{blue}\alpha_1 + 2\alpha_2} & 2\alpha_1+2\alpha_2 & {\color{red}\alpha_1+2\alpha_2+\alpha_3} & {\color{blue}2\alpha_1+2\alpha_2+\alpha_3} \\ \hline
E_3E_2 & {\color{purple}2\alpha_2+\alpha_3} & {\color{red}\alpha_1+2\alpha_2+\alpha_3} & 2\alpha_2+2\alpha_3 & {\color{purple}\alpha_1+2\alpha_2+2\alpha_3} \\ \hline
E_1E_3E_2 & {\color{red}\alpha_1+2\alpha_2+\alpha_3} & {\color{purple}2\alpha_1+2\alpha_2+\alpha_3} & {\color{blue}\alpha_1+2\alpha_2+2\alpha_3} & 2\alpha_1+2\alpha_2+2\alpha_3 \\ \hline
\end{array}
\]

As before we have highlighted the different weight spaces in different colours. From this table we can also see that the matrix of $\widecheck{\sigma}$ decomposes into four block matrices of size $1$, four block matrices of size $2$, and a single block of size $4$. Calculating $\widecheck{\sigma}$ explicitly we see the blocks of size $1$ have the single entry $q^{-2}$, and the blocks of size $2$ are of the usual form
$$
\left(
\begin{array}{cc}
0  & q^{-1}   \\
q^{-1}  & -q^{-1} \nu 
\end{array} \right) = q^{-1} \left(
\begin{array}{cc}
 0 & 1  \\
1 & - \nu 
\end{array}
\right)
$$
However, for the block of size $4$ we get something new:
$$
\renewcommand{\arraystretch}{1.6}
\setlength{\arraycolsep}{12pt}
\left(
\begin{array}{cccc}
0 & 0 & 0 & 1 \\
0 & 0 & 1 & -\nu \\
0 & 1 & 0 & -\nu \\
1 & -\nu & -\nu & \nu^2
\end{array}
\right)  = \left(
\begin{array}{cc}
0 & 1  \\
1 & - \nu
\end{array}
\right) \otimes \left(
\begin{array}{cc}
0 & 1  \\
1 & - \nu 
\end{array}
\right),
$$
where $\otimes$ denotes the Kronecker product of two matrices. The first column gives the corelation $e_{\alpha_2} \otimes e_{\theta} + e_{\theta} \otimes e_{\alpha_2}$. The second and third columns give the single corelation
$$
e_{\alpha_1 + \alpha_2} \otimes e_{\alpha_2 + \alpha_3} + e_{\alpha_2 + \alpha_3} \otimes e_{\alpha_1 + \alpha_2} - \nu e_{\theta} \otimes e_{\alpha_2}.
$$
Finally, the fourth column gives the corelation
$$
e_{\theta} \otimes e_{\alpha_2} + e_{\alpha_2} \otimes e_{\theta} - \nu \big(e_{\alpha_1 + \alpha_2} \otimes e_{\alpha_2 + \alpha_3} + e_{\alpha_2 + \alpha_3} \otimes e_{\alpha_1 + \alpha_2} - \nu e_{\theta} \otimes e_{\alpha_2}\Big).
$$
Thus we see that the fourth column simply adds the corelation 
$$
e_{\theta} \otimes e_{\alpha_2} + e_{\alpha_2} \otimes e_{\theta}.
$$

Since the block matrices are all symmetric, dualising to the set of relations is straightforward. We index the dual basis of $V^{(0,1)}$ by the roots of these elements, that is, we write
$$
\Big\{e_{\alpha_2}, ~ e_{\alpha_1 + \alpha_2}, ~ e_{\alpha_2 + \alpha_3}, ~ e_{\theta}\Big\}.
$$
In terms of this basis, we have the $q$-skew commutation relations
\begin{align*}
e_{\alpha} \wedge e_{\alpha} = 0, \textrm{ ~~~ for all } \alpha \in \overline{\Delta^+_S}, & &  e_{\beta} \wedge e_{\alpha} = - q^{(\alpha,\beta)} e_{\alpha} \wedge e_{\beta},  \textrm{ ~~~ for }\alpha \prec \beta,
\end{align*}
and the single non-$q$-skew commutation relation
\begin{align*}
e_{\alpha_1 + \alpha_2} \wedge e_{\alpha_2 + \alpha_3} =  - e_{\alpha_2 + \alpha_3} \wedge e_{\alpha_1 + \alpha_2} - \nu e_{\alpha_2} \wedge e_{\theta}.
\end{align*}
This gives a full set of relations for $V^{(0,\bullet)}$. 

\subsection{A Small Observation about the Relations}

For the  quantum Lagrangian quadric $\OO_q(\mathbf{LQ})$, the fact that $e_{\alpha_1 + \alpha_2} \wedge e_{\alpha_1 + \alpha_2} \neq 0$ is a somewhat surprising quantum feature. In the following proposition we show that this phenomenon only occurs for non-simply laced root systems.

\begin{prop}
Let $\OO_q(G/L_S)$ be a cominuscule quantum flag manifold associated to a simply laced root system, and take a choice of weight bases 
\begin{align*}
\{ f_{\alpha}\}_{\alpha \in \overline{\Delta^+_S}},  & & \{ e_{\alpha}\}_{\alpha \in \overline{\Delta^+_S}}, 
\end{align*}
of $V^{(1,0)}$ and $V^{(0,1)}$ respectively. Then it holds that 
\begin{align*}
e_{\alpha} \wedge e_{\alpha} = f_{\alpha} \wedge f_{\alpha} = 0, & & \textrm{ for all }  \alpha \in \overline{\Delta^+_S}.
\end{align*}
\end{prop}
\begin{proof}
Consider the relation 
$$
e_{\alpha} \otimes e_{\alpha} + \sigma(e_{\alpha} \otimes e_{\alpha}).
$$
If $\sigma(e_{\alpha} \otimes e_{\alpha})$ is a scalar multiple of $e_{\alpha} \otimes e_{\alpha}$, since this scalar multiple will necessarily be $q^{-(\alpha,\alpha)} \neq -1$, it must hold that $e_{\alpha} \wedge e_{\alpha} = 0$. If not, then since $\sigma$ is a $U_q(\mathfrak{l}_S)$-module map, $\sigma(e_{\alpha} \otimes e_{\alpha})$ must be a linear combination of basis elements of the form $e_{\beta} \otimes e_{\gamma}$, where  
\begin{align*}
2\alpha = \beta + \gamma, & & \textrm{ and } & & \alpha \neq \beta, \gamma.
\end{align*}
If such a pair exists, then 
$$
8 = (2\alpha,2\alpha) = (\beta + \gamma, \beta + \gamma) = (\beta,\beta) + (\gamma,\gamma) + 2(\beta,\gamma) = 4 + 2(\beta,\gamma).
$$
Thus we see that $(\beta,\gamma) = 2$. Recalling the Cauchy--Schwarz inequality, we see that 
$$
|(\beta,\gamma)| \leq \|\beta\|\|\gamma\| = \sqrt{2}\sqrt{2} = 2.
$$
The fact that we have equality implies that $\beta$ and $\gamma$ are linearly dependent, which from the axioms of a root system implies that $\beta = \pm \gamma$. Since $(\beta,\gamma) = 2$ we must have that $\beta = \gamma$, contradicting our assumptions. Thus we see that $e_{\alpha} \wedge e_{\alpha} = 0$. The same argument establishes that $f_{\alpha} \wedge f_{\alpha}$ also vanishes.
\end{proof}


\appendix

\section{Tables for the Cominuscule Quantum Flag Manifolds}\label{section:Appendix}

In this appendix we recall some useful technical results about classical Lie algebras that carry over directly to the quantum setting. Moreover, we recall some explicit calculations of quasitriangular braiding eigenvalues. These results are used throughout  \textsection \ref{section:spectrum} and \textsection \ref{section:examples}.

\subsection{The Cominuscule Roots}

We recall the standard pictorial description of the quantum Levi subalgebras defining the cominuscule quantum flag manifolds, given in terms of Dynkin diagrams.

\begin{table}[ht]
\centering
\caption{\small
Cominuscule quantum flag manifolds, organised by series, with defining
crossed node,
symbol, and name. The nodes of the diagrams are numbered according to \cite[\textsection 11.4]{Humph}.}
\label{table:CQFMs}

\small
\renewcommand{\arraystretch}{2.0}
\setlength{\extrarowheight}{3pt}
\setlength{\tabcolsep}{4pt}

\resizebox{\textwidth}{!}{%
\begin{tabular}{@{}c c c l@{}}
\hline\hline
\textbf{Type}
&
\textbf{Dynkin diagram}
&
\textbf{Symbol}
&
\textbf{Name}
\\
\hline

$A_{n\geq 1}$
&
\begin{tikzpicture}[scale=.5,baseline=-.15em]
\draw
(0,0) circle [radius=.25]
(8,0) circle [radius=.25]
(2,0) circle [radius=.25]
(6,0) circle [radius=.25];

\draw[fill=black]
(4,0) circle [radius=.25];

\draw[thick,dotted]
(2.25,0) -- (3.75,0)
(4.25,0) -- (5.75,0);

\draw[thick]
(.25,0) -- (1.75,0)
(6.25,0) -- (7.75,0);
\end{tikzpicture}
&
$\OO_q(\text{Gr}_{n+1,r})$
&
quantum Grass.
\\

$B_{n\geq 2}$
&
\begin{tikzpicture}[scale=.5,baseline=-.15em]
\draw
(4,0) circle [radius=.25]
(2,0) circle [radius=.25]
(6,0) circle [radius=.25]
(8,0) circle [radius=.25];

\draw[fill=black]
(0,0) circle [radius=.25];

\draw[thick]
(.25,0) -- (1.75,0);

\draw[thick,dotted]
(2.25,0) -- (3.75,0)
(4.25,0) -- (5.75,0);

\draw[thick]
(6.25,-.06) --++ (1.5,0)
(6.25,+.06) --++ (1.5,0);

\draw[thick]
(7,0.15) --++ (-60:.2)
(7,-.15) --++ (60:.2);
\end{tikzpicture}
&
$\OO_q(\mathbf{Q}_{2n+1})$
&
odd quantum quadric
\\

$C_{n\geq 2}$
&
\begin{tikzpicture}[scale=.5,baseline=-.15em]
\draw
(0,0) circle [radius=.25]
(2,0) circle [radius=.25]
(4,0) circle [radius=.25]
(6,0) circle [radius=.25];

\draw[fill=black]
(8,0) circle [radius=.25];

\draw[thick]
(.25,0) -- (1.75,0);

\draw[thick,dotted]
(2.25,0) -- (3.75,0)
(4.25,0) -- (5.75,0);

\draw[thick]
(6.25,-.06) --++ (1.5,0)
(6.25,+.06) --++ (1.5,0);

\draw[thick]
(7,0) --++ (60:.2)
(7,0) --++ (-60:.2);
\end{tikzpicture}
&
$\OO_q(\mathbf{LG}_{2n})$
&
quantum Lagrangian Grass.
\\

$D_{n\geq 4}$
&
\begin{tikzpicture}[scale=.5,baseline=-.15em]
\draw[fill=black]
(0,0) circle [radius=.25];

\draw
(2,0) circle [radius=.25]
(4,0) circle [radius=.25]
(6,.5) circle [radius=.25]
(6,-.5) circle [radius=.25];

\draw[thick]
(.25,0) -- (1.75,0)
(4.25,0.1) -- (5.75,.5)
(4.25,-0.1) -- (5.75,-.5);

\draw[thick,dotted]
(2.25,0) -- (3.75,0);
\end{tikzpicture}
&
$\OO_q(\mathbf{Q}_{2n})$
&
even quantum quadric
\\

$D_{n \geq 4}$
&
\begin{tikzpicture}[scale=.5,baseline=-.15em]
\draw
(0,0) circle [radius=.25]
(2,0) circle [radius=.25]
(4,0) circle [radius=.25];

\draw[fill=black]
(6,.5) circle [radius=.25];

\draw
(6,-.5) circle [radius=.25];

\draw[thick]
(.25,0) -- (1.75,0)
(4.25,0.1) -- (5.75,.5)
(4.25,-0.1) -- (5.75,-.5);

\draw[thick,dotted]
(2.25,0) -- (3.75,0);
\end{tikzpicture}
&
$\OO_q(\mathbf{S}_{2n})$
&
quantum spinor variety
\\

$E_6$
&
\begin{tikzpicture}[scale=.5,baseline=-.15em]
\draw
(2,0) circle [radius=.25]
(4,0) circle [radius=.25]
(4,1) circle [radius=.25]
(6,0) circle [radius=.25];

\draw
(0,0) circle [radius=.25];

\draw[fill=black]
(8,0) circle [radius=.25];

\draw[thick]
(.25,0) -- (1.75,0)
(2.25,0) -- (3.75,0)
(4.25,0) -- (5.75,0)
(6.25,0) -- (7.75,0)
(4,.25) -- (4,.75);
\end{tikzpicture}
&
$\OO_q(\mathbb{OP}^2)$
&
quantum Cayley plane
\\

$E_7$
&
\begin{tikzpicture}[scale=.5,baseline=-.15em]
\draw
(0,0) circle [radius=.25]
(2,0) circle [radius=.25]
(4,0) circle [radius=.25]
(4,1) circle [radius=.25]
(6,0) circle [radius=.25]
(8,0) circle [radius=.25];

\draw[fill=black]
(10,0) circle [radius=.25];

\draw[thick]
(.25,0) -- (1.75,0)
(2.25,0) -- (3.75,0)
(4.25,0) -- (5.75,0)
(6.25,0) -- (7.75,0)
(8.25,0) -- (9.75,0)
(4,.25) -- (4,.75);
\end{tikzpicture}
&
$\OO_q(\mathbf{F})$
&
quantum Freudenthal variety
\\

\hline\hline
\end{tabular}%
}
\end{table}

For a diagram of rank~$n$, to the black node $\alpha_x$ we associate the set $S := \{\alpha_1,\ldots,\alpha_n\}\backslash\{\alpha_x\}$, with corresponding Levi subgroup $L_S$. Note that any automorphism of a Dynkin diagram results in an isomorphic quantum flag manifold, which is not denoted in the diagram. In particular, for the case of $D_n$ and $E_6$, colouring the second spinor node, and the first node, respectively, produces an isomorphic copy of the corresponding quantum flag manifold.

\subsection{Tensor Product Decompositions}

We present the fusion rules for the  $U_q(\mathfrak{l}^s_S)$-module tensor product $T^{(0,1)} \otimes T^{(0,1)}$. It follows from the Weyl character formula \cite[\textsection 7.1.4]{KSLeabh} that these rules are the same as in the classical setting  \cite[Table 5]{Onishchik.Vinberg}. For the special case of $U_q(\mathfrak{l}^s_S) = U_q(\mathfrak{sl}_2)$, we adopt the convention $\varpi_2 := 0$, and use $\vee$ to denote the dual module. Note that for $D_4$, triality additionally identifies the quadric and spinor cases.

\begin{table}[h]
\centering
\renewcommand{\arraystretch}{1.3}
\small

\begin{tabular}{|c|c|c|c|c|c|}
\hline\hline 
$\mathbf{\text{Type}}$ & $\mathbf{\mathcal{O}_q(G/L_S)}$ & $\mathbf{U_q(\mathfrak{l}^s_S)}$ & $\mathbf{T^{(0,1)}}$ & $\mathbf{S^2_0 \, \oplus \, S^2_{\textrm{Cart}}}$ & $\mathbf{\Lambda^2(T^{(0,1)})}$ \\
& & & & & \\ \hline\hline 
& & & & & \\ 
$\mathbf{A}_{1}$ & $\mathcal{O}_q(S^2)$ & $\mathbb{C}1$ & $\mathbb{C}$ & $\mathbb{C}$ & $0$ \\
& & & & & \\ \hline
& & & & & \\ 
$\mathbf{A}_{n \geq 2}$ & $\mathcal{O}_q(\mathbb{C}\mathbb{P}^n)$ & $U_q(\mathfrak{sl}_n)$ & $V_{\varpi_1}$ & $  V_{2\varpi_1}$ & $V_{\varpi_2}$ \\
& $S = \Pi\backslash \{\alpha_1\}$  & & & & \\
& & & & & \\ \hline
 & & & &  & \\
$\mathbf{A}_{n \geq 2}$ & $\mathcal{O}_q(\mathbb{C}\mathbb{P}^n)$ & $U_q(\mathfrak{sl}_n)$ & $V_{\varpi_{1}}^{\vee}$ & $V_{2\varpi_{1}}^{\vee}$ & $V_{\varpi_{2}}^{\vee}$ \\
& $S = \Pi\backslash \{\alpha_{n}\}$  & & & &  \\
& & & & & \\ \hline
& & & & & \\
$\mathbf{A}_{n\geq 3}$ & $\mathcal{O}_q(\mathrm{Gr}_{n+1,r})$  & $U_q\big(\mathfrak{sl}_r \oplus \mathfrak{sl}_{n+1-r}\big)$ & $V_{\varpi_1} \boxtimes V^{\vee}_{\varpi_1}$ & $\big(V_{\varpi_2} \boxtimes V^{\vee}_{\varpi_2}\big)$ & $\,\,\,\big(V_{\varpi_2} \boxtimes V^{\vee}_{2\varpi_1}\big) $ \\
& $\,2 \leq r \leq n-1$ & & & $\oplus \, \big(V_{2\varpi_1} \boxtimes V^{\vee}_{2\varpi_1}\big)\,$ & $\oplus \, \big(V_{2\varpi_1} \boxtimes V^{\vee}_{\varpi_2}\big)$ \\
& & & & & \\ \hline
 & & & &  & \\
$\mathbf{B}_{3}$ & $\mathcal{O}_q(\mathbf{Q}_7)$ & $U_q(\mathfrak{so}_5)$ & $V_{\varpi_1}$ & $V_{0} \oplus V_{2\varpi_1}$ & $V_{2\varpi_2}$ \\
& & & & & \\ \hline
& & & & & \\ 
$\mathbf{B_{n\geq 4}}$ & $\mathcal{O}_q(\textbf{Q}_{2n+1})$ & $U_q(\mathfrak{so}_{2n-1})$ & $V_{\varpi_1}$ & $V_0 \oplus V_{2\varpi_1}\,$ & $V_{\varpi_2}$ \\
& & & & & \\ \hline
& & & & & \\
$\mathbf{C_{n \geq 2}}$ & $\mathcal{O}_q(\textbf{LG}_{2n})$ & $U_q(\mathfrak{sl}_n)$ & $V_{2\varpi_1}$ & $V_{2\varpi_2} \oplus V_{4\varpi_1}\,$ & $V_{2\varpi_1+\varpi_2}$ \\
& & & & & \\ \hline
& & & & & \\
$\mathbf{D_{n\geq 4}}$ & $\mathcal{O}_q(\textbf{S}_{2n})$ & $U_q(\mathfrak{sl}_n)$ & $V_{\varpi_2}$ & $V_{\varpi_4} \oplus V_{2\varpi_2}\,$ & $V_{\varpi_3+\varpi_1}$ \\
& & & & & \\ \hline
& & & & & \\
$\mathbf{D_{n \geq 5}}$ & $\mathcal{O}_q(\textbf{Q}_{2n})$ & $U_q(\mathfrak{so}_{2(n-1)})$ & $V_{\varpi_1}$ & $V_0 \oplus V_{2\varpi_1} \,$ & $V_{\varpi_2}$ \\
& & & & & \\ \hline

\end{tabular}

\caption{Decompositions of $T^{(0,1)}\otimes T^{(0,1)}$ into symmetric and
antisymmetric parts, for all non-exceptional cominuscule quantum flag manifolds. 
}
\label{tab:low-rank-decomposition}
\end{table}

\newpage

\subsection{Spectrum of the Quasitriangular Braiding} \label{sub.appendix.DJ.Spectrum}

In this sub-appendix we present the eigenvalues of the Drinfeld braiding $\sigma_R$ for the non-exceptional cominuscule quantum flag manifolds. These were calculated using \eqref{eq:RmatEigen}. This allows us to compare $\sigma_R$ with the generalised flip map of the Heckenberger--Kolb calculi in \textsection \ref{subsection:quasi.tri.flips}.

\begin{table}[h]
\centering
\renewcommand{\arraystretch}{2}
\setlength{\tabcolsep}{24pt}
\small

\newcommand{\thickhline}{\noalign{\hrule height 1.2pt}}

\begin{tabular}{|c|c|c|c|}
\hline\hline
$\mathbf{\mathcal{O}_q(G/L_S)}$
&
$\mathbf{T^{\otimes(0,2)}}$
&
$\mathbf{\text{Spec}(\sigma_{R}^{-1})}$
&
$\mathbf{\textrm{Spec}}(\widecheck{\sigma})$
\\
\hline\thickhline

\raisebox{-\baselineskip}[0pt][0pt]
{$\mathcal{O}_q(\mathbb{C}\mathbb{P}^{n \geq 2})$}
& $\mathbf{S^2_{\textrm{Cart}}}$ 
& $q^{-1+\frac{1}{n}}$
& $q^{-2}$ \\
\cline{2-4}
& $\mathbf{\Lambda^2}$ 
& $-q^{1+\frac{1}{n}}$
& $-1$ \\
\thickhline

$\mathcal{O}_q(\mathrm{Gr}_{n+1,r})$
& $\mathbf{S^2_{\mathrm{Cart}}}$ 
& $q^{-2-\delta}$
& $q^{-2}$ \\
\cline{2-4}
$\delta := - \frac{1}{r}-\frac{1}{n+1-r}$
& $\mathbf{S^2_{0}}$ 
& $q^{2-\delta}$
& $q^{2}$ \\
\cline{2-4}
& $\mathbf{\Lambda^2}$ 
& $-q^{-\delta}$
& $-1$ \\
\thickhline

$\mathcal{O}_q(\mathbf{Q}_{2n+1})$
& $\mathbf{S^2_{\textrm{Cart}}}$  
& $q^{-2}$
& $q^{-4}$ \\
\cline{2-4}
& $\mathbf{S^2_{0}}$  
& $q^{4n-4}$
& $q^{2(2n-3)}$ \\
\cline{2-4}
& $\mathbf{\Lambda^2}$ 
& $-q^{2}$
& $-1$ \\
\thickhline

$\mathcal{O}_q(\mathbf{LG}_{2n})$
& $\mathbf{S^2_{\textrm{Cart}}}$ 
& $q^{-4+\frac{4}{n}}$
& $q^{-4}$ \\
\cline{2-4}
& $\mathbf{S^2_0}$ 
& $q^{2+\frac{4}{n}}$
& $q^{2}$ \\
\cline{2-4}
& $\mathbf{\Lambda^2}$ 
& $-q^{\frac{4}{n}}$
& $-1$ \\
\thickhline

$\mathcal{O}_q(\mathbf{Q}_{2n})$
& $\mathbf{S^2_{\textrm{Cart}}}$  
& $q^{-1}$
& $q^{-2}$ \\
\cline{2-4}
& $\mathbf{S^2_0}$ 
& $q^{2n-3}$
& $q^{2(n-2)}$ \\
\cline{2-4}
& $\mathbf{\Lambda^2}$
& $-q$
& $-1$ \\
\thickhline

$\mathcal{O}_q(\mathbf{S}_{2n})$
& $\mathbf{S^2_{\textrm{Cart}}}$ 
& $q^{-2+\frac{4}{n}}$
& $q^{-2}$ \\
\cline{2-4}
& $\mathbf{S^2_{0}}$  
& $q^{4+\frac{4}{n}}$
& $q^{4}$ \\
\cline{2-4}
& $\mathbf{\Lambda^2}$
& $-q^{\frac{4}{n}}$
& $-1$ \\
\thickhline

\end{tabular}
\caption{Decomposition of $T^{\otimes(0,2)}$ and the corresponding
eigenvalues of the inverse of the quasitriangular braiding $\sigma_{R}$. Note that for the special case of $\OO_q(S^2)$ the subalgebra $U_q(\mathfrak{l}_S^s)$ is the unital subalgebra $\mathbb{C}1$, so there is no non-trivial braiding.}
\label{tab:braiding-eigenvalues}
\end{table}


\subsection{The Quantum Quadric Tangent Spaces} \label{Appendix.quadrics}

In this final sub-appendix we present some technicalities about the tangent spaces of the odd and even quantum quadrics, so as to be able to precisely state the results of \textsection \ref{subsection:third.eigenvalue}. 

We begin by recalling the standard choice of simple roots for the $B_n$ root system:
\begin{align*}
\Pi := \Big\{ \alpha_i := \e_i - \e_{i+1}, \, \alpha_n := \e_n \, | \, i = 1, \dots, n - 1 \Big\} \subseteq \mathbb{R}^n,
\end{align*}
where $\{\e_i\}_{i=1}^n$ is the canonical basis of $\mathbb{R}^n$. Recalling the elementary identity
$$
\e_1 = \alpha_1 + \cdots + \alpha_n,
$$
we have the following presentation 
$$
\overline{\Delta^+_S} := \Big\{ \alpha_1 + \cdots + \alpha_i\Big\}_{i = 1}^n \cup \Big\{ \e_1 + \alpha_i + \cdots + \alpha_n\Big\}_{i = 2}^n,
$$
where $S = \Pi\backslash \{\alpha_1\}$, that is, the subset of simple roots defining the odd quantum quadric $\OO_q(\mathbf{Q}_{2n+1})$. We see that the first subset of roots consists of those elements of $\overline{\Delta^+_S}$ that are less than or equal to $\e_1$, whereas the second subset of roots consists of those elements that are strictly greater than $\e_1$. Coming now to the quantised enveloping algebra $U_q(\mathfrak{so}_{2n+1})$, for a root $\alpha \in \overline{\Delta^+_S}$, we denote
\[
E_{\alpha} :=
\begin{cases}
E_i \cdots E_1, & \text{if } \alpha = \alpha_1 + \cdots \alpha_i, \textrm{ ~~~~~~~~\, for } i = 1, \dots, n \\
E^2_nE_{n-1} \cdots E_1, & \text{if } \alpha = \e_1 + \alpha_n, \\
E_i \cdots E_{n-1}E^2_nE_{n-1} \cdots E_1, & \text{if } \alpha = \e_1 + \alpha_i + \cdots + \alpha_n,  \textrm{ for } i = 2, \dots, n-1 \\
\end{cases}
\]
Note that the first case is for those roots $\alpha \leq \e_1$, and the second and third cases are for those roots $\alpha > \e_1$. 

Moving now to the $D$-series, we first recall the standard choice of simple roots for the $D_n$ root system:
\begin{align*}
\Pi := \Big\{ \alpha_i := \e_i - \e_{i+1}, \, \alpha_n := \e_{n-1} + \e_n \, | \, i = 1, \dots, n - 1 \Big\} \subseteq \mathbb{R}^n.
\end{align*}
Recalling the elementary identity
\begin{align*}
\e_1 + \e_{n-1} = \alpha_1 + \cdots + \alpha_{n},
\end{align*}
we get the presentation 
$$
\overline{\Delta^+_S} := \Big\{ \alpha_1 + \cdots + \alpha_i\Big\}_{i=1}^n \cup \Big\{\alpha_1 + \cdots + \alpha_{n-2} + \alpha_n \Big\} \cup \Big\{ \e_1 + \e_{n-1} + \alpha_i + \cdots + \alpha_{n-2}\Big\}_{i=2}^{n-2},
$$
where $S = \Pi\backslash \{\alpha_1\}$, the subset of simple roots defining the even quantum quadric $\OO_q(\mathbf{Q}_{2n})$. Coming now to the quantised enveloping algebra $U_q(\mathfrak{so}_{2n})$, for a root $\alpha \in \overline{\Delta^+_S}$, we denote
\[
E_{\alpha} :=
\begin{cases}
E_i \cdots E_1, & \text{if } \alpha = \alpha_1 + \cdots + \alpha_i, \textrm{ ~~~~~~~~~~~~~~~~~~ for } i = 1, \dots, n \\
E_{n}E_{n-2} \cdots E_1, & \text{if } \alpha = \alpha_1 + \cdots + \alpha_{n-2} + \alpha_n, \\
E_i \cdots E_{n-2}E_n \cdots E_1, & \text{if } \alpha = \e_1 + \e_{n-1} + \alpha_i + \cdots + \alpha_{n-2},  \textrm{ for } i = 2, \dots, n-2.\\
\end{cases}
\]

For both the odd and even quantum quadrics, their anti-holomorphic tangent spaces are given by 
\begin{align*}
T^{(0,1)} = \mathrm{span}_{\mathbb{C}}\Big\{ \rho_{B}(E_{\alpha}) \,|\, \alpha \in \overline{\Delta^+_S}  \Big\},
\end{align*}
where as usual $\rho_B$ denotes the restriction of domains from the quantised enveloping algebras $U_q(\mathfrak{so}_{2n+1})$ and $U_q(\mathfrak{so}_{2n})$ to the finitary dual of the corresponding quantum quadric.


\end{document}